\documentclass[letterpaper, 11pt,  reqno]{amsart}

\usepackage[margin=1.1in,marginparwidth=1.5cm, marginparsep=0.5cm]{geometry}

\usepackage{booktabs} 
\usepackage{microtype}
\usepackage{amssymb}
\usepackage{mathrsfs} 
\usepackage{amsxtra, enumerate}
\usepackage{esint}
\usepackage{color}
\usepackage[implicit=true]{hyperref}

\usepackage{cases}

\allowdisplaybreaks[2]

\definecolor{gr}{rgb}   {0.,   0.69,   0.23 }
\definecolor{bl}{rgb}   {0.,   0.5,   1. }
\definecolor{mg}{rgb}   {0.85,  0.,    0.85}
\definecolor{yl}{rgb}   {0.8,  0.7,   0.}
\definecolor{or}{rgb}  {0.7,0.2,0.2}

\newtheorem{theorem}{Theorem} [section]

\newtheorem{lemma}[theorem]{Lemma}
\newtheorem{proposition}[theorem]{Proposition}
\newtheorem{remark}[theorem]{Remark}

\DeclareMathOperator{\med}{med}

\newcommand{\noi}{\noindent}
\newcommand{\Z}{\mathbb{Z}}
\newcommand{\R}{\mathbb{R}}
\newcommand{\C}{\mathbb{C}}
\newcommand{\T}{\mathbb{T}}

\let\P= \undefined
\newcommand{\P}{\mathbf{P}}

\newcommand{\Q}{\mathbf{Q}}

\newcommand{\E}{\mathbb{E}}

\newcommand{\F}{\mathcal{F}}

\newcommand{\al}{\alpha}
\newcommand{\be}{\beta}
\newcommand{\dl}{\delta}

\newcommand{\nb}{\nabla}

\newcommand{\Dl}{\Delta}
\newcommand{\eps}{\varepsilon}
\newcommand{\kk}{\kappa}
\newcommand{\g}{\gamma}
\newcommand{\G}{\Gamma}
\newcommand{\ld}{\lambda}

\newcommand{\s}{\sigma}

\newcommand{\ft}{\widehat}

\newcommand{\wt}{\widetilde}
\newcommand{\cj}{\overline}

\newcommand{\dt}{\partial_t}
\newcommand{\dd}{\partial}

\newcommand{\embeds}{\hookrightarrow}

\newcommand{\ta}{\theta}

\renewcommand{\l}{\ell}
\renewcommand{\o}{\omega}
\renewcommand{\O}{\Omega}

\newcommand{\les}{\lesssim}
\newcommand{\ges}{\gtrsim}

\newcommand{\jb}[1]
{\langle #1 \rangle}

\newcommand{\ind}{\mathbf 1}

\newcommand{\M}{\mathcal{M}}

\newcommand{\N}{\mathbb{N}}

\renewcommand{\H}{\mathcal{H}}

\newtheorem*{ackno}{Acknowledgements}

\newcommand{\I}{\mathcal{I}}

\numberwithin{equation}{section}
\numberwithin{theorem}{section}

\newcommand{\hf}{\mathfrak{h}}

\newcommand{\Pp}{\mathcal{P}}
\newcommand{\D}{\mathbb{D}}

\makeatletter
\@namedef{subjclassname@2020}{%
  \textup{2020} Mathematics Subject Classification}
\makeatother

\begin{document}
\baselineskip = 14pt

\title[Invariant Gibbs dynamics for  NLS on the disc]
{Invariant Gibbs dynamics for the nonlinear Schr\"odinger equations on the disc}

\author[J.~Forlano and  Y.~Wang]
{Justin Forlano  and Yuzhao Wang}

\address{
Justin Forlano\\
 School of Mathematics\\
 Monash University\\
 VIC 3800\\
 Australia\\
 and\\
 School of Mathematics\\
The University of Edinburgh\\
and The Maxwell Institute for the Mathematical Sciences\\
James Clerk Maxwell Building\\
The King's Buildings\\
Peter Guthrie Tait Road\\
Edinburgh\\ 
EH9 3FD\\
United Kingdom}
\email{justin.forlano@monash.edu}

\address{
Yuzhao Wang\\
School of Mathematics\\
Watson Building\\
University of Birmingham\\
Edgbaston\\
Birmingham\\
B15 2TT\\ United Kingdom}

\email{y.wang.14@bham.ac.uk}



\dedicatory{Dedicated to Professor  Carlos E. Kenig on the
occasion of his seventieth birthday}

\begin{abstract}
We consider the two-dimensional defocusing nonlinear Schr\"odinger equation (NLS)  on the unit disc in $\R^2$ with the Gibbs initial data under radial symmetry.
By using a type of random averaging operator ansatz, 
we build a strong local-in-time  
solution theory, 
and thus 
 prove almost sure global well-posedness
and invariance of the Gibbs measure
via Bourgain's invariant measure argument. 
This work completes the program initiated by Tzvetkov
(2006, 2008) on the construction of invariant Gibbs dynamics 
(of strong solutions) for NLS on the disc.
\end{abstract}


\maketitle

\tableofcontents

\section{Introduction}
\label{SEC:1}

In this paper, we study the radial defocusing nonlinear Schr\"odinger equation (NLS) on the unit disc $\mathbb{D} \subset \R^2$:
\begin{align}
\begin{cases}
i\dt u + \Delta u - |u|^{2k} u = 0, \\
u(0) = u_{0},\\
u|_{\R \times \partial \mathbb{D}} = 0,
\end{cases}
(t,x) \in \R \times \mathbb{D}
\label{NLS}
\end{align}

\noi
where $k\in \N$,
\[
\mathbb{D} = \{(x_1,x_2) \in \R^2: x_1^2+x_2^2 \le 1\}
\]

\noi
is the unit disc with boundary $\partial \D= \{(x_1,x_2) \in \R^2: x_1^2+x_2^2 = 1\}$, $u_0$ is radially symmetric, and thus the solution
$u(t,x): \R\times \mathbb{D} \to \C$ is also radially symmetric.
Our work here aligns with the research initiated in \cite{Tz1} on the study of the long time behaviour of solutions to \eqref{NLS}.
Namely, for each $k\in \N$, we construct global strong solutions to NLS \eqref{NLS} almost surely with respect to the Gibbs measure.

The construction of Gibbs dynamics for nonlinear dispersive equations has drawn a lot of attention in recent years; see for instance, \cite{Fri85, LRS, Zhidkov91, BO94, Zhi94, McKean, BO96, BO97, Tz1, Tz2, QV, OhGibbs1, OhGibbs2, OhWhite2, NORS, Deng0, Deng, Richards, ST1, KMV, Bring0, ST2, DNY4, FOSW, CK, DNY3, GKOT, BDNY, DNY2, FO, FKV, LW24, Robert, Bring2, BS, OWnls}. A principle motivation for this study is to answer the question posed by Zakharov in 1983, detailed in \cite{Fri85}, concerning the rigorous explanation of numerical experiments which seemed to show that some nonlinear dispersive equations possessed a `returning' property: that is, solutions appear to become very close to the initial data after a long time. The classical Poincar\'{e} recurrence theorem provides a justification for this behaviour, modulo null sets with respect to some probability measure on the phase space. 

Pursuing this approach, one needs to (i) identify a candidate invariant probability measure for the given nonlinear dispersive PDE, and (ii) construct a well-defined flow $\Phi(t)$, at least almost surely with respect to the measure from part (i) and prove the invariance of the measure under this flow. 
Natural candidates for the invariant measure are informed by the Hamiltonian structure of the equation, with the canonical ensemble in this setting being the Gibbs measure; see Section~\ref{SEC:Gibbs} for the Gibbs measure for \eqref{NLS}. In the second part (ii), the uniqueness of the solutions is crucial in order to verify the group property $\Phi(t+s)=\Phi(t)\circ \Phi(s)$ for all $t,s\in \R$ and thus verify the assumptions for the Poincar\'{e} recurrence theorem.

However, completing this program  presents considerable analytic and probabilistic difficulty; see for example, \cite{BO94, BO96, DNY2, BDNY}.
In spite of these difficulties, the invariance of the Gibbs measure for nonlinear Schr\"{o}dinger equations was proved in \cite{LRS, BO94,BO96, DNY2, FOSW, OWnls} on the one and two dimensional tori. 
 In \cite{Tz1,Tz2}, Tzvetkov initiated the study aiming to extend the invariance of Gibbs measures for \eqref{NLS} posed on the (flat) tori to other compact manifolds. 
 In particular, he constructed a well-defined dynamics for \eqref{NLS} for sub-quintic nonlinearities $(k<2)$ that preserves the corresponding Gibbs measure. 
Later, Bourgain-Bulut \cite{BB2} considered all polynomial nonlinearities (\textit{any} $k\in \N$) and proved the existence of solutions on the support of the Gibbs measure. However, it is not clear whether the solutions obtained in \cite{BB2} satisfy the group property and thus whether the recurrence property can be verified.

The main purpose of this paper is to close this gap by constructing {\it strong} solutions almost surely with respect to the Gibbs measure for every $k\in \N$.
In particular, our solutions satisfy the group property and the Gibbs measure is invariant under the resulting flow. See our main results Theorem~\ref{THM:main} and Theorem~\ref{THM:main1}. Our resolution ansatz for this purpose is heavily inspired by the breakthrough work \cite{DNY2} on $\T^2$.

\subsection{NLS on the disc} \label{SEC:Gibbs}
In this section, we discuss the deterministic well-posedness theory for the radial NLS on $\mathbb{D}$. To this end, we first need to recall some
facts regarding the radial eigenfunctions of the Laplacian on $\mathbb{D}$.  We refer the reader to \cite{Tz1} and the references therein for more details. Given a radially symmetric function $f:\mathbb{D}\to \R$, there exists $\wt{f}:[0,1]\to \R$ such that
 \begin{align*}
f(x)=f((x_1,x_2)) = \wt{f}(r)
\end{align*}
where $x_1 =r\cos \ta$ and $x_2= r\sin \ta$, $0\leq \ta<2\pi$ and $0\leq r\leq 1$. Throughout this paper, we will always identify such radial functions $f$ with their corresponding function $\wt{f}$.
 
 We let $J_0 :[0,\infty) \to \R$ be the Bessel function of the first kind of order zero 
 which can be represented via the integral
\begin{align*}
J_{0}(x) = \frac{1}{\pi}\int_{0}^{\pi} e^{ix \cos \ta} d\ta. 
\end{align*}
In particular,  $J_0$ is a bounded function on $\R$.
Let $\{\ld_{n}\}_{n\in \N}$ be the simple zeroes of $J_0(x)$, which satisfy $\ld_1>0$, are strictly increasing, and have the asymptotic
\begin{align}
    \label{ev}
    \ld_n = \pi \big( n-\tfrac 14\big) + O(\tfrac 1n).
\end{align}
Thus, $\ld_n \sim n$ as $n \to \infty$. 
Then, the functions $\{J_0 (\ld_n r)\}_{n \ge 1}$
form a Dirichlet basis for $L_{\textup{rad}}^2 (\mathbb{D})$, the space of radial $L^2 (\mathbb{D})$ functions on $\mathbb{D}$ with vanishing Dirichlet boundary condition. Moreover, the functions $\{J_0 (\ld_n r)\}_{n \ge 1}$ are eigenfunctions of $- \Delta$ with eigenvalues $\ld_n^2$.
Let $e_n : \mathbb{D} \to \R$ be the $n$-th $L^2$-normalized eigenfunction defined by
\begin{align*}
e_n = e_n (r) = \frac{J_0(\ld_n r)}{\|J_0(\ld_n \cdot)\|_{L^2(\mathbb{D})} }.
\end{align*}
 Recall from \cite[Lemma 2.1]{Tz1} that
\begin{align}
\| J_{0}(\ld_n \cdot )\|_{L^{2}(\mathbb{D})} \sim n^{-\frac{1}{2}}. \label{J0L2}
\end{align}
Unlike the flat torus case $\T^2$, the eigenfunctions $\{e_n\}_{n\in \N}$ are \textit{not} uniformly bounded in $n$: in fact, they have the growth rate
\begin{align}
\|e_{n}\|_{L^{\infty}(\mathbb{D})} \sim n^{\frac 12} \to +\infty \quad \text{as} \quad n\to \infty. \label{enLinfty}
\end{align}
We refer to Lemma~\ref{LEM:efesti} for the behaviour of the $L^p(\mathbb{D})$-norms of $\{e_n\}_{n\in \N}$ in $n\in \N$,  $2\leq p\leq \infty$.

The family $\{e_n\}_{n\ge 1}$ forms an orthonormal basis of $L^2_{\textup{rad}} (\mathbb{D})$, so that, given any $f \in L^2_{\textup{rad}} (\mathbb{D})$, we have the representation:
\[
f(x) = \sum_{n=1}^\infty \ft f(n) e_n (x) = \sum_{n=1}^\infty f_n e_n (x),
\]

\noi
where the Fourier coefficients are computed through
\[
\ft f (n ) = f_n = \frac1\pi \int_{\mathbb{D}} f(x) e_n (x) dx. 
\]
We then define the Sobolev spaces $W^{s,p}_{\text{rad}}(\mathbb{D})$ as the closure of smooth radial functions under the norm 
\begin{align*}
\|f\|_{W^{s,p}(\mathbb{D})} = \bigg\| \sum_{n=1}^{\infty} \ld_{n}^{s} \,\ft f(n)e_{n}(x) \bigg\|_{L^p(\mathbb{D})}
\end{align*}
When $p=2$, we denote $H_{\text{rad}}^{s}(\mathbb{D})= W_{\text{rad}}^{s,2}(\mathbb{D})$ and note that the Plancherel theorem implies 
\begin{align*}
\| f\|_{H_{\text{rad}}^{s}(\mathbb{D})} = \|  (\sqrt{-\Dl})^{s}f \|_{L^{2}_{\text{rad}}(\D)}= \bigg( \sum_{n\in \N} \ld_{n}^{2s} |\ft f(n)|^2 \bigg)^{\frac 12}.
\end{align*}
In the remainder of this paper, we will drop the $`\text{rad}$' subscript on the functional spaces.

We now move onto discussing the low regularity well-posedness for NLS \eqref{NLS}. Smooth solutions to \eqref{NLS} conserve both the Hamiltonian (energy)
\begin{align}
H_{k}(u) = \frac12 \int_{\D}   |\nabla  u |^2  + \frac{1}{2k+2}    |u|^{2k+2} dx. \label{H}
\end{align}
and the mass 
\begin{align}
M(u)  = \frac{1}{2} \int_{\mathbb{D}} |u|^2 dx,\label{M}
\end{align}
which motivates studying the well-posedness of \eqref{NLS} in the $H^{s}(\D)$ scale of Sobolev spaces.
On $\R^2$, \eqref{NLS} has a scaling symmetry which induces the scaling critical regularity 
\begin{align}
s_{\text{crit}}(k)  =1-\tfrac{1}{k}. \label{scrit}
\end{align}
On $\T^2$ or $\D$, the scaling is no longer a symmetry\footnote{It keeps the equation \eqref{NLS} invariant but dilates the spatial domain $\T^2$ or $\D$.} but the heuristic provided by $s_{\text{crit}}(k)$ persists. For a heuristic explanation for the relevance \eqref{scrit} on $\T^2$ \cite[Section 1.3.1]{DNY2} and 
Section~\ref{SEC:scaling} for $\D$. Local well-posedness of \eqref{NLS} in $H^{s}(\D)$ for $s>1$ follows easily by applying the algebra property of $H^{s}(\D)$ and a contraction mapping argument on the Duhamel formulation of \eqref{NLS}:
\begin{align*}
u(t)  =S(t) u_0  +i\int_{0}^{t} S(t-t') |u|^{2k}u(t')dt',
\end{align*}
\noi
where $S(t) = e^{it \Delta}$ denotes the linear propagator, defined as the Fourier multiplier operator 
\begin{align}
S(t) f = \sum_{n\in \N} e^{-it\ld_n^2} \ft{f}(n) e_n. \label{St}
\end{align}
To reach the energy space $H^{1}(\D)$ and below, one needs to exploit the dispersive properties of solutions to \eqref{NLS} through Strichartz estimates and the Fourier restriction norm method \cite{BO93}.
The key tools for this are the $L^4_{t,x}$-Strichartz estimates:
\begin{align}
\| \eta(t) S(t)f\|_{L_{t,x}^{4}(\R\times \T^2)} \les \|f\|_{H^{\eps}(\T^2)} \quad \text{and} \quad \| \eta(t) S(t)f\|_{L_{t,x}^{4}(\R\times \D)} \les \|f\|_{H^{\eps}(\D)} \label{L4strich}
\end{align}
for any $\eps>0$ and where $\eta\in C_{c}^{\infty}(\R)$ and $\eta \equiv 1$ on $[-1,1]$. Bourgain \cite{BO93} proved \eqref{L4strich} on $\T^2$ and the estimate on $\D$ was essentially proved in \cite[Proposition 4.1]{Tz1}. Using \eqref{L4strich}, 
one obtains local-well posedness for \eqref{NLS} for any $s>s_{\text{crit}}$, and global well-posedness in $H^{1}(\D)$ thanks to the conservation of the Hamiltonian in \eqref{H}.  Global well-posedness below $s=1$ has been extensively studied in the $\T^2$ case; see \cite{HK1, HK2} for recent breakthroughs in the cubic case and the references therein, and \cite{YY} for the quintic case at the scaling critical regularity. There appears to be comparatively less known for \eqref{NLS} on $\D$, with global well-posedness below the energy only studied in the cubic case \cite{SY}. 
We also mention the series of highly influential works on studying nonlinear Schr\"{o}dinger equations on general manifolds without boundary \cite{BGT1, BGT2, BGT3}.  Moreover, we point out that if we remove the radial assumption in the cubic case of \eqref{NLS}, then the corresponding flow map fails to be uniformly continuous at low regularity;  see \cite[Theorem 2 and Remark 1.2.5]{BGT0}.

There is also the focusing variant of \eqref{NLS} where $-|u|^{2k}u$ is changed to $+|u|^{2k}u$. In this case, the local well-posedness is the same as that for the defocusing equation but the global theory is different. Indeed, the Hamiltonian \eqref{H} is no longer always sign definite and blow-up can occur if the mass \eqref{M} is too large; see \cite{BGT0,HK2} for further discussion.

\subsection{The Gibbs measure}

We will study \eqref{NLS} from a probabilistic point of view and construct almost sure global-in-time dynamics with respect to the Gibbs measure.  Let $k\in \N$.
The Gibbs measure, denoted by $\rho_{k}$, is formally written as
\begin{align}
    \label{Gibbs}
d\rho_{k} = Z^{-1} e^{- H_{k}(u)}  d u ,
\end{align}

\noi
where $H_k (u)$ is the Hamiltonian of \eqref{NLS} defined in \eqref{H}, and $du$ is the non-existent Lebesgue measure in infinite dimensions. Due to the Hamiltonian structure of \eqref{NLS} and the conservation of $H_k$, we expect $\rho_{k}$ to be invariant under the flow of \eqref{NLS}. 


Rigorously, we define the Gibbs measure $\rho_{k} $
as an absolutely continuous measure with respect to the (massive) Gaussian free field:
\begin{align*}
d \mu = Z^{-1} e^{- \frac12 \int_{\D}    | \nabla u  |^2    {dx}}  d u. 
\end{align*}

\noi
Here, we view $\mu$ as the induced probability measure under the
map 
\begin{align}
    \label{data}
\o \longmapsto u_0^\o (x) = \sum_{n =1}^{\infty} \frac{g_n (\o)}{ \pi \ld_{n}}  e_n (x),
\end{align}
\noi
where $\{g_n\}_{n \in \N}$ is a sequence of independent standard
complex-valued Gaussian random
variables on a probability space $(\O, \F, \mathbb{P})$. 
Using the $L^p(\D)$-estimates for the eigenfunctions $\{e_n\}_{n\in \N}$ in Lemma~\ref{LEM:efesti}, it can be shown that $u_{0}^{\o}\in W^{s,p}(\mathbb{D}) \setminus W^{s(p), p}(\mathbb{D})$ almost surely, where 
\begin{align}
s<s(p) : = \begin{cases} 
 \frac{1}{2} & \quad \text{if} \quad 1\leq p\leq 4,\\
 \frac{2}{p} & \quad \text{if} \quad 4<p< \infty.
\end{cases} \label{sp}
\end{align}
Namely,  $u_0^\o$ lies in $H^s(\mathbb{D})$ for $s < \frac12$ 
but not in $H^{\frac12} (\mathbb{D})$ almost surely. 
In particular,
we have that $\| u \|_{L^{2k+2} (\mathbb{D})} < \infty$ almost surely with respect of the Gaussian free field $\mu$. 
Hence, the Gibbs measure $\rho_{k}$ is rigorously well-defined as the following weighted Gaussian measure:
\begin{align}
d \rho_{k} =Z^{-1}  e^{-\frac1{2k+2} \int_{\D} |u|^{2k+2} dx } d\mu , \label{rhok}
\end{align}
which is a well-defined probability measure on $H^s (\mathbb{D})$, for $s < \frac{1}2$. 

We emphasise the dependence of $s(p)$ in \eqref{sp} on $p$. This means that the regularity of the samples of the Gaussian free field in \eqref{data} depends on the integrability exponent. This is a phenomenon of the geometric setting under consideration. Namely, the spatial variable $x$ plays a role due to the non-uniformity of the $L^p$-eigenfunction estimates \eqref{enLinfty}. In the torus case $\T^2$, the almost sure regularity of the Gaussian free field does not depend on the integrability exponent and it holds that $u_0^{\o}\in W^{\frac 12-\eps,\infty}(\T^2)$ almost surely for any $\eps>0$.

The roughness of the initial data on the support of $\rho_{k}$ is the primary source of the difficulty for constructing even local-in-time solutions for \eqref{NLS} with respect to $\rho_{k}$. Recall from \eqref{scrit}, that in the sub-quintic case, $s_{\text{crit}}(k)>\frac 12$, we have deterministic local well-posedness on the support of $\rho_{k}$, and indeed this is the regime covered in~\cite{Tz1, Tz2}. Strictly speaking, Tzvetkov was able to consider more general nonlinearities $F(u)$ which, with some regularity assumptions, satisfy a sub-quintic growth bound. However, for quintic and higher power-type nonlinearities $(k\geq 2)$, we have $s_{\text{crit}}(k)>\frac 12$, so that the initial data \eqref{data} is of a scaling super-critical regularity. Thus we need to go beyond the deterministic well-posedness theory to construct strong solutions with respect to $\rho_{k}$.  

In \cite{DNY2} and expanded on in \cite{DNYscale}, Deng-Nahmod-Yue introduced the notion of a probabilistic scaling regularity $s_{\text{p}}(k)$. This provides a heuristic for the probabilistic well-posedness of nonlinear dispersive equations with random Gaussian initial data mimicking the heuristic that $s_{\text{crit}}(k)$ provides. Namely, it is unlikely that we could construct strong solutions with even random initial data of regularity below $s_{p}(k)$.
In our setting of $\D$, we consider more general data of the form
\begin{align*}
u_0^\o (x) = \sum_{n =1}^{\infty} \frac{g_n (\o)}{ \pi \ld_{n}^{s+1/2}}  e_n (x),
\end{align*}
so that $\|u_0^{\o}\|_{H^{s-}}\les 1$ almost surely. When $s=\frac 12$, we get the Gaussian free field data \eqref{data} which is relevant for the Gibbs measure \eqref{rhok}.  In Section~\ref{SEC:scaling} we run the probabilistic scaling idea as in \cite{DNY2}, and find that 
\begin{align} \label{prob_scaling}
s_{\text{p}}(k)=\frac 12 - \frac{3}{4k}.
\end{align}
Notice that $s_{\text{p}}(k) < s_{\text{crit}}(k)$ for every $k\in \N$ and $s_{\text{p}}(k) <\frac 12 $. 
This provides us with some hope for the tractability of the Gibbs measure problem for \eqref{NLS} for any $k\in \N$.

The construction of the Gibbs measures \eqref{Gibbs} in the focusing case of \eqref{NLS} is more involved and can only be made in the cubic case. We discuss these issues further in Remark~\ref{RMK:focus} and how they fit with respect to our main results, which we discuss in the next section.

We also point out that Burq-Thomann-Tzvetkov \cite{BTT} considered the invariance of the Gibbs measure problem for the two-dimensional cubic NLS \eqref{NLS} on arbitrary spatial domains. Using a compactness argument, they constructed global solutions on the support of the Gibbs measure, although the uniqueness, in general, remains open.

\subsection{Main results}

In order to formulate our main results, we first need to introduce a suitable finite dimensional approximating flow for \eqref{NLS}. 
For a dyadic integer $N \in 2^{\N}$, we define $E_{N} =\{ n \, : \ld_{n}\leq N\}$ and 
 the corresponding frequency truncation operator $\P_{\le N}$ by
\[
\P_{\le N} f = \sum_{n \in E_N}  \ft f(n) e_n (x).
\]
We also define $\P_N$ to be the projection 
\[
\P_N f = \sum_{n \in E_{N}\setminus E_{N/2}}  \ft f(n) e_n (x),
\]
involving only frequencies around $N$. 
Note that $\P_{N}=\P_{\leq N} -\P_{\leq N/2}$ and since $\ld_1\sim 2.4048$, $\P_{1}u= \P_{2}u=0$.
\noi
We then consider the following truncated version of \eqref{NLS}:
\begin{align}
    \label{NLS_N}
    \begin{cases}
i\dt u_N + \Delta u_N - \P_{\le N} (|u_N|^{2k} u_N) = 0, \\
u_N(0) = \P_{\le N} u_{0}, \quad u_N\vert_{\R\times \partial \mathbb{D}}=0
\end{cases}
\end{align}
Note that $u_{1}=u_{2}\equiv 0$.
For each $N\ge 1$ and $u_0 \in L^2_{\textup{rad}} (\mathbb{D})$,
the truncated equation \eqref{NLS_N} forms a finite dimensional Hamiltonian system of ODEs with locally Lipschitz nonlinearities.
Moreover, both the mass \eqref{M} and energy $H_k (u_{N})$ are conserved under \eqref{NLS_N} and thus there exists a unique global solution $u_N$: $\R \times \mathbb{D} \to \C$ to \eqref{NLS_N} satisfying the associated Duhamel formula
\begin{align}
    \label{Duhamel}
    u_N (t) = e^{it \Delta} \P_{\le N} u_0 + i  \int_0^t e^{i(t-t') \Delta} \P_{\le N}  (|u_N|^{2k} u_N) (t') dt'.
\end{align}
It can then be shown \cite{BO94} that a frequency truncated version of \eqref{Gibbs} 
is invariant under the evolution of the frequency truncated equation \eqref{NLS_N}.

In what follows, 
we consider the Cauchy problem NLS \eqref{NLS} with Gaussian random data $u_0^\o$ given in \eqref{data}. 
Our first goal is to construct local-in-time solutions to \eqref{NLS} almost surely with respect to the random initial data \eqref{data}. In particular, we show that the sequence of finite dimensional solutions $\{u_{N}\}_{N}$ with initial data \eqref{data} converge to a limit $u$ locally in time, where $u$ solves \eqref{NLS}. 
This is addressed in
the following theorem.

\begin{theorem}\label{THM:main}
Let $k\in \N$ and $u_0$ be as in \eqref{data}. Let $\{u_N\}_{N}$ be the sequence of smooth solutions to the truncated equation \eqref{NLS_N}. Then, there exists $\ta>0$ such that for any $0<T\ll1$, there exists an event $\O_{T}\subset \O$ satisfying $\mathbb{P}(\O_{T}^{c})\leq C_{\ta}e^{-T^{-\ta}}$, such that for each $\o\in \O_T$, the sequence $\{u_N\}_{N}$ converges to a limit $u$ in $C([-T,T];H^{\frac{1}{2}-}(\mathbb{D}))$. 
The limit $u$ solves \eqref{NLS} and satisfies the flow property 
$u(t, u(s)) = u(t+s, u(0))$,
where we denote by $u(t,u_0)$ the solution to \eqref{NLS} with initial data $u_0$, evaluated at time $t$.
\end{theorem}

%
%
%
If an event $S$ on the probability space $(\O, \F, \mathbb{P})$ happens with probability $\mathbb P (S) \ge 1-C_\theta e^{-A^{\theta}}$ for some quantity $\theta > 0$ and $C_\theta$ is independent of $A$,
we say this event is $A$-certain. 
Then, Theorem \ref{THM:main} says that \eqref{NLS} with initial data \eqref{data} is $T^{-1}$-certainly locally well-posed on $[- T, T]$,
which means that the NLS \eqref{NLS} is almost surely locally well-posed with the random initial data \eqref{data}.

\begin{remark}\rm 
\label{flow_property}
We remark that in Theorem \ref{THM:main}, we establish the flow property of the limiting dynamics, which was left open by Bourgain--Bulut \cite{BB2}.  
Let $\wt u_N$ denote the solution to \eqref{NLS} with truncated initial data 
$\wt u_N(0) = \P_{\leq N} u_0$.  
By a similar argument as in Theorem \ref{THM:main} and a globalization argument,  
we can show that $\wt u_N$ converges to $u$ globally in time, which is another point that was out of reach for Bourgain--Bulut \cite{BB2}.
\end{remark}

To construct almost sure global-in-time solutions,
we use Bourgain's invariant measure argument \cite{BO94, BO96, DNY2}.
More precisely,
we  use invariance of the Gibbs measure under the finite-dimensional flow \eqref{NLS_N} to obtain a uniform in $N$ control on these solutions.
Then, we apply a PDE approximation argument to extend the local-in-time solutions to \eqref{NLS} obtained from Theorem \ref{THM:main} to global-in-time ones.
As a byproduct,
we also obtain the invariance of the Gibbs measure under the resulting global flow of the NLS \eqref{NLS}.
The structure of our solutions, as discussed in the next section and in detail in Section~\ref{SEC:Ansatz}, is robust enough that this argument passes through and the limiting global flow has the group property. 

\begin{theorem}
\label{THM:main1}
For each $k\in \N$, the defocusing NLS \eqref{NLS} on $\mathbb{D}$ is globally well-posed almost surely with respect to the Gibbs measure $\rho_{k}$ in \eqref{Gibbs}.  Moreover, $\rho_k$ is invariant under these dynamics of \eqref{NLS}.
\end{theorem}

\noi
As the proof of this result is standard by now, we omit it. See
 \cite[Sections 5.3 \& 6]{DNY2} for details.
We emphasis that Theorem~\ref{THM:main} is new for all nonlinearities of quintic and higher $(k\geq 2)$, and thus completes the program initiated by Tzvetkov~\cite{Tz1, Tz2} on the invariance of Gibbs measures for defocusing NLS \eqref{NLS} on $\D$.


\begin{remark}\rm \label{RMK:focus}

Our local well-posedness result Theorem~\ref{THM:main} also holds for the focusing NLS \eqref{NLS} as the proof is insensitive to the sign of the nonlinearity. However, the global result in Theorem~\ref{THM:main1} does not extend due to measure construction issues, as we now discuss.  
In the focusing case, the Gibbs measure is formally given by 
\begin{align*}
d \rho_{k} ``="Z^{-1}  e^{\frac1{2k+2} \int_{\D} |u|^{2k+2} dx } d\mu. 
\end{align*}
The main difficulty compared to the defocusing case is the unboundedness of the partition function $Z$ due to the bad sign in the density $e^{\frac1{2k+2} \int_{\D} |u|^{2k+2} dx }$. In this case, following \cite{LRS}, we need to introduce a mass cut-off and study the resulting measures:
\begin{align*}
d\rho_{k,R}  =Z_{k,R}^{-1} \ind_{\{M(u)\leq R\}}e^{\frac1{2k+2} \int_{\D} |u|^{2k+2} dx } d\mu,
\end{align*}
for $R>0$. Unlike the defocusing case, there is a phase transition in the construction of these measures as $k$ varies.  In the subcritical regime $k<1$, Tzvetkov \cite{Tz1} proved that $\rho_{k,R}$ is normalizable, in the sense that 
\begin{align*}
Z_{k,R}:=\E_{\mu}\big[  \ind_{\{M(u)\leq R\}}e^{\frac1{2k+2} \int_{\D} |u|^{2k+2} dx }\big] <+\infty
\end{align*}
for any finite $R>0$. In the critical case $k=1$, Oh-Sosoe-Tolomeo~\cite{OST} showed that $Z_{1,R}<\infty$ if $R<\|Q\|_{L^{2}(\R^2)}$, while if $R>\|Q\|_{L^{2}(\R^2)}$, then $Z_{1,R}=+\infty$. Here, $Q$ is the optimiser for the relevant Gagliardo-Nirenberg-Sobolev inequality, see \cite[(1.5)]{OST}. This result in \cite{OST} identified the the threshold for normalizability of $\rho_{k,R}$, where Bourgain-Bulut \cite{BB2} had previously established this for $R$ sufficiently small. It would be natural to expect that, in the super-critical case $k>1$, $Z_{k,R}=+\infty$ for any $R>0$, although we are not aware of a proof of this in the literature.
\end{remark}

In the next section, we discuss the key ingredients for the proof of Theorem \ref{THM:main}. 

\subsection{Ingredients of the proof}

Comparing to the previous works \cite{Tz1,Tz2,BB2}, the new ingredient in proving Theorem \ref{THM:main} is a finer description of the solution in the spirit of the random averaging operators of \cite{DNY2}.

One of the fundamental ideas for constructing solutions with random initial data below the scaling critical regularity is to devise an ansatz for the structure of the solution. For the cubic NLS on $\T^2$, Bourgain \cite{BO96} used the first order expansion
\begin{align}
u(t) = S(t) u_0 + v(t), \label{First}
\end{align}
which re-centers the solution $u$ around the random linear solution $S(t)u_0$. The idea here is that, by exploiting probabilistic effects, one can show that the remainder $v$ is smoother and is has a scaling sub-critical space regularity. Expansions of the form \eqref{First} have been successfully applied in a variety of situations \cite{BO97, BT1, BT2, OhWhite2, CO, NORS, NS, ST1}, for example.  Moreover, the ansatz \eqref{First} can be expanded by adding further explicit stochastic objects to the expansion \cite{BOP, OPT, OTW, CFU}, corresponding to the higher order Picard iterates, in order to eventually have the remainder terms smooth enough. However, this procedure breaks down as soon as the Picard iterates no longer display any smoothing property. For the cubic NLS \eqref{NLS} on $\mathbb{S}^2$, it was shown in \cite[Theorem 1.1]{BCST1} that the second Picard iterate does not enjoy any smoothing almost surely, no matter how regular the random initial data. This result dramatically highlighted the influence of the geometry, as there is smoothing when on the flat torus $\T^2$ \cite[Theorem 1.2]{BCST1}.
Thus, more refined ansatzes have been developed in recent years to go beyond the Picard iteration scheme and capture the essential bad interactions. See \cite{Bring0, CGI, GKO2, DNY2, DNY3, DNY4, LW24, ST2, BCST2}. 

In their breakthrough work regarding the invariance of the Gibbs measure under the flow of \eqref{NLS} on $\T^2$, Deng-Nahmod-Yue \cite{DNY2} introduced the random averaging operator ansatz (RAO), which isolates the bad high-low interactions in the nonlinearity and adds them as a ``low-frequency" operator acting on the high-frequency random initial data. Formally, the ansatz is of the form
\begin{align}
u(t)  =S(t)u_0^{\o} + \sum_{N\in 2^{\N}} \sum_{L\ll N} \mathcal{H}^{N,L}[\P_{N}S(t)u^{\o}_0](t)  +z(t) \label{RAO}
\end{align}
where the RAO operator $\mathcal{H}^{N}$ at frequency $N$ involves the truncated solutions $u_{L}$ at frequencies $L\ll N$, while it acts on the random linear solution at high frequency. 
The RAO term preserves the probabilistic independence structure between the high and low frequencies and removes the bad interactions from the remainder term. Consequently, it is then possible to show that the remainder enjoys a smoothing effect. 
The RAO ansatz was further developed by Deng-Nahmod-Yue into their broad theory of random tensors \cite{DNY3}.

We return now to our geometric setting of NLS \eqref{NLS} on the disc. To motivate our ansatz, we consider the first Picard iterate
\begin{align}
u^{(1)}(t)  = \int_{0}^{t} e^{i(t-t')\Dl} ( | e^{it' \Dl} u^{\o}_0|^{2k} e^{it'\Dl}u^{\o}_0 ) dt'. \label{Picard1}
\end{align}
Taking a Fourier transform in space, we get
\begin{align}
e^{it\ld_{n}^2}\F\{ u^{(1)}\}(t,n) = \sum_{n_1,\ldots, n_{2k+1}} \prod_{j=1}^{2k+1} \ft{u^{\o}_0}(n_j)^{\iota_j} c(n,n_1,n_2,\ldots, n_{2k+1}) \int_{0}^{t} e^{-it' \Phi(\cj{n})}dt', \label{Picard2}
\end{align}
where we defined the correlation of eigenfunctions
\begin{align}
c(\cj{n})=c(n,n_1,\ldots, n_{2k+1}) &= \int_{\mathbb{D}}e_{n}(x) \prod_{j=1}^{2k+1}e_{n_j}(x) dx,\label{c_n} 
\end{align}
and the phase function 
\begin{align}
\Phi(\cj{n})=\Phi(n,n_1,\ldots, n_{2k+1}) = \ld_{n}^2 - \ld_{n_1}^{2} + \ld_{n_2}^{2} - \ldots -\ld_{n_{2k+1}}^2,\label{phase}
\end{align} 
and $\iota_j$ denotes conjugation if $j$ is odd (and plays no role otherwise). By symmetry, we may assume that $n_1 = \max_{j=0,\ldots,k} n_{2j+1}$ and $n_2 =  \max_{j=1,\ldots,k}n_{2j}$.\footnote{Recall that on $\mathbb{D}$, the frequencies are non-negative.} We expect that the worst interactions occur whenever $\Phi(\cj{n})$ is small (resonant case) since there is little gain from the time integral in \eqref{Picard2}. Due to the signs between the terms $(\ld_n, \ld_{n_1}, \ld_{n_2})$ in \eqref{phase}, the smallness of $\Phi(\bar n)$ roughly corresponds to the case when $n_1 \gg n_2$, which is the high-low interaction. 
Then, we hope to minimise the loss in summation through counting estimates over the restricted region $\{ |\Phi(\cj{n})| \leq 1\}$. 
 
 On $\T^d$, the counting estimates are complemented with the additional hyperplane restriction 
\begin{align}
n=n_1-n_2+\ldots + n_{2k+1} \label{hyper}
\end{align}
as the correlation of eigenfunctions is trivial:
\begin{align*}
\int_{\mathbb{T}^{d}} \cj{e_{n}(x)} e_{n_1}(x) \cj{e_{n_2}(x)} \cdots e_{n_{2k+1}}(x) dx = \dl_0 (n- n_1+n_2-\ldots - n_{2k+1}).
\end{align*}
Indeed, \eqref{hyper} leads to the Fourier convolution theorem. At an analytical level, \eqref{hyper} forces a relationship between the frequencies. 
The Fourier convolution theorem fails on $\D$ since $c(\cj{n})$ in \eqref{c_n} is not supported on \eqref{hyper}. However, $c(\cj{n})$ does enjoy a kind of strong decay away from the hyperplane when $|n-n_1| \gg n_2$. 
For such off-diagonal decay on a general compact manifold without boundary, we refer the reader to \cite[Section 4]{Hani}. 
For a more explicit decay rate in the case of $\mathbb{D}$, see Lemma \ref{LEM:ev5} in the following.

 Motivated by the RAO ansatz \eqref{RAO} form \cite{DNY2}, we devise the ansatz
\begin{align}
u(t) = S(t) u_0^{\o} + \sum_{N\in 2^{\N}} \sum_{L\ll N} \mathcal{R}^{N,L}[ \P_{N}S(t)u_0^{\o}](t) +\wt{z}(t) \label{ansatzD}
\end{align}
where 
\begin{align*}
\mathcal{R}^{N,L}[ \P_{N}S(t)u_0^{\o}] = \wt{\psi}_{N,L}-\wt{\psi}_{N,L/2}
\end{align*}
with $\wt{\psi}_{N,L}$ the solution to the equation
\begin{align}
\wt{\psi}_{N,L} = \P_N S(t)u_0 - (k+1) i  \int_0^t  \sum_{n\in E_{N}}e_n  \jb{ \mathcal{N}( \jb{\wt{\psi}_{N,L}, e_n}e_n, u_{L}, \ldots, u_{L})(t')     ,e_n} dt',  \label{psiintro}
\end{align}
and 
\begin{align*}
\mathcal{N}(u_1,u_2,\ldots, u_{2k+1}) : = u_1 \cj{u_2} u_3\cdots u_{2k+1}.
\end{align*}
Notice that \eqref{psiintro} corresponds to isolating the interaction $n=n_1$ occurring when the first input function is the (frequency localized) random linear solution. 
It is remarkable that, in our setting, this precise interaction is the sole impediment to smoothing! 

Moreover, the removal of a single frequency leads to a number of simplifications compared to \cite{DNY2, LW24}, which we can leverage.
 First, this decomposes the frequencies on the right-hand side of \eqref{psiintro} allowing us to solve \eqref{psiintro} \textit{explicitly}. We find that
\begin{align}
\wt{\psi}_{N,L}(t)  =\sum_{n \in E_{N} \setminus E_{N/2}} e^{-it\ld_{n}^{2}} e^{i \Theta^{N,L}_{n}(t;\o)} \jb{u_0^{\o}, e_n} e_n
\end{align}
for some \textit{real-valued} function $\Theta^{N,L}_{n}(t;\o)$, which for $n\in E_{N}\setminus E_{N/2}$, are probabilistically independent from $\jb{u_0^{\o},e_n}$. See Lemma~\ref{LEM:unitary}. In particular, $\wt{\psi}_{N,L}(t)$ has the same $H^{\frac 12-}$-regularity as that of the initial data \eqref{data} almost surely.
Second, it is then clear that 
\begin{align}
\text{Law}( \wt{\psi}_{N,L}(t)) = \text{Law}( \P_{N}u_0^{\o})
\label{Lawinv}
\end{align}
for any $t\in \R$.  
As $\mathcal{R}^{N,L}$ corresponds to removing a \textit{single} high frequency, we refer to it as a \textit{random resonant operator} (RRO). For the detailed construction of our ansatz, see Section~\ref{SEC:Ansatz}.
We remark that the authors in \cite{BCST2} also used an ansatz of a similar flavour in studying the cubic NLS on $\mathbb{S}^2$ with rough Gaussian random initial data where the law equivalence for their random averaging operator akin to \eqref{Lawinv} played an important role.

The convergence of the summations in \eqref{ansatzD} is developed through an induction on frequencies argument as in \cite{DNY2}. However, we observe some key new features. 
Leveraging that the RRO essentially only differs from the random initial data by independent phase rotations, we can prove probabilistic Strichartz estimates for the terms $\wt{\psi}_{N,L}$ (Lemma~\ref{LEM:TypeC}) with good frequency decay. The upshot is that this allows us to handle the estimates for the RRO \textit{entirely} in physical space!  
We were motivated by \cite{BCST2} to use physical space estimates as much as possible in the geometric setting and we attest that whilst it is possible to control the RRO in Fourier space, the lack of the Fourier convolution theorem presents serious difficulties as we cannot incur any loss in the highest frequency $n$. 
The remainder $\wt{z}$ in \eqref{ansatzD} includes all of the remaining interactions which we show belongs to $H^{\s}(\D)$ for some $s_{\text{crit}}(k)<\s<1$. We handle most of the cases through physical space estimates and gains from the phase function \eqref{phase}. This allows us to reduce to the most difficult situations which do require Fourier side computations to leverage multilinear probabilistic effects.

The rest of the paper is organized as follows.
In Section \ref{SEC:Pre}, we recall some basic definitions and tools; including estimates on the eigenfunctions of the Dirichlet Laplacian on $\mathbb{D}$ and a counting estimate. We then set up the solution ansatz in Section~\ref{SEC:Ansatz} and prove Theorem~\ref{THM:main} assuming the validity of Proposition~\ref{mainprop}. 
The remaining two sections are devoted to the proof of Proposition~\ref{mainprop}. In Section \ref{SEC:RRO} we control the random resonant operator (RRO), which relies on the physical space estimates.
The physical space estimates for the input functions appearing in the RRO ansatz are also studied in that section.
Finally, Section \ref{SEC:LWP} deals with the smoother remainder term.

\section{Preliminary}
\label{SEC:Pre}

In this section,
we recall some standard tools that will be used throughout the paper.

\subsection{Notation}
We use $ A \les B $ to denote an estimate of the form $ A \leq CB $. 
Similarly,  $ A \sim B $ denotes $ A \les B $ and $ B \les A $, and $ A \ll B $ denotes $A \leq c B$ for small $c > 0$.

We define constants $b,b',b_{-}$ according to the following hierachy:
\begin{align}
0  < \tfrac{1}{2}-b'\ll b-\tfrac 12 \ll b_{-}-\tfrac{1}{2} \label{constants}
\end{align}
We also define constants $\g$ and $\kk$ such that
\begin{align}
0<10\kk< b-b'  \ll \g \ll \tfrac12 -b'<\tfrac{1}{10k+1}. \label{gamma}
\end{align}
Given $x\in \R$, we let $[x]$ denote the greatest integer less than or equal to $x$.
When we work with space-time function spaces, we may use short-hand notations such as
 $C_T H^s_x  = C([0, T]; H^s(\mathbb{D}))$. We use $t, t'$ to denote time variables and $\tau, \tau', \tau_i$ to denote Fourier variables of time.
We will abuse the notation $\ft u$, which may denote the Fourier transform in space-time, such as $\ft u (n,\tau)$, or the Fourier transform only in space of $u$, such as $\ft u(n,t)$, depending on the context. 
 Given a finite set of dyadic numbers $\{N_{j}\}$, we use $N_{(j)}$ to denote the $j$-th largest number among them.
We also use $n_{(j)}$ to denote the $j$-th largest number among a finite list of natural numbers $\{n_j\}$.
Given a finite set $A$, we will use $|A|$ to denote the cardinality of $A$. We use $\ind_{A}$ to represent the indicator function of the set $A$.


For a complex number $z$, we define $z^{+}=z$ and $z^{-}=\cj{z}$. For a finite index set $A$, we use the notation $z^{\iota_j}$ for $\{\iota_j\}_{j\in A}$ with $\iota_j\in \{\pm\}$. With a slight abuse of notation, we will identify $\iota\in \{\pm\}$ with the corresponding real number $\{\pm 1\}$.
We say that the pair $(k_i, k_j) \in \N^2$ for $i,j\in A$, is a \textit{pairing} in the variables $k_{A}$, if $k_{i}=k_{j}$ and $\iota_i +\iota_j =0$. We say that a pairing is an \textit{over-pairing} if $k_{i}=k_{j}=k_{\l}$ for some $\l\in A\setminus \{i,j\}$ and call a pairing \textit{simple} if it is not an over-pairing. 





\subsection{Norms}
\label{SUB:norms}

%
%
%
%
%

For $s, b \in \R$, we define a norm on space-time functions adapted to $L^{2}_{\text{rad}}(\mathbb{D})$ by
\begin{equation}
\label{norms:Xsb}
\begin{aligned}
\|v\|_{Y^{s,b}}^2 &:=\int_\R\jb{\tau  }^{2b}\| \ld_{n}^{s} \widehat{v}(n,\tau) \|_{\l_n^2}^2\,d\tau  = \sum_{n\in \N} \ld_{n}^{2s}  \int_\R \jb{\tau  }^{2b} |\widehat{v} (n, \tau) |^2\,d\tau.
\end{aligned}
\end{equation} 
The norm in \eqref{norms:Xsb} is nothing but the Fourier restriction norm of the interaction representation of a function \cite{BO93}. More precisely, if $v=S(-t)u$, then 
\begin{align*}
\| v\|_{Y^{s,b}}  = \|S(-t)u\|_{Y^{s,b}} = \|\ld_{n}^{s} \jb{\tau+\ld_{n}^{2}}^{b} \ft u(\tau,n)\|_{\l^{2}_{n}L^{2}_{\tau}}=: \| u\|_{X^{s,b}} . 
\end{align*}

\noi
For any interval $I$, define the corresponding localized norms
\begin{equation*}
\|v\|_{Y^{s,b}(I)}:=\inf\big\{\|w\|_{Y^{s,b}}: w=u\mathrm{\ on\ }I\big\}.
\end{equation*} 

\noi
By abusing notations, 
we will call the above $w$ an \emph{extension} of $u$ (from $I$ to $\R$).
We have the following important embedding: for any $s\in \R$ and $b>\frac 12$, we have 
\begin{align}
Y^{s,b}([-T,T]) \embeds C([-T,T]; H^{s}(\mathbb{D})).
\end{align}

Let $\chi\in C_{c}^{\infty}(\R)$ be a smooth non-negative cutoff function supported on $[-2,2]$ with $\chi\equiv 1$ on $[-1,1]$, and we set $\chi_{T}(t) :=\chi(T^{-1}t)$ for $T>0$. For later use, we also define $\wt{\chi}(\cdot) := \chi(\cdot) - \chi(2\cdot)$, which is a smooth bump function supported on $\frac 12 \leq |t|\leq 1$.
We then recall the following well-known result; see for instance \cite[Lemma 4.2]{DNY2}.

\begin{lemma}[Short time bounds]
\label{PROP:short} 
For any $0<T\leq 1$ and any $u \in Y^{s,b_1}([-T,T])$, we have
\begin{equation}
\label{localtime}
\| v\|_{Y^{s,b}([-T,T])}+\|\chi_{T} \cdot v\|_{Y^{s,b}}
\lesssim T^{b_1-b}\|v\|_{Y^{s,b_1}([-T,T])}
\end{equation} 

\noi
provided either $0<b \leq b_1<1/2$, or $u(0)=0$ 
and $1/2<b \leq b_1<1$. 
\end{lemma}

We define the truncated Duhamel operators
\begin{equation}
\label{duha}
\mathcal I_{\chi} f(t)=\chi(t)\int_0^t\chi(t') f(t')\,dt',
\end{equation} 
and let $\mathcal K(\tau, \tau')$ be the corresponding Fourier kernel defined by
\begin{equation}
\label{duhamelform}
\widehat{\mathcal I_{\chi} f}(\tau) =\int_{\R} \mathcal K(\tau,\tau') \widehat{f} (\tau')\,d\tau'.
\end{equation} 

\noi
Then, we have the following bounds on the kernels:

\begin{lemma}[\cite{DNY2,DNY3}]
\label{PROP:duhamel} 
For any $\al>0$ sufficiently large, it holds 
\begin{equation}
\label{duhamleker}
|\mathcal K (\tau,\tau')| + |\partial_{\tau}\mathcal K(\tau,\tau')|+|\partial_{\tau'}\mathcal K (\tau,\tau')|  
\les \frac{1}{\jb{\tau'}}\bigg( \frac{1}{\jb{\tau-\tau'}^{\al}}+\frac{1}{\jb{\tau}^{\al}}\bigg)\les  \frac{1}{\jb{\tau}\jb{\tau-\tau'}}
\end{equation}

\end{lemma}

As a consequence of the first inequality in \eqref{duhamleker}, we have the following estimate.

\begin{lemma}[Duhamel estimate]
\label{LEM:duhamel}  
Let $\frac 12<b<1$,  $\mathcal I_{\chi} $ be the linear operator defined in \eqref{duha} ,and $v \in Y^{s, b-1
}$.
Then, we have
\begin{equation*}
\label{localtime1}
\| \mathcal I_\chi v \|_{Y^{s,b}}
\lesssim  \|v\|_{Y^{s,b-1}},
\end{equation*} 
\end{lemma}

\subsection{Eigenfunction estimates}
In this section, we state and prove some facts about the eigenfunctions $\{e_n\}_{n\in \N}$ of the radial Laplacian on the unit disc with Dirichlet boundary conditions.

\begin{lemma}\cite[Lemma 2.1]{Tz1}
\label{LEM:efesti}
Given $2\leq p\leq \infty$, there exists $C, \wt{C}>0$ such that for every $n \ge 1$,
\begin{align}
    \label{efesti1}
    \begin{split}
    \|e_n\|_{L^p (\D)} & \le C \sigma_{p}(n) \| e_n\|_{L^2 (\D)} = C \sigma_{p}(n)\\
    \|e_n'\|_{L^p (\D)} & \le C \sigma_{p}(n) \| e_n'\|_{L^2 (\D)} \leq \wt{C}n \s_{p}(n) \|e_{n}\|_{L^{2}(\D)}, 
    \end{split}
    \end{align}
where $e_{n}'(r) = \tfrac{d}{dr} e_{n}(r)$ and
\begin{align*}
    \sigma_{p}(n) = 
    \begin{cases}
    1 & \textup{ when } \,\,\, 2\le p < 4,\\
    (\log (1+n))^{\frac14} & \textup{ when } \,\,\, p = 4,\\
    n^{-\frac2p + \frac12} & \textup{ when } \,\,\, p > 4.
    \end{cases}
\end{align*}
\end{lemma}


The following basic estimate controls the correlation of eigenfunctions $c(\cj{n})$ defined in \eqref{c_n} irrespective of any relation between $n$ and $ \max (n_1,n_2, \cdots, n_{2 k+1})$. 

\begin{lemma}
\label{LEM:ev1}
Given any $\eps>0$, there exists $C_{\eps}>0$ such that
\begin{align*}
|c(\cj{n})| \leq C_{\eps} \begin{cases}
  n_{(3)}^{\eps}  \quad & \text{if} \quad k=1, \\ 
 n_{(3)}^{\eps} \prod_{j= 4}^{2k+1} n_{(j)}^{\frac12-\frac{\eps }{2(2k-3)}}\quad & \text{if} \quad k\geq 2.
 \end{cases}
\end{align*}
\end{lemma}

\begin{proof} We consider the case when $k\geq 2$ as the case when $k=1$ follows from similar and simpler arguments.
Let $\eps>0$ and fix $0<\dl<\frac{1}{10}$ so that $\frac{\dl}{2(1+\dl)}\leq \eps$.
Then, by H\"{o}lder's inequality and \eqref{efesti1}, we have
\begin{align}
|c(\cj{n})| &\les \|e_{n} \|_{L^{4(1-\dl^2)}} \| e_{n_{(1)}} \|_{L^{4(1-\dl^2)}} \|e_{n_{(2)}} \|_{L^{4(1-\dl^2)}} \| e_{n_{(3)}}\|_{L^{4(1+\dl)}} \prod_{j=4}^{2k+1}\|   e_{n_{(j)}}\|_{L^{\frac{2(2k-3)}{\eps} \frac{1-\dl^2}{1-3\dl-4\dl^2} }}  \notag\\
&\les n_{(3)}^{\eps}  \prod_{j=4}^{2k+1} \| e_{n_{(j)}}\|_{L^{\frac{4(2k-3)}{\eps}} }  \les n_{(3)}^{\eps}  \prod_{j=4}^{2k+1} n_{(j)}^{\frac 12 - \frac{\eps}{2(2k-3)}}. \notag
\end{align}
This completes the proof.
\end{proof}

\begin{lemma}
For any $N\in \N$ and $f\in L^2_{\text{rad}}(\D)$, it holds 
\begin{align}
\| \nb \P_{\leq N}f\|_{L^{2}(\D)} \les N\|\P_{N}f\|_{L^2(\D)}. \label{L2derivs}
\end{align}
\end{lemma}
\begin{proof}
Whilst \eqref{L2derivs} is intuitively obvious, we choose to present a proof for the convenience of the reader. We have 
\begin{align*}
\nb \P_{\leq N} f = \sum_{n\in E_{N}} \ft f(n) \nb e_n.
\end{align*}
Thus, using integration by parts \eqref{ev} and Plancherel's theorem, we have
\begin{align*}
\| \nb \P_{\leq N}f\|_{L^{2}(\D)}^2 & = \sum_{n_1,n_2 \in E_{N}}\ft f(n_1) \cj{\ft f(n_2)} \int_{\D} \nb e_n \cdot \nb e_n dx \\
& = \sum_{n_1,n_2 \in E_{N}}\ft f(n_1) \cj{\ft f(n_2)} \int_{\D} e_{n_1} (-\Dl e_{n_2}) dx \\
& = \sum_{n_1,n_2 \in E_{N}} \ld_{n_2}^{2} \ft f(n_1) \cj{\ft f(n_2)} \int_{\D} e_{n_1} e_{n_2} dx \\
& = \sum_{n\in E_{N}} \ld_{n}^{2} |\ft f(n)|^2 \\
&\les  N^{2} \|\P_{\leq N}f\|_{L^{2}(\D)}^2.
\end{align*}
Now \eqref{L2derivs} follows by taking square roots of both sides.
\end{proof}

We have the following strong off-diagonal decay for the high-low interaction.

\begin{lemma}
\label{LEM:ev5}
If $|n - n_{(1)}| \ges n_{(2)} \gg 1$, 
then, for any $\eps>0$, it holds that
\begin{align}
|c(n,n_1,\cdots,n_{2 k+1})| 
\les \frac{n_{(2)} n_{(3)}^{\eps}}{|n-n_{(1)}|} \prod_{j=4}^{2k+1} n_{(j)}^{\frac12}.
\label{cdecay}
\end{align}
\end{lemma}

\begin{proof}
Without loss of generality,
we assume $n_1 \ge n_2 \ge \cdots n_{2 k+1}$.
We note from integration by parts (Green's theorem), that we have
\begin{align}
(&\ld_n^2 - \ld_{n_1}^2) c(n,n_1,\cdots,n_{2 k+1}) \notag \\
& = \int_{\D} \Delta e_n(x) e_{n_1} (x) \cdots  e_{n_{2 k+1}} (x) dx - \int_{\D} e_n(x)  \Delta e_{n_1} (x) \cdots e_{n_{2 k+1}} (x) dx \notag \\
&   =2\int_{\D} e_{n} \nb e_{n_1} \cdot \nb \bigg(  \prod_{j=2}^{2k+1}e_{n_j}      \bigg) dx 
+\int_{\D} e_{n} e_{n_1} \Dl \bigg(  \prod_{j=2}^{2k+1}e_{n_j}      \bigg) dx   \label{ev12}.
\end{align}
We estimate each of the terms in \eqref{ev12}. By distributing the derivative, we have 
\begin{align*}
\int_{\D} e_{n} \nb e_{n_1} \cdot \nb \bigg(  \prod_{j=2}^{2k+1}e_{n_j}      \bigg) dx 
= \sum_{\l=2}^{2k+1} \int_{\D} e_{n} (\nb e_{n_1} \cdot  \nb e_{n_{\l}}) \prod_{ \substack{j=2 \\ j\neq \l}}^{2k+1} e_{n_j}dx
\end{align*}
We only consider the worst term when $\l=2$. 
By H\"{o}lder's inequality (with the same exponents as in the proof of Lemma~\ref{LEM:ev1}) and \eqref{efesti1}, we have 
\begin{align*}
\bigg|\int_{\D} e_{n} (\nb e_{n_1} \cdot  \nb e_{n_{2}}) \prod_{j=3}^{2k+1} e_{n_j}dx\bigg| 
&\les \|e_n\|_{L^{4-}} \| \nb e_{n_1}\|_{L^{4-}} \|\nb e_{n_2}\|_{L^{4-}} \|e_{n_3}\|_{L^{4+}} \prod_{j=4}^{2k+1} \|e_{n_j}\|_{L^{\infty-}} \\
& \les n_1 n_2 n_{3}^{\eps} \prod_{j=4}^{2k+1}n_{j}^{\frac 12},
\end{align*}
for any $\eps>0$.
Now from \eqref{ev},
\begin{align*}
|\ld_n^2 -\ld_{n_1}^2|=  |\ld_n-\ld_{n_1}| (\ld_n+\ld_{n_1})  \ges |n-n_1| ( n+n_1),
\end{align*}
we see that 
\begin{align}
\frac{1}{ |\ld_n^2 -\ld_{n_1}^2|}\bigg|\int_{\D} e_{n} \nb e_{n_1} \cdot \nb \bigg(  \prod_{j=2}^{2k+1}e_{n_j}      \bigg) dx\bigg| \les \frac{n_2}{|n-n_1|}n_{3}^{\eps} \prod_{j=4}^{2k+1}n_{j}^{\frac 12}. \label{ev10}
\end{align}
This completes the proof for this contribution in \eqref{ev12}.

We move onto the second contribution in \eqref{ev12}. Distributing the derivatives, we see that the worst case is 
\begin{align*}
\int_{\D} e_{n} e_{n_1} \Dl e_{n_2}  \prod_{j=3}^{2k+1}e_{n_j} dx & =\ld_{n_{2}}^{2}\, c(n,n_1,\ldots, n_{2k+1}).
\end{align*}
Then, from Lemma~\ref{LEM:ev1}, we have 
\begin{align*}
\frac{1}{ |\ld_n^2 -\ld_{n_1}^2|}\bigg| \int_{\D} e_{n} e_{n_1} \Dl e_{n_2}  \prod_{j=3}^{2k+1}e_{n_j} dx\bigg|
&\les  \frac{n_2}{n+n_1} \frac{n_2}{|n-n_1|} n_{3}^{\eps} \prod_{j=4}^{2k+1}n_{j}^{\frac 12}.
\end{align*}
This is a better bound than in \eqref{ev10}, which completes the proof.
\end{proof}

\subsection{Counting estimates} 

In this section,
we prove some counting estimates,
which will be used in establishing the crucial trilinear estimates in Section \ref{SEC:LWP}. 

\begin{lemma}\cite[Lemma 4.1 (1)]{DNY2}
\label{LEM:counting}
Let $\mathcal{R}=\mathbb{Z}$ or $\mathbb{Z}[i]$.
Then, given $0\neq m\in\mathcal{R}$, 
and $a_0,b_0\in\mathbb{C}$, 
the number of choices for $(a,b)\in\mathcal{R}^2$ 
that satisfy
\begin{equation*}
m=ab,\,\,|a-a_0|\leq M,\,\,|b-b_0|\leq N
\end{equation*}

\noi
is $O(M^\theta N^\theta)$ 
with constant depending only on $\theta>0$.
\end{lemma}
 
\begin{lemma}
\label{LEM:c2}
Let $\ld_n^2$ be the $n$-th eigenvalue of the Dirichlet Laplacian on $\mathbb{D}$ and let $Q \subset \R^2$ be a box of side lengths $R$.
Then for any $m \in \R_+$,
\begin{align}
    \label{c2}
    |\{ (n_1,n_2) \in \Z^2: |\ld_{n_1}^2 + \ld_{n_2}^2 - m| < 1, (n_1,n_2) \in Q  \}| &\les \exp \left( c \frac{\log R}{\log \log R} \right)\\
    \label{c3}
    |\{ (n_1,n_2) \in \Z^2: |\ld_{n_1}^2 - \ld_{n_2}^2 - m| < 1, (n_1,n_2) \in Q, n_1\neq n_2  \}| &\les O(R^\theta),
\end{align}
where the constant in \eqref{c3} depends only on $\theta >0$.
\end{lemma}

\begin{proof}
For \eqref{c2}, 
see \cite[Lemma 2.2]{BB2}.
We only need to prove \eqref{c3}.
From \eqref{ev}, we see that $\ld_n^2 = \pi^2 (n-\frac14)^2 + O(1)$.
Then,
the equation $|\ld_{n_1}^2 - \ld_{n_2}^2 - m| < 1$ implies
\begin{align*}
    \left| \pi^2 \left(n_1-\tfrac14 \right)^2 -  \pi^2 \left(n_2-\tfrac14\right)^2 -m  \right|= \Big|  \left(2n_1 + 2n_2 - 1 \right) \left(n_1 - n_2\right) - \frac{2m}{ \pi^2}  \Big| &\le O(1).
\end{align*}
At this point, \eqref{c3} follows from Lemma \ref{LEM:counting}.
\end{proof}

Given $m \in \Z$,
and dyadic $1\leq N_j, N$, $j=1,\ldots,2k+1$, we define
the base tensor ${\mathrm T}^{\mathrm{b},m}$ by
\begin{align}
{\mathrm T}_{n n_1 \cdots n_{2 k+1}}^{\mathrm{b},m} 
=   \mathbf{1}_{\{ [\Phi(\cj{n}) ] = m \}}(n,n_1,\ldots,n_{2k+1}) \cdot 
 \ind_{E_N}(n) \cdot \prod_{j=1}^{2 k+1} \ind_{ E_{N_j}}(n_j) ,
\label{basetensor0}
\end{align}
where $\Phi(\cj{n})$ was defined in \eqref{phase}.
With a slight abuse of notation, we do not distinguish the tensor operator ${\mathrm T}^{\mathrm{b},m}$ from its kernel ${\mathrm T}_{n n_1 \cdots n_{2 k+1}}^{\mathrm{b},m}$.
Moreover, given a set $S\subseteq  \mathbb{N}^{2k+2}$, we define the restricted base tensor
\begin{align}
{\mathrm T}_{n n_1 \cdots n_{2 k+1}}^{\mathrm{b},m,S}= {\mathrm T}_{n n_1 \cdots n_{2 k+1}}^{\mathrm{b},m}  \ind_{S}. \label{baseS}
\end{align}
We have the following bound for the Hilbert-Schmidt norm of the corresponding tensor operator.

\begin{lemma}
\label{LEM:Tbs}
Let $k\ge 1$ and $n \in E_{N}$, $n_j \in E_{N_j}$ for $j= 1,\cdots, 2 k+1$. 
Then, for any $\eps>0$, 
\begin{align}
   \sup_{m\in \Z} \big\| {\mathrm T}_{n n_1 \cdots n_{2 k+1}}^{\mathrm{b},m} \ind_{n \neq n_{\textup{odd}}}  \big\|_{n n_1\cdots n_{2 k+1}} & \les  N^\eps (N_{(1)})^\eps \prod_{j = 2}^{2 k+1} (N_{(j)})^{\frac12} \label{Tbs1}  
    \end{align}
\noi
where 
\begin{align*}
n_{\textup{odd}} = \Big \{n_i: n_i = \max_{j=0,1,\cdots, k} n_{2j+1}\Big \} ,
\end{align*}

\noi
i.e. $n_{\textup{odd}}$ is the largest frequency of the odd factors in the nonlinearity.
\end{lemma}

\begin{proof} 
The inequality \eqref{Tbs1} can be reduced to show the following counting estimate
\[
\begin{split}
|\{(n,n_1,\cdots,n_{2 k+1})  &  \in \N^{2 k+1}, \,\textup{no simple pairing}: \\
& \hphantom{XXXX} \ld_{n}^2- \sum_{j=1}^{2 k+1} \iota_j \ld_{n_j}^2 = m + O(1), \ld_{n_j} \le N_j\} |  \les  N^\eps (N^{(1)})^\eps \prod_{j = 2}^{2 k+1} N^{(j)}.
\end{split}
\]

\noi
Without loss of generality,
we assume $\iota_i = +$ for odd $i$, and $\iota_i = -$ for even $i$.
Due to symmetry, we only need to consider two cases: (1) $N_1 \ge N_2 \ge \cdots \ge N_{2 k+1}$ and (2)  $N_2 \ge N_1 \ge \cdots \ge N_{2 k+1}$. 
For the case (1), 
by fixing $n_2, \cdots, n_{2 k+1}$, we are left to count 
\[
\bigg\{(n,n_1) \in \mathbb{N}^2 \textup{ and } n \neq n_1: \ld_{n}^2 - \ld_{n_1}^2 = m + \sum_{j=2}^{2 k+1} \iota_j \ld_{n_j}^2 + O(1), \ld_{n} \le N, \ld_{n_1} \le N_1 \bigg\},
\]
From Lemma \ref{LEM:c2}, 
the above is clearly bounded by $(N N_1)^{\eps}$,
and this completes the proof of \eqref{Tbs1} for this case.
For case (2), we first fix all $n_j$ for $j \neq 2$, and are left to count 
\[
\bigg\{(n,n_2) \in \Z_+^2: \ld_{n}^2 + \ld_{n_2}^2 = m + \ld_{n_1}^2 + \sum_{j=3}^{2 k+1} \iota_j \ld_{n_j}^2 + O(1), \ld_{n} \le N, \ld_{n_1} \le N_1 \bigg\},
\]

\noi
which is again bounded by $(N N_2)^{\eps}$ by using Lemma \ref{LEM:c2}.
Thus we conclude \eqref{Tbs1}.
\end{proof}


\subsection{Probabilistic lemmas}

In this section, we recall some probabilistic lemmas.
In this paper,
we use $\{g_k\}_{k =1}^\infty$ to denote a family of independent standard complex-valued Gaussian random variables on a probability space $(\O, \F, \mathbb{P})$.
 \begin{lemma}\cite[Lemma 3.4]{CO}.
 \label{LEM:growth}
 Let $\eps>0$ and $R>0$.
 Then,
 there exists $\O_{R,\eps} \subset \O$ with $\mathbb{P}(\O_{R,\eps}^c) < e^{-c R^2}$,
 such that
 \begin{align*}
     |g_n (\o)| \le C R\jb{n}^\eps 
 \end{align*}
 \noi
 uniformly in $n \in \N$ for $\o \in \O_{R, \eps}$.
 \end{lemma}

We recall the following conditional Wiener chaos estimate.

\begin{lemma}\cite[Lemma 4.1]{DNY2}.
\label{LEM:condchaos}
Let $E\subset \N$ be a finite subset, and let $\mathcal{B}$ be the $\s$-algebra generated by $\{g_{k}: k\in E\}$. Let $\mathcal{C}$ be a $\s$-algebra independent with $\mathcal{B}$ and let $\mathcal{C}^{+}$ be the smallest $\s$-algebra containing both $\mathcal{C}$ and the $\s$-algebra generated by $\{|g_{k}|^{2}: k\in E\}$.  
Assume that in the support of $a_{k_1 \ldots, k_{n}}$, there is no pairing in $\{k_1,\ldots, k_n\}$.
Consider the expression
\begin{align*}
F(\o) = \sum_{(k_1,\ldots, k_{n})\in E^{n}} a_{k_1 \ldots k_n}(\o) \prod_{j=1}^{n} g_{k_j}(\o)^{\iota_j}.
\end{align*}
where $n\leq 2r+1$, $\iota_j \in \{\pm\}$ and the coefficients $a_{k_1 \ldots k_n}(\o)$ are $\mathcal{C}^{+}$-measurable. Let $A\geq |E|$. Then, $A$-certainly, we have 
\begin{align*}
|F(\o)| \leq A^{\ta} \bigg( \sum_{k_1,\ldots, k_n} |a_{k_1,\ldots,k_n}(\o)|^2\bigg)^{\frac12}.
\end{align*}
\end{lemma}

\subsection{Scaling heuristics}  \label{SEC:scaling}

In this section, we give some heuristic computations to justify the scaling critical threshold and the probabilistic scaling for \eqref{NLS}. The computations rely on some of the notation and results discussed above which is why we include this discussion here. We apply the arguments from \cite[Section 1.2]{DNY2}. 

The heuristic is essentially based on the fact that if we expect \eqref{NLS} to be locally well-posed in $H^{s}(\mathbb{D})$ for some $s\in \R$, then whenever the initial data $u_0$ has roughly unit norm $\|u_0\|_{H^{s}}\sim 1$, then we should expect the first Picard iterate (at say $t=1$) in \eqref{Picard1}
to also have norm roughly one; $\| u^{(1)}(1)\|_{H^{s}}\sim 1$. 
Fix $s\in \R$ and $N\gg 1$ dyadic and choose $u_0$ so that $\ft{u_0}(n) = N^{-(s+\frac 12)} \ind_{E_{N}\setminus E_{N/2}}$. It follows from \eqref{ev} that $\|u_0\|_{H^s}\sim 1$.
Taking the Fourier transform, we essentially have
\begin{align*}
\F\{ u^{(1)}\}(n) &\sim \sum_{ \substack{n_1, \ldots, n_{2k+1} \\ |n_j|\les N}} \frac{1}{\jb{\Phi(\cj{n})}} \prod_{j=1}^{2k+1} \ft u_0(n_j)  c(n,n_1,\ldots,n_{2k+1}).
\end{align*}
We will assume that $n_1, n_2, \ldots, n_{2k+1} \sim N$ and, to simplify the computations, that $n \in E_N$ so that $n\les n_1$. Then, ignoring logarithmic losses, we have
\begin{align*}
\ind_{n\in E_{N}}|\F\{ u^{(1)}(1)\}(n)| \les N^{-(2k+1)(s+\frac 12)}   \sup_{|m|\les N^2} \bigg| \sum_{n_1,\ldots, n_{2k+1}} {\mathrm T}_{n n_1 \cdots n_{2 k+1}}^{\mathrm{b},m} \bigg| \| c(\cj{n})\|_{\l^{\infty}_{nn_1 \cdots n_{2k+1}}}
\end{align*}
Assuming there is no pairing, by Lemma \ref{LEM:ev1} and Lemma \ref{LEM:Tbs}, we have
\begin{align}
\sup_{|m|\les N^2}\sup_{n\in E_{N}} \bigg| \sum_{n_1,\ldots, n_{2k+1}} {\mathrm T}_{n n_1 \cdots n_{2 k+1}}^{\mathrm{b},m} \bigg| \| c(\cj{n})\|_{\l^{\infty}_{nn_1 \cdots n_{2k+1}}} \les N^{2k-1}N^{k-1} \sim N^{3k-2}. 
\label{scalecount}
\end{align}
Therefore,
\begin{align*}
\|  \P_{\leq N} u^{(1)}(1)\|_{H^{s}} \sim N^{-(2k+1)(s+\frac 12)}N^{s+\frac 12} N^{3k-2} \sim N^{-2k(s+\frac 12)+3k-2} \les 1
\end{align*}
provided that $s>1-\frac 1k$, which recovers the scaling criticality regularity $s_{\text{crit}}(k)$ given in \eqref{scrit}.

The heuristic for the probabilistic scaling asserts that due to multilinear random oscillations \'{a} la the central limit theorem, we expect a square root gain in the ``counting" bound \eqref{scalecount}; see \cite{DNY2, DNYscale}. Thus, in the probabilistic setting, we have
\begin{align*}
\|  \P_{\leq N} u^{(1)}\|_{H^{s}}  \sim N^{-2k(s+\frac 12)+2k-\frac32} \les 1
\end{align*}
provided that $s>\frac{1}{2}-\frac{3}{4k}$. This justifies the significance of the probabilistic scaling index $s_{\text{p}}(k)$ given in \eqref{prob_scaling}.

%

%


\section{Outline of the proof} \label{SEC:Ansatz}
In this section,
we exploit the idea of the random averaging operator in \cite{DNY2} to introduce
a random resonant operator,
which is designed to remove a troublesome resonance term in the nonlinearity.
We then construct a strong solution to \eqref{NLS} through an induction.  

\subsection{Reformulation}
Let $N\geq1$ be a dyadic number and let $u_N(t)$ be the solution to \eqref{Duhamel}. We consider the interaction representation $v_{N}(t) = S(-t)u_{N}(t)$, where $S(t)$ is the linear propagator defined in \eqref{St}.
Then $v_N$ satisfies the integral equation
\begin{align}
    v_N (t) = \P_{\le N} u_0 - i \int_0^t \P_{\le N} \mathcal M(v_N, \ldots, v_{N}) (t') dt',
    \label{NLS_N2}
\end{align}
where the $(2k+1)$-linear operator $\mathcal{M}$ is defined via its Fourier coefficients
\begin{align}
\jb{ \mathcal M (v_1, \ldots, v_{2k+1}) ,e_n}  = \sum_{n_1,\ldots, n_{2k+1}}  e^{-it \Phi(\bar n)}c(\cj{n}) \prod_{j=1}^{2k+1} (v_j)_{n_j}^{\iota_j} 
    \label{Mk}
\end{align}
with $v_n = \jb{v,e_n} = \int_{\D} v(x) e_n(x) dx$,
\begin{align}
\iota_j &= 
\begin{cases}
1 &\quad  \text{if}  \quad  j\in \{1,\ldots, 2k+1\}\,\, \text{is odd}\\
-1&\quad  \text{if} \quad j\in \{1,\ldots, 2k+1\}\,\, \text{is even}, 
\end{cases} \notag 
\end{align}
$\Phi(\cj{n})$ is the phase function from \eqref{phase}, and $c(\cj{n})$ is given in \eqref{c_n}. Note that $v_{2}\equiv 0 .$ 
Starting with the system \eqref{NLS_N2},
let $y_N = v_N - v_{N/2}$ which satisfies the integral equation
\begin{align}
    \begin{split}
        y_N (t) & = \P_N u_0 - i \sum_{N_{\max} = N} \int_0^t \P_{\le N} \mathcal M (y_{N_1}, \cdots, y_{N_{2k+1}}) (t') dt'\\
        & \hphantom{XX} - i \sum_{N_{\max} \le N/2} \int_0^t \P_{N} \mathcal M (y_{N_1}, \cdots, y_{N_{2k+1}}) (t') dt',
    \end{split}
    \label{NLS_N3}
\end{align}

\noi
where $N_{\max} = \max (N_1, N_2, \cdots N_{2k+1})$.
Note that 
\begin{align}
y_{1}=y_{2}\equiv 0 \quad \text{and} \quad  y_{2^{2}}=v_{2^{2}}. \label{y2case}
\end{align}
It turns out that the main difficulty comes from a resonance in the high-low term (second on the right-hand side of \eqref{NLS_N3}), which will be approximated in the following section.

\subsection{The random resonant operator}
Let $L \in 2^{\N\cup\{0\}}$ be such that $1 \le L \le N/2$ and
consider the following linear equation for a function $\Psi=\Psi(t,x)$ at spatial frequency $n$:
\begin{align}
\partial_t \Psi_n (t) = - (k+1) i \int_{\mathbb{D}} \P_{\le N} \mathcal M \big( \{\Psi \}_{n}, v_L, \cdots, v_L \big)(t',x) e_n(x) dx, 
    \label{Psi}
\end{align}
\noi
where  $\{ \Psi\}_{n} (x) = \Psi_{n} e_{n}(x)$.
 We recover $\Psi$ via $\Psi= \sum_{n=1}^{\infty} \{\Psi\}_{n}$.
If \eqref{Psi} has initial data $\Psi(0) = \Psi_0$,
then in order to solve \eqref{Psi},
it suffices to solve the integral equation
\begin{align}
\Psi(t) - \Psi_0 = \mathcal P_{N,L} [\Psi (t)],
    \label{Psi2}
\end{align}

\noi
where $\mathcal P_{N,L}$ is the linear operator defined by
\begin{align}
    \mathcal P_{N,L} [f] &= - (k+1) i \sum_{n\in E_{N}} e_n \int_0^t \int_{\D} \P_{\le N} \mathcal M (\{f\}_{n}, v_L, \cdots, v_L)(t',x)  e_{n}(x) dx dt'     \label{PNL} .
\end{align}

\noi
For $n\in E_{N}$, we define the complex functions
\begin{align}
\G^{N,L}_{n}(t) 
: =
 -(k+1)i  \sum_{n_2,\ldots, n_{2k+1}} e^{-it \Phi(\cj{n})} c(n,n,n_2,\ldots, n_{2k+1}) \prod_{j=2}^{2k+1} \ft{(v_L)}_{n_j}^{\iota_j}(t), \label{Gamma}
\end{align}
and  $\G^{N,L}_{n}(t) =0$ for  $n\in \N \setminus E_{N}$.
Then we may rewrite \eqref{Psi} as
\begin{align*}
\partial_t \Psi_n (t) = \Gamma_n^{N,L} (t) \Psi_n (t)
\end{align*}
which, together with the initial condition $\Psi_n (0)$, implies that
\begin{align}\label{Psi_2}
\Psi_n (t) = \Psi_n (0) \exp \bigg(\int_0^t \Gamma_n^{N,L} (t') dt' \bigg)
\end{align}
At a formal level, 
we may also solve \eqref{Psi2} to get
\begin{align}
    \Psi (t) = (1- \mathcal P_{N,L})^{-1} \Psi_0 
    \label{Psi3}
\end{align}
Continuing on, we denote
\begin{align*}
    \mathcal H^{N,L} = (1- \mathcal P_{N,L})^{-1}
\end{align*}
and from \eqref{Psi_2} and $\{\Psi_0\}_n = \Psi_n (0) e_n$, it follows that
\begin{align}
H^{N,L}_n (t) = \exp \bigg(\int_0^t \Gamma_n^{N,L} (t') dt' \bigg)
\end{align}
so that 
\[
 \mathcal H^{N,L} [e_n]  (x) = H^{N,L}_n (t)  \cdot e_n(x) \quad \text{for} \quad n\in E_{N}
\]
and $H^{N,L}_n (t)=1$ for $n\in \N \setminus E_{N}$.
With these notations, the solution to \eqref{Psi} can be written as
\begin{align}
\Psi (t) = \mathcal H^{N,L} [\P_N \Psi_0] = \sum_{n\in E_{N}\setminus E_{N/2}} H^{N,L}_n (t) 
 \jb{\Psi_0, e_n} e_n.
    \label{Psi4}
\end{align}

\noi
When we set $\Psi_0 = \P_N u_0$,
the corresponding solution \eqref{Psi4} to \eqref{Psi} will be denoted by $\psi_{N,L}$, i.e.
\begin{align}
    \psi_{N,L}(t) = \mathcal H^{N,L} [\P_N u_0] = \sum_{n\in E_{N}\setminus E_{N/2}}  H^{N,L}_n (t) \cdot \frac{g_n (\o)}{\pi \ld_n} \cdot e_n. \label{psiNL}
\end{align}
Note that the restriction in the summation in \eqref{psiNL} arises from the definition of $\P_{N}$ and the fact that $n=n_1$.
\noi
From \eqref{PNL}, we then obtain
\begin{align}
\partial_t \psi_{N,L} = - (k+1) i \sum_{n\in E_{N}}e_n \int_{\D} \mathcal M (\{\psi_{N,L}\}_{n}, v_L, \cdots, v_L)(t',x)  e_{n}(x) dx,
\label{psi1}
\end{align}

\noi
with initial condition $\psi_{N,L} (0) = \P_N u_0^\o$.
Therefore, we have
\begin{align}
\psi_{N,L}(t) = \P_N u_0 - (k+1) i  \int_0^t  \sum_{n\in E_{N}}e_n \int_{\D} \mathcal M (\{\psi_{N,L}\}_{n}, v_L, \cdots, v_L)(t',x)  e_{n}(x) dx dt'.
\label{psi2}
\end{align}

\noi
It is clear that \eqref{psi2} consists of resonant terms in high-low interaction at frequency $N$. 
We refer to \eqref{psi2} as the random resonant operator, since it resembles the random averaging operator introduced in \cite{DNY2}. A significant simplification here as compared to \cite{DNY2} is that the kernel for the linear operator $\mathcal{H}^{N,L}$ is a diagonal matrix, while in  \cite{DNY2} the kernel of the random averaging operator has non-zero off-diagonal components. 

From the above construction, 
we observe the following crucial unitarity property.

\begin{lemma}
\label{LEM:unitary}
    For fixed $t \in [0,T]$, the operator $\mathcal H_n^{N,L} (t)$ is unitary. In particular, we have
    \begin{align}
H_{n}^{N,L}(t) = \exp\bigg(-i(k+1) \int_{0}^{t} \int_{\mathbb{D}} e_n^2 (x)  |u_{L} (t',x)|^{2k}  dx dt \bigg), \label{HnNL}
\end{align}
  and hence   $|H_n^{N,L} (t)| = 1$ for all $t \in \R$.
\end{lemma}

\begin{proof}
Notice that when $n=n_1$, $\Phi(\cj{n}) = \sum_{j=2}^{2k+1}\iota_j \ld_j^{2}$, that is, the dependence of $\Phi(\cj{n})$ on $(n,n_1)$ is removed.
Thus, from \eqref{psiNL} and \eqref{psi1}, we see that
\begin{align}
\label{RAO3}
\begin{cases}
i \partial_t H^{N,L}_n (t) = (k+1) H^{N,L}_n (t) \jb{  |u_{L} (t)|^{2k} e_n ,\, e_n}\\
H^{N,L}_n (t)|_{t=0} = 1.
\end{cases}
\end{align}
where
\begin{align*}
\jb{  |u_{L} (t)|^{2k} e_n ,\, e_n} = \int_{\mathbb{D}} e_n^2 (x)  |u_{L} (t,x)|^{2k}  dx 
\end{align*}
is real. The unique solution to the linear ODE \eqref{RAO3} is then \eqref{HnNL}.

Alternatively, we see from \eqref{Gamma} that
\begin{align*}
\G_{n}^{N,L}(t)&= -(k+1)i \int_{\mathbb{D}} e_{n}(y)^2 |u_{L}(t,y)|^{2k}dy , 
\end{align*}
when $n\in E_{N}$,
with $\text{Re} \, \G_{n}^{N,L}(t)=0$. 
\end{proof}

Recall that $H^{N,L}_n (t)$ is $\mathcal B_{\le L}$-measurable. In particular, since $L\leq \frac{N}{2}$, $H^{N,L}_n (t)$ is independent of $\{g_n (\o)\}_{n\in E_{N}\setminus E_{N/2}}$.
Therefore, the term $\psi_{N,L}$ can be seen as the same as $\P_N u_0^\o$ conditioned on $\mathcal B_{\le L}$ since, by Lemma~\ref{LEM:unitary} and the invariance of complex Gaussian random variables under (random but independent) rotations, we have
$
{\rm Law} (\psi_{N,L} (t)) = {\rm Law} (\psi_{N,L}  (0)) = {\rm Law} (\P_N u_0^\o).
$
We will exploit this fact to show that the random resonant terms $\psi_{N,L}$ enjoy the same regularity properties (in terms of Strichartz spaces) as the initial data $\P_{N}u_0$. See Lemma~\ref{LEM:TypeC}.

In what follows, we denote
\begin{align}
    \mathfrak h^{N,L}: = \mathcal H^{N,L} - \mathcal H^{N,\frac{L}2}, \quad 
    \zeta^{N,L}: = \psi_{N,L} - \psi_{N,\frac{L}2} = \hf^{N,L}[\P_{N} u_0] .
    \label{zeta}
\end{align}
One of the key observation here is that $\mathfrak{h}^{N,L}$ and $\mathcal H^{N,L}$ are Borel functions of $(g_n (\o))_{n\in E_{N/2}}$ and thus are independent of the Gaussians in $\P_N u_0^\o$.

\subsection{The decomposition}
Now we are ready to introduce the ansatz.
We note that the ansatz here is similar to the one in \cite{DNY2}, but with the random averaging operator being replaced by the random resonant operator \eqref{psi2}. 
Given $\kk>0$ as in \eqref{gamma}, we define
\begin{align*}
L_{N} = \max\{ L\in 2^{\N} \,: \, L <N^{1-\kk}\}.
\end{align*}
Let $z_N$ be a remainder term defined via
\begin{align}
    y_N (t) = \psi_{N,L_{N}} (t) + z_N(t).
    \label{decom1}
\end{align}
From \eqref{NLS_N3},
one may check that $z_N$ solves the equation
\begin{align}
\begin{split}
z_N (t) +  \psi_{N,L_{N}} (t) & =  \P_N u_0 
- i   \sum_{N_{\max} = N_{\med} = N }   \int_0^t \P_{\le N}   \mathcal{M} ( y_{N_1} , \cdots y_{N_{2k+1}}) (t')  dt' \\
& - i  \sum_{N_{\max} = N, \,N_{\med} \le \frac{N}2}   \int_0^t \P_{\le N}   \mathcal{M} ( y_{N_1} , \cdots, y_{N_{2k+1}}) (t')  dt'  \\
 & -i  \sum_{N_{\max} \le \frac{N}2}  \int_0^t \P_N \mathcal{M} (y_{N_1}, \cdots, y_{N_{2k+1}}) (t') dt', \\
 =  - i &  \sum_{N_{\max} = N_{\med} = N }   \int_0^t \P_{\le N}   \mathcal{M} ( y_{N_1}, \cdots y_{N_{2k+1}}) (t')  dt' \\
& - (k+1)i  \sum_{N_j \leq \frac{N}2,\, j \neq 1}   \int_0^t \P_{\le N}   \mathcal{M} ( z_{N} ,  y_{N_2}, \cdots, y_{N_{2k+1}}) (t')  dt'  \\
& - (k+1)i  \sum_{N_j \leq \frac{N}2,\, j \neq 1}   \int_0^t \P_{\le N}   \mathcal{M} ( \psi_{N,L_{N} },  y_{N_2}, \cdots, y_{N_{2k+1}}) (t')  dt'  \\
& - k i  \sum_{N_j \leq  \frac{N}2, \, j \neq 2}   \int_0^t \P_{\le N}   \mathcal{M} ( y_{N_1} ,  y_{N}, y_{N_3}, \cdots, y_{N_{2k+1}}) (t')  dt'  \\
 & -i  \sum_{N_{\max} \le \frac{N}2}  \int_0^t \P_N \mathcal{M} (y_{N_1}, \cdots, y_{N_{2k+1}}) (t') dt',
\end{split}
\label{eqn:zn0}
\end{align}

\noi 
where in the second equality we used symmetry amongst the odd and even indexed inputs giving rise to the combinatorial factors $(k+1)$ and $k$.
It turns out that the third term on the right-hand side of the second equality in \eqref{eqn:zn0} is the cause of the main difficulty.
In particular, there is a resonance occurring at high-low to high interaction.
Expanding this term further, we see that
\begin{align}
&-  (k+1)i  \sum_{N_j \leq \frac{N}2,\, j \neq 1}   \int_0^t \P_{\le N}   \mathcal{M} ( \psi_{N,L_{N}} ,  y_{N_2}, \cdots, y_{N_{2k+1}}) (t')  dt' \notag \\
& = - (k+1)i \sum_{ \substack{ n \in E_N \\ n_1 \in E_N \backslash E_{N/2} }} e_{n}  \sum_{N_j \leq  \frac{N}2,\, j \neq 1} \int_0^t  \int_{\D} \mathcal{M} ( \{\psi_{N,L_{N}}\}_{n_1},  y_{N_2}, \cdots, y_{N_{2k+1}}) (t',x) e_n(x) dx  dt'\notag \\
& = - (k+1)i \sum_{ \substack{ n \in E_N \\ n_1 \in E_N \backslash E_{N/2} \\ n \neq n_1}} e_n    \sum_{N_j \leq L_{N},\, j \neq 1} \int_0^t  \int_{\D} \mathcal{M} (\{\psi_{N,L_N}\}_{n_1},  y_{N_2}, \cdots, y_{N_{2k+1}}) (t',x) e_n(x) dx  dt' \notag\\
& \hphantom{XX}  - (k+1)i \sum_{n \in E_N \backslash E_{N/2}} e_n  \sum_{ N_{\max} \leq L_{N}  } \int_0^t  \int_{\D} \mathcal{M} ( \{\psi_{N,L_N}\}_{n},  y_{N_2}, \cdots, y_{N_{2k+1}}) (t',x) e_n(x) dx  dt'  \notag\\
& \hphantom{XX}   - (k+1)i \sum_{  \substack{N_j \leq \frac{N}{2},\, j \neq 1 \\  N_{\max} >L_{N} }}  \int_0^t \P_{\le N}   \mathcal{M} ( \psi_{N,L_{N}} ,  y_{N_2}, \cdots, y_{N_{2k+1}}) (t')  dt' \notag\\
&  = R_{1} + R_2 + R_3,\label{expand}
\end{align}
where we used the multilinearity of $\mathcal M$ given in \eqref{Mk}.
\noi
Now we have isolated the most dangerous term which is the resonant part $R_2$ in \eqref{expand}.
Due to the construction of $\psi_{N,L_{N}}$ in \eqref{psi2},
we have $R_2 = \psi_{N,L_{N}} (t) - \P_N u_0,$
which results in a cancellation of $R_2$ in \eqref{eqn:zn0}. 
We point out that in this computation it is important that the other inputs $\{v_{N_j}\}_{j=2}^{2k+1}$ are supported on low frequencies $\le L_{N}$ so that there are no additional conditions forcing a relation between the frequencies $(n_2,\ldots, n_{2k+1})$ (which there would be from our symmetrization in \eqref{eqn:zn0}) and so $R_{2}+\P_N u_0$ can be identified as $\psi_{N,L_{N}}$ in \eqref{psiNL}.
By collecting \eqref{eqn:zn0}, \eqref{expand}, and \eqref{psi2},
we get
\begin{align}
    z_N (t)  
=  - i & \sum_{N_{\max} = N_{\med} = N }   \int_0^t \P_{\le N}   \mathcal{M} ( y_{N_1}, \cdots y_{N_{2k+1}}) (s)  ds \notag \\
& - (k+1)i  \sum_{N_j \le \frac{N}2,\, j \neq 1}   \int_0^t \P_{\le N}   \mathcal{M} ( z_{N} ,  y_{N_2}, \cdots, y_{N_{2k+1}}) (t')  dt' \notag \\
& - (k+1)i \sum_{ \substack{ n \in E_N, n_1 \in E_N \backslash E_{N/2} \\ n\neq n_1 }  } e_n \sum_{  N_{\max} \leq L_{N}   } \notag\\
& \hphantom{XXXXX} \int_0^t  \int_{\D} \mathcal{M} ( \{\psi_{N,L_{N}}\}_{n_1},  y_{N_2}, \cdots, y_{N_{2k+1}}) (t',x) e_n(x) dx  dt' \label{eqn:zn1} \\
&- (k+1)i \sum_{  \substack{N_j \le \frac{N}{2},\, j \neq 1 \\  N_{\max} >L_{N} }}  \int_0^t \P_{\le N}   \mathcal{M} ( \psi_{N,L_{N}} ,  y_{N_2}, \cdots, y_{N_{2k+1}}) (t')  dt' \notag\\
& - k i  \sum_{N_j \le \frac{N}2, \, j \neq 2}   \int_0^t \P_{\le N}   \mathcal{M} ( y_{N_1} ,  y_{N}, y_{N_3}, \cdots, y_{N_{2k+1}}) (t')  dt' \notag \\
 & -i  \sum_{N_{\max} \le \frac{N}2}  \int_0^t \P_N \mathcal{M} (y_{N_1}, \cdots, y_{N_{2k+1}}) (t') dt', \notag
\end{align}
The main advantage of this formulation is the appearance $n \neq n_1$ in the third term of the right-hand side of \eqref{eqn:zn1},
which manifests the non-pairing at high-low to high interaction.
This restriction comes from the removal of $\psi_{N,L_{N}}$ from $y_N$, i.e. \eqref{decom1}.
In the following, 
we shall consider a time-localized version of \eqref{NLS_N3},
\begin{align}
    \begin{split}
    z_N (t)  &
=  - i  \sum_{N_{\max} = N_{\med} = N }   \chi (t) \int_0^t \chi (t')  \P_{\le N}   \mathcal{M} ( y_{N_1}, \cdots y_{N_{2k+1}}) (t')  dt' \\
& \hphantom{XX}  - (k+1)i  \sum_{N_j \le \frac{N}2,\, j \neq 1}  \chi (t) \int_0^t \chi (t') \P_{\le N}   \mathcal{M} ( z_{N} ,  y_{N_2}, \cdots, y_{N_{2k+1}}) (t')  dt'  \\
&  \hphantom{XX} - (k+1)i \sum_{ n \in E_{N}, n_1 \in E_{N}\setminus E_{N/2}, n \neq n_1} e_n(x) \sum_{ \max_{j=2,\ldots, 2k+1}N_j \leq L_{N}} \\
& \hphantom{XXXXXX} \chi (t) \int_0^t \chi (t') \int_{\D} \mathcal{M} (  \{\psi_{N,L_{N}}\}_{n_1} ,  y_{N_2}, \cdots, y_{N_{2k+1}}) e_n(x) dx  dt'  \\
& \hphantom{XX}  - (k+1)i \sum_{  \substack{N_j \le \frac{N}{2},\, j \neq 1 \\  N_{\max} >L_{N} }}  \chi(t)\int_0^t  \chi(t')\P_{\le N}   \mathcal{M} ( \psi_{N,L_{N}} ,  y_{N_2}, \cdots, y_{N_{2k+1}}) (t')  dt'\\
&  \hphantom{XX} - k i  \sum_{N_j \le \frac{N}2, \, j \neq 2}  \chi (t) \int_0^t  \chi (t') \P_{\le N}   \mathcal{M} ( y_{N_1} ,  y_{N}, y_{N_3}, \cdots, y_{N_{2k+1}}) (t')  dt'  \\
 & \hphantom{XX} -i  \sum_{N_1, \cdots, N_{2k+1} \le \frac{N}2} \chi (t) \int_0^t \chi (t') \P_N \mathcal{M} (y_{N_1}, \cdots, y_{N_{2k+1}}) (t') dt'.
\end{split}
\label{eqn:zn2}
\end{align}
We write 
\begin{align}
z_{N}(t) = \text{RHS}\eqref{eqn:zn2}=\sum_{N_1,\ldots, N_{2k+1}}\sum_{J=1}^{6} \mathcal{A}_{N}^{(J)}( y_{N_1}, \ldots, y_{N_{2k+1}}) \label{zeqA}
\end{align}
By symmetry between the odd frequencies and the even frequencies, we assume that $N_1 =\max_{j=0,1,\ldots, k} N_{2j+1}$ and $N_2 = \max_{j=1,\ldots, k} N_{2j}$ in the supports of $ \mathcal{A}_{N}^{(1)}$  and  $\mathcal{A}_{N}^{(6)}$.
 Here in \eqref{eqn:zn2}, the random resonant term $ \psi_{N,L_{N}}$ is also time-localized, still denoted by $ \psi_{N,L_{N}}$, having the form
 \begin{align}
 \begin{split}
 \psi_{N,L_N} (t) = 
\chi(t)\P_N u_0 - (k+1) i  \I_{\chi}\bigg[\sum_{n\in E_{N}}e_n \int_{\D} \mathcal M (\{\psi_{N,L_N}\}_{n}, v_L, \cdots, v_L)(x)  e_{n}(x) dx \bigg].
 \end{split}
 \label{psi3}
 \end{align}

\subsection{The a priori estimates}

We now state the main a priori estimates, and prove that they imply Theorem \ref{THM:main}.

\begin{proposition}
\label{mainprop} 
Given $0< T \ll 1$, and let $I=[- T, T]$.  Let $b,b', b_{-}$ be as in \eqref{constants} and $\g, \kk>0$ be as in \eqref{gamma}.
For $M\in 2^{\N}$, consider the following statements, which we call $\textup{\textsf{Local}}(M)$:

\smallskip
\noi
\textup{(i)} For any $s\geq 0$, given $L_2, \cdots, L_{2k+1} < \min(M,N^{1-\kk})$, the operators $\mathcal P_{N,L_2,\ldots, L_{2k+1}}$ defined by 
\begin{equation}
\label{linmap}
\begin{split} 
\mathcal{P}_{N,L_2,\ldots, L_{2k+1}}[\psi] (t) & = - (k+1)i   \I_{\chi}\Big[ \sum_{n\in E_{N}} e_n \jb{ \mathcal{M} (\{\psi\}_{n},  y_{L_2}, \cdots, y_{L_{2k+1}}) ,e_n}  \Big]
\end{split}
\end{equation} 

\noi
satisfy
\begin{align}
\sup_{N} \sup_{L_{\max}< \min(M,N^{1-\kk})}  L_{\max}^{\ta_0} \| \Pp_{N,L_2,\ldots, L_{2k+1}}\|_{Y^{s, b}(I)\to Y^{s ,b}(I)} \les T^{b_{-}-b},
\label{PwZbb}
\end{align} 
for some $\ta_0>0$, which does not depend on $s$, and where $L_{\max}=\max(L_2, \cdots, L_{2k+1})$.

\smallskip
\noi
\textup{(ii)} For the operators $\mathfrak h^{N,L}$ in \eqref{zeta}, 
we have
\begin{equation}
\label{hNLbound}
\sup_{N}\max_{L<\min(M,N^{1-\kk})} L^{\theta_0} (\log L)^{-4k} \|\mathfrak h^{N,L} \|_{Y^{s, b}(I)\to Y^{s,b}(I)} \les   T^{b_{-}-b}, 
\end{equation} 

\noi
for $0 < \theta_0 \ll 1$.

\smallskip
\noi
\textup{(iii)} For the terms $z_N$, we have
\begin{equation}\label{mainz} 
\max_{N\leq M} N^{1-\g}\|z_N\|_{Y^{0,b}(I)}\leq 1.
\end{equation}

\noi
Given the above statements, there exists $\ta>0$ depending only on $(b,b_-, b', \g,\kk, \ta_0)$ such that
 \[\mathbb{P}(\textup{\textsf{Local}}(M/2)\wedge\neg\,\textup{\textsf{Local}}(M))\leq C_\theta e^{-(T^{-1}M)^\theta}.\]
\end{proposition}

%

\subsection{Proof of Theorem \ref{THM:main} assuming Proposition~\ref{mainprop} }

By Proposition \ref{mainprop}, the event 
\begin{align*}
\O_{T} : = \{ \o \in \O \,: \, \textsf{Local}(M)\,\, \text{is true for all} \,\, M\in 2^{\N}\}
\end{align*}
has complemental probability $\mathbb{P}(\O_{T}^{c}) \leq ce^{-T^{-\ta}}$.  As $\textsf{Local}(M)$ is true for all $M$, $\zeta^{N,L}$ from \eqref{zeta}, and $z_N$, the solution to \eqref{eqn:zn2}, are well-defined for each $N$ and thus so too is $y_N$ with the expansion
\begin{align}
y_N= \chi(t) \P_N u_0 +\sum_{L\leq L_{N}}\zeta^{N,L}  +z_N (t), \label{ydecomp}
\end{align}
where $u_0$ is given in \eqref{data}.

We show that the sequence
\begin{align*}
v_{N}(t) = \sum_{2\leq K\leq N} y_{K}(t)
\end{align*}
 converges in a suitable sense for each $\o\in \O_{T}$ as $N\to \infty$. Using \eqref{ydecomp}, we have
 \begin{align*}
 \sum_{2\leq K\leq N} y_{K}(t) & = \sum_{2\leq K\leq N} \chi(t) \P_{K} u_0 + \sum_{2\leq K\leq N}\sum_{L\leq L_{K}}\hf^{K,L}[\chi(\cdot)\P_{K}u_0] + \sum_{2\leq K\leq N} z_{K}(t).
\end{align*}
By direct computations using \eqref{data}, 
\begin{align}
u_{0}^{\o}= \lim_{N\to \infty}\sum_{2\leq K\leq N} \P_{K}u_0^{\o} \quad \text{in} \quad H^{\frac{1}{2}-\eps}(\mathbb{D}), 
\label{u0conv}
\end{align}
for any $\eps>0$.
In view of \eqref{mainz}, we see that 
\begin{align}
\sum_{N\geq 2} z_{N} \quad \text{converges in} \quad C(I;H^{1-2\g}_{x}(\mathbb{D})) 
\label{zconv}
\end{align}
almost surely.
Given small $0<\eps\ll \min(b-\frac{1}{2}, b_{-}-b)$, we have from Lemma~\ref{LEM:growth} that $T^{-1}K-$certainly,
\begin{align*}
\|\hf^{K,L}[\chi(\cdot)\P_{K}u_0]  \|_{Y^{s,b}(I)} & \les \|\hf^{K,L}\|_{Y^{s,b}(I)\to Y^{s,b}(J)}\|\P_{K}u_0\|_{H^s}  \les  T^{b_{-}-b-\eps} L^{-\ta_0} (\log L)^{4k} K^{-\frac 12 +s+\eps}.
\end{align*}
Thus, 
\begin{align*}
\bigg\| \sum_{2\leq K\leq N}\sum_{L\leq L_{K}}\hf^{K,L}[\chi(\cdot)\P_{K}u_0]  \bigg\|_{C(I; H^{s}_{x}(\mathbb{D}))} & \les \sum_{2\leq K\leq N}\sum_{L\leq L_{K}} \big\| \hf^{K,L}[\chi(\cdot)\P_{K}u_0] \big\|_{Y^{s,b}(I)} \\
& \les T^{b_{-}-b-\eps}  \sum_{2\leq K\leq N}\sum_{L\leq L_{K}} L^{-\ta_0} (\log L)^{4k} K^{-\frac 12+s+\eps},
\end{align*}
which is finite uniformly in $N$ provided that $s<\frac{1}{2}-\ta$. Moreover, by the linearity of the operators $\hf^{N,L}$ we have that 
\begin{align}
\sum_{N\geq 2} \sum_{1\leq L\leq L_{N}} \hf^{N,L}[\chi(\cdot)\P_{N}u_0] \quad \text{converges in} \quad C(I;H^{\frac{1}{2}-2\eps}_{x}(\mathbb{D})) \label{hconv}
\end{align}
almost surely.
Combining \eqref{u0conv}, \eqref{zconv}, and \eqref{hconv}, we see that the sequence $\{v_N-\P_{\le N} u_0\}$ converges in $C(I; H_x^{\frac{1}{2}-2\eps}(\mathbb{D}))$, and hence $\{v_N\}$ converges in $C(I;H^{\frac12 -2\eps}(\mathbb{D}))$, and so does the original sequence $\{u_N\}$.

\subsection{Reformulation}
\label{SUB:reform}
In the remainder of this paper, we are concerned with the proof of Proposition \ref{mainprop}.
The proof will proceed by strong induction on $M\in 2^{\N}$.
It follows immediately from \eqref{y2case} that the estimates involved in $\textsf{Local}(2)$ hold automatically.
Then we assume that $\textsf{Local}(M')$ for all $2\leq M'\leq M$ dyadic, and our goal is to show that $\textsf{Local}(2M)$ holds true with large probability; more precisely, the probability that $\textsf{Local}(2M)$ does not hold is less than $C_{\ta}e^{-(T^{-1}M)^{\ta}}$ for some $\ta,C_{\ta}>0$.

From now on we assume that $\textsf{Local}(M')$ is true for all $2\leq M'\leq M$. This means that for every $N\leq M$ and for all $L'<\min(M,N^{1-\kk})$, we have 
\begin{align*}
y_{L'}= \chi(t) \P_{L'}u_0 + \sum_{2\leq R<(L')^{1-\kk}}\hf^{L',R}[ \chi(\cdot)\P_{L'}u_0] + z_{L'}(t)
\end{align*}
on $I=[-T,T]$. 
Thus in order to prove the estimates (i), (ii), and (iv) in Proposition~\ref{mainprop} at the level $\textsf{Local}(2M)$, we may assume that the input function $w_{L'}$ is of one of the following three types:
%

\begin{itemize}
\item[(i)] \underline{\textbf{Type (G)}}, where
\begin{equation}
\label{TypeG}
\ft{w_L}(\tau, n)= \ind_{E_{L}\setminus E_{L/2}}(n) \frac{g_{n}(\omega)}{\pi \ld_{n}}\widehat{\chi_{T}}(\tau).
\end{equation}

\item[(ii)] \underline{\textbf{Type (C)}}, where
\begin{equation}
\label{TypeC}
\ft{w_L}(\tau, n)
=\ind_{E_{L}\setminus E_{L/2}}(n) \ft h_{n}^{L,R}(\tau,\omega) \frac{g_{n}(\omega)}{\pi \ld_{n}},
\end{equation} 
with $h_{n}^{L,R}(t ,\o)$ supported in the set $n \in E_{L}\setminus E_{L/2}$ and $t \in [-2T,2T]$, is $\mathcal{B}_{\leq R}$ measurable for some $R<L^{1-\kk}$, and satisfying the bounds 
\begin{equation}
\label{TypeCnorm}
\begin{split}
\sup_{n\in E_{L}\setminus E_{L/2}}\|\langle \tau\rangle^{b} \ft h_{n}^{L,R}(\tau,\o)\|_{L_{\tau}^2 } & \lesssim R^{-\dl_0},
\end{split}
\end{equation} 
\noi
for some $b>\frac 12$ and $0 < \dl_0 \ll 1$.
 Note that 
we have $h_{n}^{L,R} =\ind_{E_{L}\setminus E_{L/2}}(n)\cdot {\chi_T}(t)$
for $R=2$,
where Type (C) reduces to Type (G).

\item[(iii)] \underline{\textbf{Type (D)}}, where $\ft{w_L}(\tau, n)$ is supported in $E_{L}$, and satisfies
\begin{equation}
\label{TypeD}
\|\langle\tau \rangle^b \ft{w_L}(\tau, n)\|_{\ell_{n}^2 L_{\tau}^2}\lesssim L^{-1+\g}.
\end{equation} 
for some $b>\frac 12$, provided that $2\leq L \leq M$.

\end{itemize}

The remaining sections of this manuscript are organised as follows: we first control the random resonant operator in \eqref{linmap} by proving \eqref{PwZbb} in (i) of Proposition~\ref{mainprop}. We use these bounds to also prove (iii) in Proposition~\ref{mainprop}. This takes place in Section~\ref{SEC:RRO}. Then Section~\ref{SEC:LWP} completes (iv) in Proposition~\ref{mainprop} by proving \eqref{mainz} on the remainder terms.

\section{Physical space estimates and the random resonant operator}
\label{SEC:RRO}

\subsection{Strichartz estimates}

In this section, we observe some further physical space estimates that Type (C) and Type (D) inputs enjoy. In the following, we assume the setup of the previous section: that $\textsf{Local}(M')$ holds for all $1\leq M'\leq M$.

We begin with estimates for Type (D) terms. 
 We first recall the $L^4$-Strichartz estimate from \cite[Lemma 5.1]{Tz2}.
Given $N,K\in 2^{\N_0}$, we let $\Q_{K}$ be a smooth space-time projection onto modulations $\jb{\tau-\ld_n^2} \sim K$ given by the multiplier $\wt{\chi}(K^{-1} \jb{\tau-\ld_n^2})$.

%
%
%
%

\begin{lemma}
\label{LEM:L4Strich}
Let $N,K\geq 1$ dyadic. Then, for any $\eps>0$, there exists $b_0<\frac 12$ such that 
\begin{align*}
\| \Q_{K}\P_{N}f\|_{L^{4}_{t,x}} \les N^{\eps} K^{b_0} \|\Q_{K}\P_{N}f\|_{L^{2}_{t,x}}.
\end{align*}
\end{lemma}

 We denote by $\P_{K}^{t}$ the smooth projection to temporal frequencies $|\tau | \sim K$ by the multiplier $\wt{\chi}(K^{-1}\tau)$.
Combining Lemma~\ref{LEM:L4Strich} with the Type (D) assumption \eqref{TypeD}, we have the following:

\begin{lemma}\label{LEM:Dstrich}
Let $w_{L}$ be of \textup{Type (D)}. Then, for any $4\leq p\leq \infty$, it holds that
\begin{align}
\| S(t) w_{L}\|_{L^{p}_{t,x}} &\les L^{\g - \frac{4}{p}+} \label{Dphys} \\ 
\| S(t)\P^{t}_{K}w_{L}\|_{L^{p}_{t,x}} &\les K^{-b} L^{\g - \frac{4}{p}+}. \label{Dphys2}
\end{align}
Here $S(t)$ is the linear semi-group defined in \eqref{St}.
\end{lemma}

\begin{proof}
Recall that $w$ is supported on the spatial frequencies in $E_{L}$.
 Then, it holds that $S(t)\P_{K}^{t}w= \Q_{K}S(t) w$.
Thus, by Lemma~\ref{LEM:L4Strich} and Sobolev embedding, we have
\begin{align*}
\| S(t)w_{L}\|_{L^{4}_{t,x}} &\les L^{\eps} \|w_{L}\|_{Y^{0,b}} \les L^{-1+\g+\eps}, \\
\| S(t)w_{L}\|_{L^{\infty}_{t,x}}& \les L \|w_{L}\|_{Y^{0,b}} \les L^{\g}.
\end{align*}
Then, \eqref{Dphys} follows by interpolating the above bounds and similarly for \eqref{Dphys2}.
\end{proof}

We move onto bounds for the Type (C) terms.
\begin{lemma}
    \label{LEM:TypeC}
    Let $w_j$ be of \textup{Type (C)} defined in \eqref{TypeC} and \eqref{TypeCnorm}. 
    Let $0 < \eps_0 \ll 1$.
    We also assume that $w_j(t)$ is supported temporally in $[-T, T]$. Then, given $\theta_0 > 0$, there exists $C_0$ such that there exists a $\mathcal B_{\le 2N_j}$-measurable set $\Xi$ such that $\mathbb P(\O \setminus \Xi) < C_0 e^{-  R^{2\theta_0}}$, such that on $\Xi$, we have
    \begin{align}\label{C_estimate1}
   \sup_{K\in 2^{\N}} K^{-\eps}\|S(t)\P^{t}_{K}w_j\|_{L^{q}_{t}L^{r}_{x}}+ \| S(t)w_j \|_{L^q_t L^r_x} &\le C_1 T^{\frac1q} R^{\theta_0} D(r, N_j), \\
  \sup_{K\in 2^{\N}} K^{(1-\eps_0) b - \eps} \| \P^t_{K} w_{j}\|_{L^{2}_{t,x}} 
  & \le C_1 T^{\eps_0} R^{\theta_0} N^{- \frac12}_j L_j^{- (1-\eps_0) \dl_0} . \label{C_estimate2}
    \end{align}
    where $b > \frac12$ is as in \eqref{TypeCnorm}, $C_0$ and $C_1 > 0$ depend only on $q ,r\geq  2$ and 
    \begin{align*}
D(r, N_j):=\Big( \ind_{\{ 2\leq r<4\}} N_j^{-\frac 12 }  +\ind_{\{ r=4\}} N_j^{-\frac 12+}+  \ind_{\{4< r\leq \infty\}} N_j^{-\frac 2r}\Big).
\end{align*}
Moreover, it holds that 
    \begin{align}
   \sup_{K\in 2^{\N}} K^{-\eps}\|S(t)\P^{t}_{K} \nb w_j\|_{L^{q}_{t}L^{r}_{x}}+ \| S(t) \nb w_j \|_{L^q_t L^r_x} &\le C_1 T^{\frac1q} R^{\theta_0} N_j D(r, N_j),\label{C_estimate3} \\
      \sup_{K\in 2^{\N}} K^{-\eps}\|S(t)\P^{t}_{K} \Dl w_j\|_{L^{q}_{t}L^{r}_{x}}+ \| S(t) \Dl w_j \|_{L^q_t L^r_x} &\le C_1 T^{\frac1q} R^{\theta_0} N_j^2 D(r, N_j).\label{C_estimate4}
    \end{align}
\end{lemma}

\begin{remark}\rm
\label{RMK:4.4}
The purpose of the bound for the first quantity in \eqref{C_estimate1} is to control the dyadic summations over modulation variables for Type (C) quantities. See Section~\ref{SEC:LWP}.
In the following, we may take $R = T^{-1} N_j$ and $\theta_0 = \theta \ll \eps_0$.
We also remark that in \eqref{C_estimate1}–\eqref{C_estimate4}, there appear factors with positive powers of $T$, such as $T^{\eps_0}$ and $T^{\frac1q}$. Such factors are sufficient to absorb any $T^{-\theta}$ terms when taking $R = T^{-1} N_j$.
Therefore, in what follows, we may omit these factors of $T$, as they are under control in any case.
\end{remark}

\begin{proof}
    We first consider when $w_j$ is of Type (G), as defined in \eqref{TypeG}. We note that the second bound in \eqref{C_estimate1} also holds without the linear propagator $S(t)$ appearing. Indeed, we will give the proof in this case.
    Now 
    \[
    w_j (t,x) = \chi_T (t) \sum_{\ld_{n} \sim N_j} \frac{g_{n} (\o)}{\ld_{n}} e_{n} (x). 
    \]
    It is easy to see that
    \begin{align}
    \label{L2G}
    \begin{split}
    \E \big[ |w_j (t,x)|^2 \big] & = |\chi_T (t) |^2 \sum_{\ld_{n} \sim N_j} \frac{e_{n}^2 (x) }{\ld_{n}^2} .
    \end{split}
    \end{align}
    Then, by using Minkowski's inequality and \eqref{L2G}, 
    we have
    \begin{align}
    \label{LpG}
    \begin{split}
    \big\| \|w_j\|_{L^q_t L_{x}^r} \big\|_{L^p_\o} 
     \le \big\| \|w_j\|_{L^p_\o} \big\|_{L^q_t L_{x}^r} &\le C \sqrt p \|w_j\|_{L^q_t L_{x}^r L^2_\o} \\
    & = C \sqrt p T^{\frac1q} \bigg\|   \sum_{n\sim N_j} \frac{e_n ^2(x)}{\ld_n^2}   \bigg\|_{L^{\frac{r}2}(\mathbb{D})}^{\frac{1}2} \\
    & \leq C\sqrt{p} T^{\frac 1q} \bigg( \sum_{n\sim N_j} \frac{1}{\ld_{n}^{2}} \| e_{n}\|_{L^{r}_{x}}^{2} \bigg)^{\frac 12} 
    \end{split} 
    \end{align}
    provided $p \ge \max(q, r)$.
    Then, \eqref{C_estimate1} follows from \eqref{efesti1}, $\ld_{n}\sim n$ and  a Nelson type argument. See \cite[Lemma 4.5]{TzBO}.
    Now we turn to the case where $w_j$ is of Type (C). Recall that the multiplier $h^{L,R}_{n}$ is independent of the Borel $\s$-algebra generated by Gaussians $\mathcal B_{\le L_j}$. Therefore, for fixed $(t, x)$ and using Lemma~\ref{LEM:unitary}, we get
    \begin{align}
    \label{L2C}
    \begin{split}
    \E \big[ |w_j (t,x)|^2 | \mathcal B_{\le L_j}\big]  =
    & | \chi_T (t)|^{2}  \sum_{\ld_{n} \sim N_j} \frac{e_{n}(x)^{2} }{\ld_{n}^2} .
    \end{split}
    \end{align}
    Then, by a similar argument as in \eqref{LpC} with the conditional Wiener chaos estimate (Lemma~\ref{LEM:condchaos}), we arrive at
    \[
    \big\| \|w_j\|_{L^q_t L_{x}^r} \big\|_{L^p_\o | \mathcal B_{\le L_j}} \le C \sqrt p T^{\frac1q}   D(r, N_j).
    \]
    Thus, we have 
    \begin{align}
        \label{LpC}
        \big\| \|w_j\|_{L^q_t L_{x}^r} \big\|_{L^p_\o} \le \Big\| \big\| \|w_j\|_{L^q_t L_{x}^r} \big\|_{L^p_\o | \mathcal B_{\le L_j}} \Big\|_{L^p_\o} \le C \sqrt p T^{\frac1q}  D(r, N_j),
    \end{align}
    which again suffices to \eqref{C_estimate1} with a Nelson type argument.
    
    Consider now the estimates for $S(t) \P_{K}^{t} w_{j}$. We have 
    \begin{align*}
S(t) \P_{K}^{t} w_{j} = \sum_{\ld_{n}\sim N_j} e^{-it \ld_{n}^{2}} \P_{K}^{t}[h_{n}^{N_j,L_j}](t) \frac{g_{n}}{\ld_{n}}e_{n}.
\end{align*}
Then, as in \eqref{L2C}, for fixed $(t,x)$ we have
\begin{align*}
\E[ |S(t) \P_{K}^{t} w_{j}(t,x)|^2 \vert \mathcal{B}_{\leq L_{j}}]& = \sum_{\ld_n \sim N_j} | \P_{K}^{t}[h_{n}^{N_j,L_j}](t)|^2 \frac{e_{n}(x)^2}{\ld_{n}^{2}}.
\end{align*}
Then, by the conditional Wiener chaos estimate, Minkowski's inequality ($q,r\geq 2$), we have
\begin{align*}
  \big\| \| S(t) \P_{K}^{t}w_j\|_{L^q_t L_{x}^r} \big\|_{L^p_\o} &\le \Big\| \big\| \|S(t)\P_{K}^{t}w_j\|_{L^q_t L_{x}^r} \big\|_{L^p_\o | \mathcal B_{\le L_j}} \Big\|_{L^p_\o}  \\
  & \les \sqrt{p}\Big\| \big\| \|S(t)\P_{K}^{t}w_j\|_{L^q_t L_{x}^r} \big\|_{L^2_\o | \mathcal B_{\le L_j}} \Big\|_{L^p_\o} \\
&   \les  \sqrt{p}  D(r,N_j) \Big\|    \sup_{n\in E_{N}\setminus E_{N/2}}\big\|  \P_{K}^{t}[h_{n}^{N_j,L_j}]  \big\|_{L^{q}_{t}}  \Big\|_{L^{p}_{\o}}.
\end{align*}

We write 
\begin{align*}
\P_{K}^{t}[ h_{n}^{N_j,L_j}](t) =\int_{\R} \chi_{T}(t')H_{n}^{N_j,L_{j}}(t') G_{K}(t-t')dt',
\end{align*}
where $G_{1}(t) := \int_{\R} e^{it \tau} \chi(\tau)d\tau$ and $G_{K}(t):= K G_{1}(Kt)$. By Lemma~\ref{LEM:unitary}, 
\begin{align*}
|\P_{K}^{t}[ h_{n}^{N_j,L_j}](t) | \leq \int_{\R} \chi_{T}(t') |G_{K}(t-t')|dt'.
\end{align*}
Therefore, 
\begin{align*}
 \Big\|    \sup_{n\in E_{N}\setminus E_{N/2}}\big\|  \P_{K}^{t}[h_{n}^{N_j,L_j}]  \big\|_{L^{q}_{t}}   \Big\|_{L^{p}_{\o}} \leq  \| |\chi_{T}| \ast |q_{K}| \|_{L^{q}_{t}} \leq \| \chi_{T}\|_{L^{q}_{t}} \| G_{K}\|_{L^{1}_{t}} \les T^{\frac{1}{q}},
\end{align*}
uniformly in $K\in 2^{\N}$.  This establishes that 
\begin{align*}
 \sup_{K\in 2^{\N}}\big\| \| S(t) \P_{K}^{t}w_j\|_{L^q_t L_{x}^r} \big\|_{L^p_\o}  \les \sqrt{p} T^{\frac 1q} D(r,N_j).
\end{align*}
In order to obtain this bound with the supremum over $K$ inside the $L^p (\O)$ norm, we use an additional factor $K^{\eps}$ as in \eqref{C_estimate1} and the embedding property $\l^{\infty}(\N) \subset \l^{p}(\N)$.
    
We move onto \eqref{C_estimate2}, for which we only consider the case where $w_j$ is of Type (C). By \eqref{TypeC}, \eqref{TypeCnorm}, we have
 \begin{align}
 \begin{split}
\E \big[ \| \P_{K}^t w_j\|_{L^{2}_{t,x}}^2 \big] & = \sum_{n \in E_{N_j} \backslash E_{N_j/2}} {\ld_{n}^{-2}} \E \big[ \|\P_{K}^{t}[h_{n}^{N_j,L_j}] \|_{L^2_t}^2 \big] \\
& \les  N_j^{-1} \sup_{\o \in \O} \sup_{n \in E_{N_j} \backslash E_{N_j/2}} \Big\| \ind_{\jb{\tau}\sim K} \ft{h_{n}^{N_j, L_j}}(\tau, \o) \Big\|_{L^{2}_{\tau} }^2 \\
& \les  N_j^{- 1} L_{j}^{- 2\dl_0}K^{- 2b}.
\end{split}
\label{C22}
\end{align}
On the other hand, we have the crude estimate
 \begin{align}
 \begin{split}
\E \big[ \| \P_{K}^t w_j\|_{L^{2}_{t,x}}^2 \big] 
& = \sum_{n \in E_{N_j} \backslash E_{N_j/2}} {\ld_{n}^{-2}} \| h_{n}^{N_j,L_j} \|_{L^2_t}^2   \les T N_j^{-1},
\end{split}
\label{C23}
\end{align}
where we used the fact that $h_{n}^{N_j,L_j}$ is compactly supported on $[-2T,2T]$.
By interpolating \eqref{C22} and \eqref{C23},
we arrive at 
\[
\E \big[ \| \P_{K}^t w_j\|_{L^{2}_{t,x}}^2 \big] \les T^{\eps_0} N^{-1}_j L_j^{- 2(1-\eps) \dl_0} K^{- 2(1-\eps) b}.
\]
This completes the proof of \eqref{C_estimate2} with a Nelson type argument.

For the bound \eqref{C_estimate3}, we use the same ideas as above but instead of using the $L^r$-bounds for the eigenfunctions $\{e_n\}$ in \eqref{LpG}, we use the $L^r$ bounds for the functions $\{e'_n\}$, which are also given in \eqref{efesti1}. For \eqref{C_estimate4}, the proof is the same as in \eqref{C_estimate1} since $-\Dl e_n = \ld_n^2 e_n$.
This completes the proof. 
\end{proof}

\subsection{RRO estimate and the proof of Proposition \ref{mainprop} - (i)}
\label{SUB:RRO}
In this section, 
we shall study the random resonant operator $\Pp_{N,L_2,\ldots,L_{2k+1}}$ defined in \eqref{linmap}
and prove \eqref{PwZbb}. 
By our induction hypothesis, it suffices to prove \eqref{PwZbb} for the case when $L_{\max}=M$. 
To simplify the notation, we write $\Pp_{N,L_{\max}}$ in place of $\Pp_{N,L_2,\ldots,L_{2k+1}}$.

Given $f\in Y^{s,b}(J)$ for $I = [-T,T]$, let $f^{\dag}\in Y^{s,b}$ be any extension of $f$.
First, notice that $\Pp_{N,L_{\max}}[f](t=0)=0$. 
Therefore, by Lemma~\ref{PROP:short}, we have
\begin{align}
\| \Pp_{N,L_{\max}}[f]\|_{Y^{s,b}(I)} \leq \| \chi_{T} \Pp_{N,L_{\max}}[f^{\dag}]\|_{Y^{s,b}}  \les  T^{b_{-}-b} \| \Pp_{N,L_{\max}}[f^{\dag}]\|_{Y^{s,b_{-}}},
\end{align}
thus gaining the small factor of $T$ required from \eqref{PwZbb}. 
Recalling \eqref{linmap} and using Lemma~\ref{LEM:duhamel}, we further obtain
\begin{align*}
 \| \Pp_{N,L_{\max}}[f^{\dag}]\|_{Y^{s,b_{-}}} &\les  \bigg\| \sum_{n\in E_N}e_n  \jb{\M( \{f^{\dag}\}_{n}, y_{L_2},\ldots, y_{L_{2k+1}}),e_n}  \bigg \|_{Y^{s,b_{-}-1}} \\
 & \sim  \bigg\| \sum_{n\in E_N}e_n  \jb{\M( \{f^{\dag}\}_{n}, w_{L_2},\ldots, w_{L_{2k+1}}),e_n}  \bigg \|_{Y^{s,b_{-}-1}}  
\end{align*}
where $w_{L_j}=y_{L_j}$.
Here we may further decompose each $w_{L_j}$ to be of either Type (C) \eqref{TypeC} or Type (D) \eqref{TypeD}. Moreover, since $n=n_1$ in the frequency support of the term in \eqref{RROcomp1}, we may assume that $s=0$ by associating the $s$-derivative weight with $f$.

By duality, \eqref{psiNL}, \eqref{HnNL}, Cauchy-Schwarz and H\"{o}lder's inequality (for $\dl>0$ sufficiently small), 
\begin{align}
\bigg\|& \sum_{n\in E_N}e_n \jb{ \M( \{f^{\dag}\}_{n}, w_{L_2},\ldots, w_{L_{2k+1}}) ,e_n} \bigg \|_{Y^{0,b_{-}-1}}  \notag \\
 &= \sup_{\|h\|_{Y^{0,1-b_{-}}}\leq 1} \bigg| \int_{\R} \int_{\D}   \cj{h}\sum_{n\in E_N}e_n \jb{ \M( \{f^{\dag}\}_{n}, w_{L_2},\ldots, w_{L_{2k+1}}) ,e_n}     dx dt  \bigg| \notag \\
 & \sim \sup_{\|h\|_{Y^{0,1-b_{-}}}\leq 1} \bigg| \int_{\R} \sum_{n\in E_{N}} \cj{\ft h(t,n)}\ft{f^{\dag}}(t,n) \bigg( \int_{\D}e_{n}(y)^2 \prod_{j=2}^{2k+1} (S(t)w_{L_j})^{\iota_j}(t,y)dy \bigg)  dt  \bigg| \notag \\
 & \les   \sup_{\|h\|_{Y^{0,1-b_{-}}}\leq 1}  \int_{\R} \|\ft h(t,n)\|_{\l^2_n}  \| \ft f(t,n)\|_{\l^{2}_{n}}
 \sup_{n\in E_{N}} \bigg| \int_{\D}e_{n}(y)^2 \prod_{j=2}^{2k+1} (S(t)w_{L_j})^{\iota_j}(t,y)dy \bigg| dt  \notag\\
 & \les \sup_{\|h\|_{Y^{0,1-b_{-}}}\leq 1} \| h\|_{L^{4}_{t}L^{2}_{x}} 
 \|f\|_{L_t^{ 4\frac{2+\dl}{2+3\dl}}L^{2}_{x}} \bigg\| \sup_{n\in E_{N}} \bigg| \int_{\D}e_{n}(y)^2 \prod_{j=2}^{2k+1} (S(t)w_{L_j})^{\iota_j}(t,y)dy \bigg| \bigg\|_{L^{2+\dl}_{t}}. \label{RROcomp1}
\end{align} 
By Sobolev embedding in time, it remains to control the last factor in \eqref{RROcomp1}.
By H\"{o}lder's inequality and \eqref{efesti1}, we have
\begin{align}
\bigg\| & \sup_{n\in E_{N}} \bigg| \int_{\D}e_{n}(y)^2 \prod_{j=2}^{2k+1} (S(t)w_{L_j})^{\iota_j}(t,y)dy \bigg| \bigg\|_{L^{2+\dl}_{t}} \notag \\
&\leq \bigg\| \sup_{n\in E_{N}}  \| e_{n}\|^2 _{L_x^{4-\frac{6\dl}{1+2\dl}}} \bigg\| \prod_{j=2}^{2k+1} (S(t)w_{L_j}) \bigg\|_{L^{2+\dl}_{x}} \bigg\|_{L^{2+\dl}_{t}} \notag \\
& \les  \bigg\|  \prod_{j=1}^{2k} (S(t)w_{L_{(j)}})   \bigg\|_{L^{2+\dl}_{t,x}} \notag \\
& \les \|S(t) w_{L_{(1)}}\|_{L^{4}_{t,x}} \prod_{j=2}^{2k} \| S(t) w_{L_{(j)}}\|_{L^{2(2+\dl)(2k-1)}_{t,x}} \label{RROcomp2}
\end{align}
It follows from \eqref{Dphys} and \eqref{C_estimate1} that 
\begin{align*}
\|S(t) w_{L_{(1)}}\|_{L^{4}_{t,x}} \les 
\begin{cases}
L_{(1)}^{-1+\g+} \quad & \text{if of Type (D)},\\
L_{(1)}^{-\frac 12+}\quad & \text{if of Type (C)},
\end{cases}
\end{align*}
and 
\begin{align*}
\|S(t) w_{L_{(j)}}\|_{L^{4(2k-1)+O(\dl)}_{t,x}} \les 
\begin{cases}
L_{(j)}^{\g-\frac{1}{2k-1}+O(\dl)} \quad & \text{if of Type (D)},\\
L_{(j)}^{-\frac{1}{2(2k-1)}+O(\dl)}\quad & \text{if of Type (C)}.
\end{cases}
\end{align*}
Thus, we find that in over all possible choices $\{w_{L_j}\}_{j=1}^{2k+1}$, we have the bound
\begin{align*}
\eqref{RROcomp2} \les L_{(1)}^{-\frac 12+} \ll L_{\max}^{-\ta_0}.
\end{align*} 
This completes the proof of \eqref{PwZbb}.

\subsection{Proof of Proposition \ref{mainprop} - (ii)}
\label{SUB:main2}

In Subsection \ref{SUB:RRO},
we have verified Proposition~\ref{mainprop} - (i), i.e. \eqref{PwZbb}. 
We now verify that (ii) of Proposition~\ref{mainprop} holds; namely, 
we show \eqref{hNLbound}.  
We have
\begin{align}
\begin{split}
\Pp^{N,L}[f] &= \Pp^{N,\frac{L}{2}}[f]-(k+1)i \sum_{L_{\max}=L} \I_{\chi}\Big[ \sum_{n\in E_{N}} e_n \jb{ \M(\{f\}_{n}, y_{L_2}, \ldots, y_{L_{2k+1}}),e_n} \Big] \\
& = \Pp^{N,\frac{L}{2}}[f]+ \sum_{L_{\max}=L}\Pp_{N,L_2,\ldots,L_{2k+1}}[f],
\end{split}\label{PNLformula}
\end{align}
and hence 
\begin{align*}
\Pp^{N,L} = \sum_{L_{\max}\leq L}\Pp_{N,L_2,\ldots,L_{2k+1}}.
\end{align*}
Then, by choosing $0<T\ll 1$ sufficiently small, it follows from \eqref{PwZbb} that 
for all $L'\leq L$, the resolvent $\H^{N,L}$ of $\Pp^{N,L}$ satisfies 
\begin{align}
\max_{L'\leq L}\|\H^{N,L}\|_{Y^{s,b}(I) \to Y^{s,b}(I)} \les 1. \label{HNLbd}
\end{align}
with convergent Neumann series in the $Y^{s,b}(I) \to Y^{s,b}(I)$ operator norm with
\begin{align*}
\H^{N,L} =\text{Id}+\sum_{\l=1}^{\infty} (\Pp^{N,L})^{\l},
\end{align*}
Now, by the second resolvent identity, we have 
\begin{align*}
\mathfrak{h}^{N,L} = \H^{N,L}-\H^{N,\frac{L}{2}} &= (1-\Pp^{N,L})^{-1} -(1-\Pp^{N,\frac{L}{2}})^{-1}  \\
& =(1-\Pp^{N,L})^{-1} ( \Pp^{N,L}-\Pp^{N,\frac{L}{2}}) (1-\Pp^{N,L})^{-1} \\
& = \H^{N,\frac{L}{2}} ( \Pp^{N,L}-\Pp^{N,\frac{L}{2}}) \H^{N,L},
\end{align*}
and thus \eqref{hNLbound} follows from \eqref{PNLformula}, \eqref{HNLbd}, and \eqref{PwZbb} with the logarithmic loss due to the sums over $L_{2},\ldots, L_{2k+1}\leq L$, since 
\[
\Pp^{N,L}-\Pp^{N,\frac{L}{2}} = \sum_{\frac{L}2 < L_{\max} =L} \Pp_{N,L_2,\ldots,L_{2k+1}}.
\]
We remark that the lower bound for $L_{\max}$ on the right-hand side of the above formula is crucial for obtaining the decay in $L$ in \eqref{hNLbound}.

\section{The smoother remainder term}
\label{SEC:LWP}

In this section, we complete the proof of Proposition \ref{mainprop} by verifying the a priori bounds for the remainder terms $z_N$ in part (iii). In view of the inductive setup to establishing Proposition \ref{mainprop}, it remains to consider the case when $N=2M$. More precisely, we aim to show that 
\begin{align}
\|z_{2M}\|_{Y^{0,b}(I)} \leq (2M)^{-1+\g}. \label{z2M}
\end{align}
assuming that $\textsf{Local}(M')$ holds for all $M'\leq M$. Note that we have also just proved that (i) and (ii) in Proposition \ref{mainprop} hold as part of $\textsf{Local}(2M)$.

  We verify \eqref{z2M} by a continuity argument recalling that $z_{2M}$ satisfies the equation \eqref{eqn:zn2}. In particular, recalling \eqref{zeqA}, the right hand side of \eqref{eqn:zn2} is a linear combination of the terms
\begin{align*}
\mathcal{A}_{2M}^{(J)}( w_{N_1}, \ldots, w_{N_{2k+1}}), \quad J\in \{1,2,3,4,5,6\},
\end{align*}
where the input functions $\{w_{N_j}\}$ are either of Type (G), (C), or (D), and depend on $J$ in the following way: 

\medskip
\noi
\underline{\textbf{$J=1$:}} $\{w_{N_j}\}_{j=1}^{2k+1}$ are of any type and $N_{(1)}=N_{(2)}=2M$ .

\medskip
\noi
\underline{\textbf{$J=2$:}} $w_{N_1}=z_{2M}$, $\{w_{N_j}\}_{j=2}^{2k+1}$ are of any type and $N_{j}\le M$ for all $j\in \{2,\ldots, 2k+1\}$.

\medskip
\noi
\underline{\textbf{$J=3$:}} $w_{N_1}=\psi_{2M,L_{2M}}$, $n\neq n_1$, $\{w_{N_j}\}_{j=2}^{2k+1}$ are of any type, $N_{j}\le L_{2M}$ for all $j\in \{2,\ldots, 2k+1\}$.

\medskip
\noi
\underline{\textbf{$J=4$:}} $w_{N_1}=\psi_{2M,L_{2M}}$, $\{ w_{N_j}\}_{j=1}^{2k+1}$ are of any type with $\max_{j=2,\ldots, 2k+1}N_j > L_{2M}$.

\medskip
\noi
\underline{\textbf{$J=5$:}} $\{ w_{N_j}\}_{j=1}^{2k+1}$ are of any type with $N_2=2M$ and $N_j\le M$ for all $j\in \{1,\ldots, 2k+1\}\setminus \{2\}$.

\medskip
\noi
\underline{\textbf{$J=6$:}}  $\{w_{N_j}\}_{j=1}^{2k+1}$ are of any type, $N_{j}\leq M$ for all $j\in \{1,\ldots, 2k+1\}$, output frequency $\sim 2M$.

\medskip
\noi
Thus, in order to show \eqref{z2M}, it suffices to prove that
\begin{align}
\max_{J=1,\ldots, 6}\| \mathcal{A}_{2M}^{(J)}(w_{N_1},\ldots, w_{N_{2k+1}})\|_{Y^{0,b}(I)} \ll (2M)^{-1+\frac{1}{2}\g},
\label{Ab}
\end{align}
using each of the possibilities for $\{w_{N_j}\}_{j=1}^{2k+1}$ as delineated above. 


We first follow the reductions in Section~\ref{SEC:RRO}, using \eqref{localtime} to gain a small factor of $T$ at the expense of a worse exponent $b_{-}$ instead of $b$ and after passing to the extensions $\{w^{\dag}_{N_j}\}_{j=1}^{2k+1}$ of $\{w_{N_j}\}_{j=1}^{2k+1}$. Then, the required estimate with time regularity $b_{-}$, i.e. \eqref{Ab} with $b$ replaced by $b-$, 
follows from interpolation between the bounds:
\begin{align}
\max_{J=1,\ldots, 6}\| \mathcal{A}_{2M}^{(J)}(w^{\dag}_{N_1},\ldots, w^{\dag}_{N_{2k+1}})\|_{Y^{0,1}} &\les (2M)^{100 k},  \label{Ab1}\\
\max_{J=1,\ldots,6}\| \mathcal{A}_{2M}^{(J)}(w^{\dag}_{N_1},\ldots, w^{\dag}_{N_{2k+1}})\|_{Y^{0,b'}} &\les (2M)^{-1+\frac{1}{4}\g}. \label{Ab'}
\end{align}

\begin{remark}\rm
As a matter of fact, an extra factor of $T^{\varepsilon_0}$ is allowed in \eqref{Ab'} in view of Lemma~\ref{PROP:short} and Lemma~\ref{LEM:TypeC} for some $\eps_0 > 0$. 
Moreover, this factor can absorb any occurrence of $T^{-\theta}$ appearing in the analysis by choosing $\theta \ll \eps_0$. 
For simplicity of notation, however, we will ignore these $T$-factors throughout the section and assume that the implicit constant in $\lesssim$ may depend on a positive power of $T$.
See also Remark \ref{RMK:4.4} for a similar consideration.
\end{remark}

Since $N_{j} \les M$ for all $j=1,\ldots,2k+1$, \eqref{Ab1} follows trivially from H\"{o}lder's inequality, Sobolev embeddings and that by the inductive set-up, we have
\begin{align*}
\| w^{\dag}_{N_j}\|_{Y^{0,b}}\les N_{j}. 
\end{align*}
Our main concern for the rest of this section is to establish \eqref{Ab'}. In the following, we will drop the $\dag$ notation. 

We dispense with the cases $J\in \{1,4,5\}$ entirely by physical space estimates. 
For $J=2$, we use physical space estimates, multilinear analysis in Fourier space, 
and our bound in \eqref{PwZbb} which we justified already in Section \ref{SEC:RRO}. 
Most of the cases for $J=6$ can be dealt with by arguing as in the appropriate $J\in \{1,2,4,5\}$ case. 
Then, it remains to control the $J=3$ and the remaining cases in $J=6$, for which we need to use Fourier space arguments and multilinear probabilistic oscillations.

\subsection{$J\in \{1,4\}$ cases} 

When $J=1$, two input functions (of any type) have frequencies of size $2M$. 
The $J=4$ case is similar except the largest frequency corresponds to a Type (C) term and we recover the final regularity thanks to the lower bound $N_{\rm (2)}\geq L_{2M} \ges (2M)^{1-\kk}$,
where $N_{\rm (2)}$ is the second largest frequency among $N_j$ for $1 \le j \le 2k+1$.

\begin{lemma}\label{LEM:J1}
It holds that 
\begin{align*}
\max_{J=1,4}\| \mathcal{A}_{2M}^{(J)}(w^{\dag}_{N_1},\ldots, w^{\dag}_{N_{2k+1}})\|_{Y^{0,b'}}  &\les (2M)^{-1+\frac{1}{4}\g}
\end{align*}
for $\{w_{N_j}\}_{j=1}^{2k+1}$ of any combination of \textup{Types} \textup{(G), (C)} or \textup{(D)} when $J=1$ and $w_{N_1}=\psi_{2M,L_{2M}}$ when $J=4$.
\end{lemma}

\begin{proof}
We focus first on the case $J=1$.
As we will not make use of the conjugations, by symmetry, we assume that $N_{1}=N_{(1)}$ and $N_{2}=N_{(2)}$. 
By duality, it suffices to control 
\begin{align}
\bigg|\iint_{\R\times \mathbb{D}} \P_{\leq 2M}v  \prod_{j=1}^{2k+1} (S(t)w_{N_j})^{\iota_j} dxdt \bigg|, \label{duality}
\end{align}
for any $v\in X^{0,b'}$ with $\|v\|_{X^{0,b'}}\leq 1$, $b'<\frac 12$, and with $N_1=N_2=2M$.
Here $S(t)$ in \eqref{duality} is the linear semi-group defined in \eqref{St}.

We consider a few cases depending on the types of the two highest frequency terms. 
Assume that $w_{N_1}=z_{2M}$ and $w_{N_2}=z_{2M}$. Then for $\eps>0$ to be chosen later, by H\"older's inequality with $r=2(2k-1)[1+\eps^{-1}]$ and \eqref{Dphys}, we have
\begin{align}
\eqref{duality} &\leq \|S(t)z_{2M}\|^2_{L^{4(1+\eps)}_{t,x}}  \|\P_{\leq 2M}v\|_{L^{2}_{t,x}} \prod_{j=3}^{2k+1} \|S(t)w_{N_j}\|_{L^{r}_{t,x}} \notag \\
& \les  (2M)^{-2+2\g +\frac{2\eps}{1+\eps}}\prod_{j=3}^{2k+1} \|S(t)w_{N_j}\|_{L^{r}_{t,x}}. 
\label{J1zz}
\end{align}
In view of \eqref{Dphys} and Lemma~\ref{LEM:TypeC}, we have the following crude bound $T^{-1} M$-certainly,
\begin{align}
\label{Lr_general} 
\|S(t)w_{N_j}\|_{L^{r}_{t,x}} \les N_j^{\g} \les (2M)^{\g}
\end{align}
amongst all $w_{N_j}$ of either Type (C) or Type (D),
provided $r \le \infty$. 
Plugging \eqref{Lr_general} in \eqref{J1zz}, we get
\begin{align*}
\eqref{duality} \les (2M)^{-2+(2k+1)\g + 2\eps}  \ll (2M)^{-1+\frac{1}{4}\g} 
\end{align*}
provided that $\g$ and $\eps>0$ are chosen small enough.

Now, assume that, say, $w_{N_1}$ is of Type (D) and $w_{N_2}$ is of Type (C). We proceed as before now using Lemma~\ref{LEM:TypeC} and obtain
\begin{align*}
\begin{split}
\eqref{duality} &\leq \|S(t) z_{2M}\|_{L^{4(1+\eps)}_{t,x}} \|S(t) \psi_{2M,L_{2M}}\|_{L^{4(1+\eps)}_{t,x}}  \|\P_{\leq 2M}v\|_{L^{2}_{t,x}} \prod_{j=3}^{2k+1} \|S(t)w_{N_j}\|_{L^{r}_{t,x}} \notag \\
& \les  (2M)^{-1 - \frac1{2(1+ \eps)} + 2\eps +\g }\prod_{j=3}^{2k+1} \|S(t)w_{N_j}\|_{L^{r}_{t,x}} \\
& \les (2M)^{- \frac32 + 3\eps + 2k\g }  \ll (2M)^{-1+\frac{1}{4}\g},
\end{split}
\end{align*} 
for $\g$ chosen sufficiently small.

Finally, we assume that both $w_{N_1}$ and $w_{N_2}$ are of Type (C). Let $q=8(k-1)(1+\eps^{-1})$. By H\"{o}lder's inequality Lemma~\ref{LEM:L4Strich}, Lemma \ref{LEM:Dstrich}, Lemma~\ref{LEM:TypeC}, 
and recalling that $N_{4}\les M$,  we get 
\begin{align}
\begin{split}
\eqref{duality} & \les \| S(t) \psi_{2M,L_{2M}}\|_{L^{4}_{t,x}}^{2} \|\P_{\leq 2M}v\|_{L^{4}_{t,x}}  \|S(t)w_{N_3}\|_{L^{4(1+\eps)}_{t,x}} \prod_{j=4}^{2k+1}\|S(t)w_{N_j}\|_{L^{q}_{t,x}}  \\
&\les (2M)^{-1+\eps} N_{3}^{-\frac{1}{2(1+\eps)}}N_{4}^{(2k-1)\g}  \\
& \les (2M)^{-1+\eps} \les (2M)^{-1+\frac{1}{4}\g}
\end{split} \label{J1J4}
\end{align}
provided that $4(2k-1)\g <1$ and $\eps<\frac{1}{4}\g$.
The bound on $w_{N_3}$ follows by considering the two cases of Type (C) or (D), with Type (C) being the most restrictive case.
This completes the proof of Lemma~\ref{LEM:J1} for $J=1$.

When $J=4$, we have two cases: either $w_{N_2}$ is of Type (C) or (D). We only give the details for the harder Type (C) case. We argue as in \eqref{J1J4} and use that $N_{2} \ges (2M)^{1-\kk}$ to obtain
\begin{align*}
\eqref{duality}  & \les \|\P_{\leq 2M}v\|_{L^{4}_{t,x}}  \| S(t) \psi_{2M, L_{2M}}\|_{L^{4}_{t,x}}\| S(t) w_{N_2} \|_{L^{4}_{t,x}} \|S(t)w_{N_3}\|_{L^{4(1+\eps)}_{t,x}} \prod_{j=4}^{2k+1}\|S(t)w_{N_j}\|_{L^{q}_{t,x}}  \\
& \les (2M)^{-\frac 12+\eps}N_{2}^{-\frac 12+\eps} N_{3}^{-\frac{1}{2(1+\eps)}} N_{4}^{(2k-1)\g} \\
& \les (2M)^{-1+\frac{\kk}{2}+5\eps} \ll (2M)^{-1+\frac{1}{4}\g},
\end{align*}
by using \eqref{gamma} and choosing $\eps>0$ sufficiently small. 
\end{proof}

\subsection{$J=5$ case}

In this case, the physical space argument used in the proof of Lemma \ref{LEM:J1} is insufficient. 
To complete the proof, a Fourier-analytic approach is essential, making use of the significant gain from the phase function as detailed in \eqref{J42}.

\begin{lemma}\label{LEM:J5}
It holds that
\begin{align}
\| \mathcal{A}_{2M}^{(5)}(w^{\dag}_{N_1},\ldots, w^{\dag}_{N_{2k+1}})\|_{Y^{0,b'}}  &\les (2M)^{-1+\frac{1}{4}\g} \label{J50}
\end{align}
for $\{w_{N_j}\}_{j=1}^{2k+1}$ of any combination of \textup{Types} \textup{(G), (C)} or \textup{(D)}.
\end{lemma}

\begin{proof}

We may assume that 
\begin{align}
N_{(2)} \leq \tfrac{1}{100 k } 2M, \label{J41}
\end{align}
 since otherwise we can apply the same argument as in Lemma~\ref{LEM:J1}.  Under the assumption \eqref{J41}, we have 
 \begin{align}
| \Phi(\cj{n})| =| \ld_{n}^{2}+\ld_{n_2}^{2}-\ld_{n_1}^{2} + \ldots+ \ld_{n_{2k+1}}^{2}| \geq \tfrac{1}{2} \ld_{n_2}^{2} \sim M^{2}. \label{J42}
\end{align}
As in the proof of Lemma~\ref{LEM:J1} we also argue by duality, controlling the left-hand side of \eqref{J50} by
\begin{align}
\iint_{\R\times \mathbb{D}} \P_{\leq 2M}v  \prod_{j=1}^{2k+1} S(t)w_{N_j}^{\iota_j} dxdt , \label{duality2}
\end{align}
for $v\in X^{0,b'}$ with $b'<\frac 12$, $N_{2}=2M$ and with \eqref{J41}. With $\s_0:=\tau+\ld_{n}^{2}$ the modulation variable for $\P_{\leq 2M}v$ and $\s_j:= \tau_j+\ld_{n_j}^{2}$ the modulation for $S(t)w_{N_j}$, $j\in \{1,\ldots, 2k+1\}$, \eqref{J42} implies 
\begin{align}
\max_{j=0,\ldots, 2k+1} \jb{\s_j} \geq C M^{2}. \label{J43}
\end{align}
For $K>0$, let $\Q_{\geq K}$ be the space-time Fourier multiplier operator with symbol $\ind_{\{ \jb{\tau+\ld_n^2} \geq K\}}(\tau,n)$.
We perform a dyadic decomposition amongst the modulations 
\begin{align*}
\eqref{duality2} = \sum_{ \substack{K_0, \ldots ,K_{2k+1} \\ K_{\max} \geq CM^2}} \iint_{\R\times \mathbb{D}} \P_{\leq 2M} \Q_{K_0}v  \prod_{j=1}^{2k+1} (\Q_{K_j}S(t)w_{N_j})^{\iota_j} dxdt =:\sum_{ \substack{K_0, \ldots ,K_{2k+1} \\ K_{\max} \geq CM^2}} I(\cj{N}, \cj{K})
\end{align*}
with $K_{\max}:=\max_{j=0,\ldots, 2k+1}K_j$ and the summation restriction arising from \eqref{J43}.
Assume first that $w_{N_2}$ is of Type (D).
If $K_0=K_{\max}$, then by \eqref{Dphys2} (for Type (D) terms) and \eqref{C_estimate1} (for Type (C) terms),
\begin{align*}
|I(\cj{N}, \cj{K})| &\leq \| \P_{\leq 2M} \Q_{K_0}  v \|_{L^{2}_{t,x}} \| \Q_{K_2} S(t) z_{2M}\|_{L^4_{t,x}} \prod_{ \substack{j=1 \\ j\neq 2}}^{2k+1} \| \Q_{K_j} S(t)w_{N_j}\|_{L^{8k}_{t,x}} \\
& \les K_{0}^{-b'} (2M)^{-1+\g+\eps} K_{\max}^{2k \eps} (2M)^{2k \g}  \\
& \les K_{0}^{-b'+2k \eps}  (2M)^{-1+4k\g} \\
& \les K_0^{-\eps} (2M)^{-1-2b'+6k \eps +4k\g} \ll (2M)^{-1+\frac{1}{4}\g} K_0^{-\eps},
\end{align*}
provided that $\eps,\g>0$ are small enough.
If $S(t)w_{N_2} = S(t) z_{2M}$ has the largest modulation, then we place it instead into $L^{2}_{t,x}$ and similarly, if one of $S(t)w_{N_j}$ for some $j\in \{1,3,\ldots, 2k+1\}$ has the maximum modulation, it is placed into $L^2_{t,x}$, the $\P_{\leq 2M}\Q_{K_0}v$ into $L^{4+}_{t,x}$ and $\Q_{K_2} S(t)w_{N_2}$ into $L^4_{t,x}$.
In each of these cases, we end up with a final factor of $(2M)^{-2+O(\g)}$. 

Now suppose that $ w_{N_2}$ is of Type (C). We apply the same argument, but instead of always gaining at least $(2M)^{-1+\g+}$ from $S(t)w_{N_2}$, we only gain at least $(2M)^{-\frac 12}$. This is more than sufficient to obtain a final bound of size $(2M)^{-1+\frac{1}{4}\g}K_{\max}^{0-}$ on $I(\cj{N}, \cj{K})$ by considering cases depending on which function has the maximum modulation and 
using  \eqref{Dphys2}, \eqref{J43}, and \eqref{C_estimate2}, when necessary.
\end{proof}

\subsection{$J=2$ case}

This case requires a combination of physical space and Fourier-analytic arguments, along with the off-diagonal decay of the eigenfunction correlation function $c(\bar{n})$ (defined in \eqref{c_n}), as provided by Lemma \ref{LEM:ev5}.

\begin{lemma}\label{LEM:J2}
It holds that
\begin{align}
\| \mathcal{A}_{2M}^{(2)}(z_{2M}, w_{N_2}^{\dag}, \ldots, w^{\dag}_{N_{2k+1}})\|_{Y^{0,b}}  &\les (2M)^{-1+\frac{1}{4}\g} \label{J20}
\end{align}
for $\{w_{N_j}\}_{j=2}^{2k+1}$ of any combination of \textup{Types} \textup{(G), (C)} or \textup{(D)}.
\end{lemma}

\begin{proof}
By using the same arguments that we made for the $J\in \{1,4\}$ cases, we reduce to an essential part. 
Suppose that there exists $j_{\ast}\in\{2, \ldots, 2k+1\}$ such that $N_{j_{\ast} }\ges M^{\be}$ for some $0<\be  \ll 1$ to be chosen later. In this case, we establish \eqref{Ab'} and then interpolate with \eqref{Ab1} to obtain the desired result at time regularity $b>\frac 12$.
 We only consider the worst case that $w_{N_{j_{\ast}}}$ is of Type (C) or (G). Then by duality and H\"{o}lder's inequality,
\begin{align*}
\bigg|\iint_{\R\times \mathbb{D}} &\P_{\leq 2M}v \cdot S(t) z_{2M}  \prod_{j=2}^{2k+1} S(t)w_{N_j}^{\iota_j} dxdt \bigg|  \\
& \les  \|\P_{\leq 2M}v\|_{L^{4(1+\eps)}_{t,x}} \| S(t) z_{2M}\|_{L^{4(1+\eps)}_{t,x}}   \| S(t)w_{N_{j_{\ast}}} \|_{L^{2}_{t,x}} 
\prod_{ \substack{j=2 \\ j\neq j_{\ast} }}^{2k+1}  \| S(t)w_{N_j}\|_{L^{r}_{t,x}} \\
& \les (2M)^{-1+\g+2\eps} N_{j_{\ast}}^{-\frac{1}{2}+\eps} (2M)^{2k \g} \\
& \les (2M)^{-1+(2k+1)\g+3\eps - \frac{\be}{2}}  \ll (2M)^{-1+\frac{1}{4}\g} 
\end{align*}
provided for any $\be =c_0 k\g<1$ with $\g \ll 1/k$ and $c_0>0$ a fixed universal constant.

Thus, we may assume that
\begin{align} 
N_{j} \ll M^{\be} \quad \text{for all}\quad j\in \{2,\ldots, 2k+1\}. \label{J21}
\end{align}
We have thus reduced to considering the following portion of $\mathcal{A}_{2M}^{(2)}$:
\begin{align*}
&   - (k+1)i\, \ind_{\{N_j \ll M^{\be},\, j \neq 1\}}  \chi (t) \int_0^t \chi (t') \P_{\le 2M}   \mathcal{M} ( z_{2M} ,  w_{N_2}, \cdots, w_{N_{2k+1}}) (t')  dt' \\ 
& = - (k+1)i \,  \ind_{\{N_j \ll M^{\be},\, j \neq 1\}}  \chi (t) \int_0^t \chi (t') \sum_{n\in E_{2M}}e_{n} \jb{ \mathcal{M} ( z_{2M} ,  w_{N_2}, \cdots, w_{N_{2k+1}}) (t'), e_n}  dt'
\end{align*}
Consider the contribution
\begin{align}
\begin{split}
&  - (k+1)i  ,  \ind_{\{N_j \ll M^{\be},\, j \neq 1\}}  \chi (t) \int_0^t \chi (t') \sum_{n\in E_{2M}}e_{n} \jb{ \mathcal{M} ( \{z_{2M}\}_{n} ,  w_{N_2}, \cdots, w_{N_{2k+1}}) (t'), e_n}  dt'.
\end{split} \label{J2eq}
\end{align}
This corresponds to the case when $n=n_1$. We establish the estimate for this piece \eqref{J2eq} directly in $Y^{0,b}$ to show this contribution to \eqref{z2M}. Namely, we do not need to use the interpolation argument as in \eqref{Ab1}-\eqref{Ab'}.  Indeed, we recognise this contribution to be
\begin{align*}
 \mathcal{P}_{2M, N_2, \ldots, N_{2k+1}}[z_{2M}] 
\end{align*}
and thus by \eqref{PwZbb}, we have 
\begin{align*}
\|\mathcal{P}_{2M, N_2, \ldots, N_{2k+1}}[z_{2M}]\|_{Y^{0,b}} \les N_{(2)}^{-\ta_0} \| z_{2M}\|_{Y^{0,b}} \les N_{(2)}^{-\ta_0} (2M)^{-1+\g}.
\end{align*}
The negative power of $N_{(2)}$ allows us to sum over the dyadic scales $N_2,\ldots, N_{2k+1}$ in \eqref{zeqA}.

Therefore, for \eqref{J20}, 
it remains to consider the contribution 
\begin{align}
 \mathcal{A}_{2M, \text{nr}}^{(2)}(t)= \chi (t) \int_0^t \chi (t') \sum_{n\in E_{2M}}e_{n} \jb{ \mathcal{M} ( z_{2M} -\{z_{2M}\}_n,  w_{N_2}, \cdots, w_{N_{2k+1}}) (t'), e_n}  dt' \label{J2neq}
\end{align}
under the assumption \eqref{J21}. We rewrite this with the tensor notation from \eqref{basetensor0} and \eqref{baseS} and use \eqref{duhamelform} to obtain
\begin{align}
\begin{split}
\mathcal{F}\{ \mathcal{A}_{2M, \text{nr}}^{(2)}\}(\tau,n) &=   \int_{\R^{2k+1}} \sum_{n_1,\cdots,n_{2 k+1}}  \mathcal K \bigg(\tau,\Phi+\sum_{j=1}^{2 k+1} \iota_j \tau_j \bigg)  \ind_{S_2} (n, n_1,n_2,\cdots,n_{2k+1})  \\
& \hphantom{XXXX} \times c(n,n_1,n_2,\ldots,n_{2k+1}) \cdot
\prod_{j=1}^{2 k+1} \widehat{w_{N_{j}}} (  \tau_j, n_j)^{\iota_i} \, d \tau_1 \cdots  d\tau_{2 k+1}
\end{split} \label{J2FT}
\end{align}
where
\begin{align*}
S_2& =  \{ n_j \in E_{N_j}\setminus E_{N_j/2} \,\, \text{and}\,\,  N_{j} \ll M^{\be} \,\, \forall \, j\in \{2,\ldots, 2k+1\}, \ld_n \le 2M= N_1, n\neq n_1\}.
\end{align*}
We decompose \eqref{J2FT} as
\begin{align*}
\mathcal{F}\{ \mathcal{A}_{2M, \text{nr}}^{(2)}\}(\tau,n) = \ft{\mathcal{A}}_{\text{nr},1}(\tau,n) + \ft{\mathcal{A}}_{\text{nr},2}(\tau,n),
\end{align*}
where 
\begin{align}
\begin{split}
\ft{\mathcal{A}}_{\text{nr},1}(\tau,n) &=  \int_{\R^{2k+1}}  \ind_{\{ |\tau|\ll M\}}   \prod_{j=1}^{2k+1} \ind_{\{ |\tau_j|\ll M\}}   \sum_{n_1,\cdots,n_{2 k+1}}  \mathcal K \bigg(\tau,\Phi+\sum_{j=1}^{2 k+1} \iota_j \tau_j \bigg)  \ind_{S_2}  \\
& \hphantom{XXXX} \times c(n,n_1,n_2,\ldots,n_{2k+1}) \cdot
\prod_{j=1}^{2 k+1} \widehat{w_{N_{j}}} (  \tau_j, n_j)^{\iota_i} \, d \tau_1 \cdots  d\tau_{2 k+1}
\end{split} \label{J2A1}
\end{align}
and $\ft{\mathcal{A}}_{\text{nr},2}(\tau,n): = \mathcal{F}\{ \mathcal{A}_{2M, \text{nr}}^{(2)}\}(\tau,n) - \ft{\mathcal{A}}_{\text{nr},1}(\tau,n)$.

We consider first the bound for $\ft{\mathcal{A}}_{\text{nr},2}(\tau,n)$. 
We set $m=[\Phi(\cj{n})]$, note that $|m|\les M^2$, and write
\begin{align*}
\ft{\mathcal{A}}_{\text{nr},2}(\tau,n) &= \sum_{|m|\leq M^2} \int_{\R^{2k+1}}  \bigg[1-\ind_{\{ |\tau|\ll M\}}   \prod_{j=1}^{2k+1} \ind_{\{ |\tau_j|\ll M\}} \bigg]   \sum_{n_1,\cdots,n_{2 k+1}}  \mathcal K \bigg(\tau,m+\Phi-[\Phi]+\sum_{j=1}^{2 k+1} \iota_j \tau_j \bigg)   \\
& \hphantom{XXXX} \times c(\cj{n}) \cdot
 {\mathrm T}^{\mathrm b,m,S_2}_{n n_1\cdots n_{2 k+1}}\prod_{j=1}^{2 k+1} \widehat{w_{N_{j}}} (  \tau_j, n_j)^{\iota_i} \, d \tau_1 \cdots  d\tau_{2 k+1}.
\end{align*}
It follows by Minkowski's inequality (to bring inside the $\l^2_n$-norm), \eqref{duhamleker}, the bound
\begin{align}
    \label{Msump2}
 \sup_{(\tau,\tau_1, \ldots, \tau_{2k+1})}   \sum_{|m|\les M^{2}} \bigg\langle m - \tau + \sum_{i=1}^{2 k+1} \iota_i \tau_i \bigg\rangle^{-1} \les\log (1 +  M),
\end{align}
and Cauchy-Schwarz, that 
\begin{align}
\begin{split}
&\| \ft{\mathcal{A}}_{\text{nr},2}(\tau,n)\|_{Y^{0,b'}}  \\
& \les \bigg\| \jb{\tau}^{b'-1} \int_{\R^{2k+1}}\bigg[1-\ind_{\{ |\tau|\ll M\}}   \prod_{j=1}^{2k+1} \ind_{\{ |\tau_j|\ll M\}} \bigg]   \sum_{|m|\les M^2} \frac{1}{\jb{\tau+m+\sum_{j=1}^{2k+1}\iota_j \tau_j}} \\
& \hphantom{XXX} \times \bigg\| \sum_{n_1,\ldots,n_{2k+1}}  c(\cj{n}) {\mathrm T}^{\mathrm b,m,S_2}_{n n_1\cdots n_{2 k+1}}\prod_{j=1}^{2 k+1} |\widehat{w_{N_{j}}} (  \tau_j, n_j)| \bigg\|_{\l^2_n} d\tau_1 \cdots d\tau_{2k+1} \bigg\|_{L^2_{\tau}} \\
& \les \log(1+M)\Big\| \jb{\tau}^{b'-1} \bigg[1-\ind_{\{ |\tau|\ll M\}}   \prod_{j=1}^{2k+1} \ind_{\{ |\tau_j|\ll M\}} \bigg]  \prod_{j=1}^{2k+1}\jb{\tau_j}^{-b} \Big\|_{L^{2}_{\tau \tau_1 \cdots \tau_{2k+1}}}   \\
& \hphantom{XXX} \times \sup_{|m|\les M^2}  \big\|  c(\cj{n}) {\mathrm T}^{\mathrm b,m,S_2}_{n n_1\cdots n_{2 k+1}}\big\|_{\l^2_{nn_1\cdots n_{2k+1}}}  \prod_{j=1}^{2k+1} \| \jb{\tau}^{b} \ft w_{N_j}(\tau,n)\|_{L^2_{\tau} \l^2_n}.
\end{split} \label{J2A2cs}
\end{align}
 Now, we have at least one Fourier time variable that is $\ges M$, so 
 \begin{align}
\bigg\| \jb{\tau}^{b'-1} \bigg[1-\ind_{\{ |\tau|\ll M\}}   \prod_{j=1}^{2k+1} \ind_{\{ |\tau_j|\ll M\}} \bigg]  \prod_{j=1}^{2k+1}\jb{\tau_j}^{-b} \bigg\|_{L^{2}_{\tau \tau_1 \cdots \tau_{2k+1}}} \les M^{-\min(b-\frac 12, \frac 12-b')}. \label{tauints}
\end{align}
It follows \eqref{TypeG}, \eqref{TypeCnorm}, \eqref{TypeD}, that amongst all types, we have $T^{-1}M$-certainly that
 \begin{align}
\prod_{j=1}^{2k+1} \| \jb{\tau}^{b} \ft w_{N_j}(\tau,n)\|_{L^2_{\tau} \l^2_n} \les (2M)^{-1+\g} \prod_{j=2}^{2k+1}N_{j}^{-\frac 12+}, \label{J2fns}
\end{align}
where we used that the frequency $N_1$ corresponds to a Type (D) term since $J=2$. Next, by Lemma~\ref{LEM:ev1} and \eqref{Tbs1}, we have
\begin{align}
\sup_{|m|\les M^2}  \big\|  c(\cj{n}) {\mathrm T}^{\mathrm b,m,S_2}_{n n_1\cdots n_{2 k+1}}\big\|_{\l^2_{nn_1\cdots n_{2k+1}}}  \les M^{\eps} N_{(2)}^{\frac 12} N_{(3)}^{\frac 12+\eps} \prod_{j=4}^{2k+1}N_{(j)} \les M^{\eps+2k \be}. \label{J2count}
\end{align}
Combining \eqref{tauints}, \eqref{J2fns}, \eqref{J2count}, \eqref{J21}, and \eqref{gamma}, we find
\begin{align*}
\| \ft{\mathcal{A}}_{\text{nr},2}(\tau,n)\|_{Y^{0,b'}}  \les (2M)^{-1+\g -(\frac 12-b')+3k\be+\eps} \ll (2M)^{-1+\frac 14 \g},
\end{align*}
for $\g \ll k^{-2}(\frac 12 -b')$.

We move onto the bound for \eqref{J2A1}. We let $\be_2>\be$ small and to be chosen later, and decompose $\ft{\mathcal{A}}_{\text{nr},1}(\tau,n)$ further according to whether $|n-n_1|\gg M^{\be_2}$, denote the ensuing function $\ft{\mathcal{A}}_{\text{nr},1,1}(\tau,n)$ by or $|n-n_1|\les M^{\be_2}$, denoted by $\ft{\mathcal{A}}_{\text{nr},1,2}(\tau,n)$. 
When  $|n-n_1|\les M^{\be_2}$, we note that since $n\neq n_1$, we have $|n-n_1| \geq 1$ and thus it follows from \eqref{ev} and \eqref{J21} that 
\begin{align*}
|\Phi(\cj{n})| \ges |n-n_1||n+n_1|  \ges M  . 
\end{align*}
Moreover, since $|\tau_j|, |\tau| \ll M$ for all $j=1,\ldots,2k+1$, we have
\begin{align*}
\jb{ \tau +\Phi +\sum_{j=1}^{2k+1}\iota_j \tau_j} \sim \jb{\Phi} \ges M.
\end{align*}
Then, using \eqref{duhamleker}, a Cauchy-Schwarz argument as in \eqref{J2A2cs}, Lemma~\ref{LEM:ev1}, and \eqref{J2fns}, we have
\begin{align*}
\| \ft{\mathcal{A}}_{\text{nr},1,2}(\tau,n)\|_{Y^{0,b'}} &\les  M^{-1}\big\| \ind_{\{ |n-n_1|\les M^{\be_2}\}}  c(\cj{n}) \ind_{S_2}\big\|_{\l^2_{nn_1\cdots n_{2k+1}}}  \prod_{j=1}^{2k+1} \|\jb{\tau}^b \ft w(\tau,n)\|_{L^2_{\tau}\l^2_n} \\
& \les  M^{-1}\cdot M^{-1+\g}\cdot M^{\frac 12+\frac{\be_2}{2}}\cdot M^{2k\be}   \\
& \ll M^{-1+\frac{1}{4}\g},
\end{align*}
for $\be_2>0$ sufficiently small.

Now we consider $\ft{\mathcal{A}}_{\text{nr},1,1}(\tau,n)$. Then, by the same argument as above we find
\begin{align*}
\| \ft{\mathcal{A}}_{\text{nr},1,1}(\tau,n)\|_{Y^{0,b'}} &\les  M^{-2+\g+2k\be}\big\| \ind_{\{ |n-n_1|\gg M^{\be_2}\}}  c(\cj{n}) \ind_{S_2}\big\|_{\l^2_{nn_1\cdots n_{2k+1}}}.
\end{align*}
As $\be_2>\be$, we can apply \eqref{cdecay}, to get
\begin{align*}
\| \ft{\mathcal{A}}_{\text{nr},1,1}(\tau,n)\|_{Y^{0,b'}} & \les M^{-2+\g+2k\be} 
\bigg( \sum_{ \substack{n,n_1 \les 2M \\ |n-n_1|\gg M^{\be_2}}}\sum_{\substack{ n_2,\ldots, n_{2k+1} \\ n_j \ll M^{\be}}}  |c(\cj{n})|^2  \bigg)^{\frac{1}{2}} \\
&\les M^{-2+\g+10k\be} 
\bigg( \sum_{ \substack{n,n_1 \les 2M \\ |n-n_1|\gg M^{\be_2}}} \frac{1}{\jb{n-n_{(1)}}^2}  \bigg)^{\frac{1}{2}}\\
&\les M^{-2+\g+10k\be} \cdot M^{\frac12 -\frac{\be_2}{2}}   \\ 
& \les M^{-1+\frac{1}{4}\g}.
\end{align*}
This completes the proof for the contribution \eqref{J2neq} and also that of the case $J=2$. 
\end{proof}

\subsection{$J\in \{3,6\}$ cases} 

This case requires a combination of all ideas from the previous subsections. 
Additionally, a meshing argument is needed when the highest frequency input is of Type (C). 
This technique, first introduced for the study of probabilistic dispersive equations in \cite{DNY2} and later used in \cite{DNY3,BDNY,LW24}, 
involves discretizing the $\tau$ variable to keep the probability of the exceptional set under control. 
For completeness, we formulate a new, explicit version of this argument in Subsection \ref{SUB:case1} for completeness.

\subsubsection{Preparations}

Due to the similarity between the cases $J = 3$ and $J = 6$, 
we will mainly focus on the case $J = 6$ in what follows, 
and indicate the necessary modifications for the case $J = 3$ when appropriate.

We first make some reductions for the $J=6$ case, where $n\sim M$.
In the following, we may assume that $N_{(2)} \ll M$, since otherwise, we may apply the same argument as in the case $J=1$ (Lemma~\ref{LEM:J1}). Moreover, we may assume that $N_{(1)}=N_1$ and that $N_{1} \ges M$. Otherwise, we either have $N_{(1)}=N_{\text{even}}$ or $N_{(1)}\ll M$.
If $N_{(1)}=N_{\text{even}}$,
 then we apply the argument as in Case $J=4$ (Lemma \ref{LEM:J1}).

 If $N_{(1)} \ll M$, we use mostly similar arguments as previous cases to handle this part:
 
 \begin{lemma}\label{LEM:J6nbig}
 Assume that $N_{(1)} \ll M$. Then, we have
 \begin{align}
\| \mathcal{A}_{2M}^{(6)}(w_{N_1}^{\dag}, w_{N_2}^{\dag}, \ldots, w^{\dag}_{N_{2k+1}})\|_{Y^{0,b}}  &\les (2M)^{-1+\frac{1}{4}\g} \label{J60}
\end{align}
for $\{w_{N_j}\}_{j=1}^{2k+1}$ of any combination of \textup{Types} \textup{(G), (C)} or \textup{(D)}.
 \end{lemma}
 \begin{proof}
 As $N_{(1)} \ll M$, we have $n\sim M \gg n_1$ and thus $|\Phi(\cj{n})| \ges (2M)^{2}$. 
Similar to the proof of Lemma~\ref{LEM:J5}, by duality and dyadic decompositions, we reduce to estimating
\begin{align*}
\sum_{ \substack{K_0, \ldots ,K_{2k+1} \\ K_{\max} \geq CM^2}} I_{6}(\cj{N}, \cj{K}) :=\sum_{ \substack{K_0, \ldots ,K_{2k+1} \\ K_{\max} \geq CM^2}} 
\iint_{\R\times \mathbb{D}} \P_{ 2M} \Q_{K_0}v  \prod_{j=1}^{2k+1} (\Q_{K_j}S(t)w_{N_j})^{\iota_j} dxdt,
\end{align*}
for $v\in X^{0,b'}$ with $b'<\frac 12$.

If $K_{\max}= K_{j_{\ast}}$ for some $j_{\ast}\in \{1,\ldots, 2k+1\}$, then by H\"{o}lder's inequality, \eqref{TypeD}, Lemma~\ref{LEM:L4Strich},  \eqref{Dphys2}, Lemma~\ref{LEM:TypeC}, and \eqref{C_estimate2}, we have
\begin{align*}
|I_{6}(\cj{N}, \cj{K})| &\les \|\P_{ 2M} \Q_{K_0}v \|_{L^{4}_{t,x}} \|\Q_{K_{j_{\ast}}} S(t)w_{N_{j_{\ast}}} \|_{L^{2}_{t,x}} \prod_{ \substack{j=1 \\ j\neq j_{\ast}}  }^{2k+1} \|\Q_{K_j} S(t)w_{N_j}\|_{L^{8k}_{t,x}} \\
& \les M^{\eps}K_{\max}^{-\eps} K_{\max}^{-(b-\eps)} N_{j_{\ast}}^{-\frac 12+ } \prod_{ \substack{j=1 \\ j\neq j_{\ast}}  }^{2k+1} N_{j}^{\g} \\
& \les K_{\max}^{-\eps}(2M)^{-2b+5\eps+2k\g} \\
& \ll K_{\max}^{-\eps} (2M)^{-1+\frac{1}{4}\g},
\end{align*}
 since $\g \ll b-\frac 12$ and for any $\eps>0$ sufficiently small.
 
 Now we consider the case when $K_{\max}=K_0$. 
 Suppose first that $N_{(1)} \ges M^{\be_3}$ for some $0<\be_3<1$ to be chosen later. 
 Then, we have
\begin{align*}
&|I_{6}(\cj{N}, \cj{K})|  \\
&\les \|\P_{ 2M} \Q_{K_0}v \|_{L^{2}_{t,x}} \|\Q_{K_{j_{(1)}}} S(t)w_{N_{j_{(1)}}} \|_{L^{4}_{t,x}}  \|\Q_{K_{j_{(2)}}} S(t)w_{N_{j_{(2)}}} \|_{L^{4-}_{t,x}} \prod_{ \l=3  }^{2k+1} \|\Q_{K_{j_{(\l)}}} S(t)w_{N_{j_{(\l)}}}\|_{L^{8k}_{t,x}} \\
& \les K_{\max}^{-\eps} K_{\max}^{-(b'-\eps)}  N_{(1)}^{-\frac 12+}N_{(2)}^{-\frac 12+} \prod_{ j=3  }^{2k+1} N_{(j)}^{\g} \\
& \les K_{\max}^{-\eps}(2M)^{-2b'+4\eps - \frac{\be_3}{2}} \\
& \ll K_{\max}^{-\eps} (2M)^{-1+\frac{1}{4}\g},
\end{align*}
 provided that we choose $\be_3>2(1-2b')$ and $\eps>0$ sufficiently small. 
 
 Now we assume that $N_{(1)}\ll M^{\be_3}$. Define the function
 \begin{align*}
V :  =(-\Dl)^{-1} v = \sum_{n\in \N} \ft v(n) \ld_{n}^{-2} e_{n},
\end{align*}
so that $-\Dl V = v$. Moreover, since $K_{\max}=K_0 \ges M^2$, we have
\begin{align}
\| \P_{ 2M} \Q_{K_0}V\|_{L^{2}_{t,x}} \les K_{\max}^{-\eps} M^{-2b'-2+\eps} \label{dualV}
\end{align} 
for any $\eps>0$.
Then, by integration by parts, 
\begin{align}
&\iint_{\R\times \mathbb{D}} \P_{ 2M} \Q_{K_0}v  \prod_{j=1}^{2k+1} (\Q_{K_j}S(t)w_{N_j})^{\iota_j} dxdt \notag \\
& = \iint_{\R\times \mathbb{D}} \P_{ 2M} \Q_{K_0}V \cdot   (-\Dl) \bigg(  \prod_{j=1}^{2k+1} (\Q_{K_j}S(t)w_{N_j})^{\iota_j} \bigg) dxdt  \label{dualV1}
\end{align}
We now estimate \eqref{dualV1} by placing $\P_{ 2M} \Q_{K_0}V$ into $L^2_{t,x}$. 
For the other functions, by the product rule, we have essentially two cases: (i) $(-\Dl)$ falls completely on a single $S(t)w_{N_j}$, or (ii) we have one derivative (through $\nb$) on two different $S(t)w_{N_j}$ terms. 
In either case, we handle the Type (D) terms by using the Strichartz estimates (Lemma~\ref{LEM:Dstrich}) followed by \eqref{L2derivs} or \eqref{C_estimate3} and \eqref{C_estimate4} for any Type (C) or (G) terms, and \eqref{dualV}. 
We then obtain
\begin{align*}
|\eqref{dualV1}| & \les K_{\max}^{-\eps}  M^{-2b'-2+\eps}N_{(1)}^{2+2k\g} \ll M^{-1+\frac{1}{4}\g} M^{-1-2b'-\frac{1}{4}\g +\eps +2(1+k\g)\be_3} \ll M^{-1+\frac{1}{4}\g},
\end{align*}
as $k\g \ll 1$ so that we only need to have $\be_3 < \frac{1}{2}+b'$ which is not restrictive to our previous condition $2(1-2b')<\be_3$ since $b'<\frac 12$ but close to $\frac 12$, and we can take $\eps>0$ sufficiently small.
 \end{proof}

In view of these reductions, when $J=6$, we may assume that:
 \begin{align}
N_{(1)}=N_1 \sim M, \quad N_{(2)}\ll M, \quad n\neq n_1. \label{J35a}
\end{align}
 If $w_{N_1}$ is of Type (D), then we argue as in the Case $J=2$ (Lemma~\ref{LEM:J2}). Note that \eqref{J35a} is already satisfied when $J=3$.
 Therefore, for $J\in \{3,6\}$, we may assume that \eqref{J35a} holds and $w_{N_1}$ is of Type (C).

To prove \eqref{Ab'},
we rewrite $\mathcal{A}_{2M}^{(J)}( w_{N_1}, \cdots, w_{N_{2k+1}})$ with the tensor notation from  \eqref{basetensor0} and \eqref{baseS}. By \eqref{eqn:zn2}, \eqref{duhamelform}, and introducing $m=[\Phi]$, 
we have
\begin{equation}
\label{formulaK}
\begin{split}
\F_{t,x} ( & \mathcal{A}_{2M}^{(J)} ( w_{N_1}, \cdots, w_{N_{2k+1}}) (\tau,n) \\
& = \sum_m  \int_{\R^{2k+1}}\sum_{n_1,\cdots,n_{2 k+1}}  \mathcal K \bigg(\tau,m +\Phi-[\Phi]+\sum_{j=1}^{2 k+1} \iota_j \tau_j \bigg)  {\mathrm T}^{\mathrm b, m,S_J}_{n n_1\cdots n_{2 k+1}}  \\
& \hphantom{XXXX} \times c(n,n_1,n_2,\ldots,n_{2k+1}) \cdot
\prod_{j=1}^{2 k+1} \widehat{w_{N_{j}}} (  \tau_j, n_j)^{\iota_i} \, d \tau_1 \cdots  d\tau_{2 k+1}
\end{split}
\end{equation} 
where $\{S_{J}\}$ are defined as the following subsets of $(n, n_1,\ldots, n_{2k+1})\in \N^{2k+2}$:
\begin{align*}
S_3& =  \{   n_j \in E_{N_j}\backslash E_{N_j/2}  \,\, \&\,\,  N_{j}<L_{M}, \, \forall \, j\in \{2,\ldots, 2k+1\}, M < \ld_{n_1}\le 2M, \ld_n \le 2M, n\neq n_1\}, \\
S_6 & =  \{  n_j \in E_{N_j}\backslash E_{N_j/2}  \,\, \&\,\,  N_{j}\le M, \, \forall \, j\in \{1,\ldots, 2k+1\}, M < \ld_{n}\le 2M, N_{\text{even}}=N_{2}, \,\, N_{\text{odd}}=N_{1}\},
\end{align*}
and where we used the notation \eqref{baseS}, $c(\bar n)$ is given in \eqref{c_n}, and $\Phi$ is as in \eqref{phase}. As $N_j \les 2M$ for all $j\in \{1,\ldots, 2k+1\}$ on the support of each $\mathcal{A}_{2M}^{(J)}$ and $N=2M$, we have $|m|\les M^{2}$.

We note that $\Phi$ is generally not an integer and thus it is necessary to keep track of the dependency of the kernel $\mathcal{K}$ on $(n,n_1,n_2,\ldots, n_{2k+1})$ through the difference $\Phi-[\Phi]$. 
In view of \eqref{J35a}, we have $|m|\les M^2$.

In the following, we need to use meshing arguments, as in \cite[Proof of Lemma 4.2]{DNY2}, whose goal is to ensure that all exceptional $T^{-1}M$-certain sets that need to be removed \textit{do not} depend on the time frequency variables $\{\tau_j\}$. 
 We define the large box in the time frequency variables \begin{align}
Q : = \big\{ (\tau, \tau_1,\ldots, \tau_{2k+1}) \in \R^{2k+2} \,: \max_{j=1,\ldots, 2k+1}( \jb{\tau_j},\jb{\tau}) \ll K_M\big\}, \label{Qbox}
\end{align}
where $K_M:= M^{\ta^{-2}k^3}$ and $\ta=\max(2b-1, 1-2b')$.
We split the large box into smaller disjoint boxes $\{Q_{\vec \l}\}_{ \vec \l}$ of width $K_M^{-1}$, where 
\[
\vec \l = (\l, \l_1, \cdots, \l_{2k+1}) \in (\Z)^{2k+1}. 
\]
We write 
$$Q_{\vec \l}=  \pi_0 Q_{\vec \l} \times \pi_{0}^{\perp} Q_{\vec \l}= \pi_0 Q_{\vec \l} \times \pi_{1}Q_{\vec \l} \times \cdots \times  \pi_{2k+1}Q_{\vec \l} = Q_\l \times \prod_{j=1}^{2k+1} Q_{\l_j},$$
where $\pi_0$ is the projection onto the $\tau$ variable with $\tau \in Q_\l$ and $\pi_j$ is the projection onto the $\tau_j$ variable with $\tau_j \in Q_{\l_j}$, $j=1,\ldots, 2k+1$. 
We decompose the operator $\mathcal{A}^{(J)}_{2M}$ according to whether the integration is restricted to $Q$ or its complement:
\begin{align*}
\mathcal{A}^{(J)}_{2M}=\mathcal{A}^{(J)}_{2M, Q}+\mathcal{A}^{(J)}_{2M,Q^{c}}
\end{align*}
Finally, it remains to consider the contribution from $\mathcal{A}^{(J)}_{2M,Q^{c}}$ with any combination of inputs of Type (C) or (D). Here we may gain from the difference $b-b'>0$. We first note that for any $w_{j}$  of Type (C) or Type (G),  Lemma~\ref{LEM:growth} implies 
\begin{align*}
\| \jb{\tau}^{b} \ft w_{j}(\tau, n)\|_{L^{2}_{\tau}\l^{2}_{n}} \les T^{-\frac{\eps}{2}} N_{j}^{-\frac{1}{2}} (2M)^{\eps}
\end{align*}
outside of a set of measure $e^{-(T^{-1}(2M)^{2})^{\theta}}\sim e^{-(T^{-1}M^{2})^{\theta
}}$. Thus, by Cauchy-Schwarz, \eqref{duhamleker}, Lemma~\ref{LEM:ev1}, \eqref{TypeD}, and \eqref{J35a},
\begin{align*}
\|&\mathcal{A}^{(J)}_{2M,Q^{c}}( w_{N_1},\ldots, w_{N_{2k+1}})\|_{Y^{0,b'}(Q^{c}\times \mathbb{D})} \\
& \les  M^{O(k)} \bigg( \int_{Q^{c}} \jb{\tau}^{-2(1-b')} \prod_{j=1}^{2k+1}\jb{\tau_j}^{-2b}d\tau_{1 \cdots 2k+1}\bigg)^{\frac 12} \prod_{j=1}^{2k+1} \| \jb{\tau}^{b} \ft w_{N_j}(\tau,n)\|_{L^{2}_{\tau}\l^2_n} \\
& \les M^{O(k)} K_{M}^{-\ta} \ll (2M)^{-1}.
\end{align*}

We now focus on the contribution from $\mathcal{A}^{(J)}_{2M, Q}$.
Restriction to the box $Q$ (as defined in \eqref{Qbox}) will now allow us to apply a meshing argument.
We consider the average $(\ft{h}_{\text{avg}})_{n_j}^{L_{j}, R_{j}}(\tau_j)$ which is defined by taking the average of
 $\ft{h}_{n_j}^{L_{j}, R_{j}}(\tau_j)$ over $\pi_{j}Q_{\vec \l} = O_{\l_j} \ni \tau_j$; that is, we define 
 \begin{align}\label{havg}
(\ft{h}_{\text{avg}})_{n_j}^{L_{j}, R_{j}}(\tau_j) :  =\frac{1}{|Q_{\l_j}|} \int_{Q_{\l_j}}\ft{h}_{n_j}^{L_{j}, R_{j}}(\tau) d\tau
\end{align}
for $\tau_j \in Q_{\l_j}$.
The (usual) Poincar\'e inequality implies
 \begin{align}
\big\| \ft h^{L_{j},R_{j}}_{n_j}(\tau_j) - (\ft{h}_{\text{avg}})_{n_j}^{L_{j}, R_{j}}(\tau_j) \big\|_{L^{2}(Q_{\l_j})} \les |Q_{\l_j}| \Big\| \dd_{\tau_j} \ft h^{L_{j},R_{j}}_{n_j}(\tau_j) \Big\|_{L^{2}(Q_{\l_j})},
\label{Poincare}
\end{align}
where the right hand side is controlled on $\textsf{Local}(M)$ in view of Plancherel, that $h_{n_j}^{L_j, R_j}$ is compactly supported in time, and \eqref{TypeCnorm}. The large gain from the factor $ |Q_{\l_j}| \sim K_M^{-1}$ 
allows us to then replace all occurrences of $\ft{h}_{n_j}^{L_{j}, R_{j}}(\tau_j)$ by their averaged versions $(\ft{h}_{\text{avg}})_{n_j}^{L_{j}, R_{j}}(\tau_j)$ which are constant over $Q_{\l_j}$. 
Note that in the following, we need to carefully deal with the summation over $\vec \l \in (\Z)^{2k+2}$ in order to not incur a loss in the number of choices of $\vec \l$.

Similarly, for fixed $m\in \R$ and $(n,n_1,\ldots, n_{2k+1})\in \N^{2k+2}$, we define an averaged version of the weighted kernel
\begin{align*}
\wt{\mathcal{K}}\bigg( \tau , m+ \Phi-[ \Phi] + \sum_{j=1}^{2k+1}\iota_j \tau_j \bigg):=\jb{\tau} \mathcal{K}\bigg( \tau , m+ \Phi-[ \Phi] + \sum_{j=1}^{2k+1}\iota_j \tau_j \bigg)
\end{align*}
For $(\tau, \tau_1, \ldots, \tau_{2k+1})\in Q_{\vec \l}$, we define 
\begin{align}
\wt{\mathcal{K}}_{\vec \l}(m,\cj{n}) : = \frac{1}{|Q_{\vec \l}|} \int_{Q_{\vec \l}}  \wt{\mathcal{K}}\bigg( \tau , m+ \Phi-[ \Phi] + \sum_{j=1}^{2k+1}\iota_j s_j \bigg) ds ds_{1, \ldots, 2k+1}. \label{avgker}
\end{align}
As in \eqref{Poincare} and using Lemma~\ref{PROP:duhamel} (the derivative estimates), in what follows, we may replace the weighted kernel $\wt{\mathcal{K}}$ by its averaged version $\wt{\mathcal{K}}_{\vec \l}$. In particular, since $\Phi-[\Phi]=O(1)$, Lemma~\ref{PROP:duhamel} implies 
\begin{align}
|\wt{\mathcal{K}}_{\vec\l}(m,\cj{n})| \les \frac{1}{|Q_{\vec\l}|} \int_{Q_{\vec\l}} \frac{1}{\jb{ m + \sum_{j=1}^{2k+1}\iota_j s_j   }} ds ds_{1, \ldots, 2k+1}. \label{Kavgbd}
\end{align}
Note that the right hand side of \eqref{Kavgbd} only depends on $m$ (and $\vec\l$) and no longer on $(n,n_1,\ldots, n_{2k+1})$. This step, which does not appear in \cite{DNY2}, is necessary in our setting as $\Phi$ is, in general, non-integer and thus the kernel still depends on the spatial frequencies. 

We note that, in the following, while some of the exceptional sets we remove depend upon $\vec\l$, there are only $O(M^{k^3})$ many values for $\vec\l$.

By \eqref{formulaK}, Minkowski's inequality, and Cauchy-Schwarz,
we have
\begin{align}
\begin{split}
&\| \mathcal{A}^{(J)}_{2M}( w_{N_1},\ldots, w_{N_{2k+1}})\|_{Y^{0,b'}(Q\times \mathbb{D})}^{2} \\
& =\sum_{\vec\l \in \Z^{2k+2}} \| \mathcal{A}^{(J)}_{2M}( w_{N_1},\ldots, w_{N_{2k+1}})\|_{Y^{0,b'}(Q_{\l}\times \mathbb{D})}^{2}\\
& \les \sum_{\vec\l \in \Z^{2k+2}} \int_{Q_{\l}} \jb{\tau}^{2(b'-1)} \bigg\|  \sum_{|m|\les M^2} \int_{Q_{\l_1}} \cdots \int_{ Q_{\l_{2k+1}}} \sum_{n_1,\ldots,n_{2k+1}} \wt{\mathcal{K}}_{\vec \l}(m,\cj{n}) c(\cj{n})  \\
&\hphantom{XXXXXXXXXXXXXXXX}\times {\mathrm T}^{\mathrm b, m,S_{J}}_{nn_1\cdots n_{2 k+1}} \prod_{j=1}^{2k+1} \ft {w_{N_j}}(\tau_j, n_j)^{\iota_j} d\tau_{1 \cdots 2k+1} \bigg\|_{\l^{2}_{n}}^{2} d\tau.
\end{split}
\label{redu1}
\end{align}
In the following,
we write
\begin{align}
\label{Xnt}
\begin{split}
\mathcal X_n^{(J)}  & = \sum_{|m|\les M^{2}} \int_{\pi_0^{\perp} Q_{\vec \l}}   
\sum_{n_1,\cdots,n_{2 k+1}}  \wt{\mathcal{K}}_{\vec \l}(m,\cj{n})
{\mathrm T}^{\mathrm b, m,S_{J}}_{nn_1\cdots n_{2 k+1}}  \cdot c(\bar n) \\
&\hphantom{XXXXXXXXXXXXXXX} \times \prod_{j=1}^{2 k+1}   \widehat{w_{N_j}}^{\iota_j} ( \tau_j, n_j) d\tau_{1 \cdots 2k+1}
 \end{split}
\end{align}

\noi
where $\pi_0^{\perp} Q_{\vec \l} = \prod_{j=1}^{2k+2} Q_{\l_j}$.
We note that \eqref{Xnt} is the core of \eqref{redu1}.
We prove \eqref{Ab'} case by case, namely each of the $2^{2 k+1}$ many possible cases for the types of $(w_{N_1}, w_{N_2}, w_{N_3}, \cdots, w_{N_{2 k+1}})$ being 
(1): (C, C, C, $\cdots$, C),
(2): (D, C, C, $\cdots$, C),
(3): (D, D, C, $\cdots$, C),
$\cdots$,
and ($2^{2 k}$): (D, D, D, $\cdots$, D).

\subsubsection{Case 1: All of Type (D)}
\label{SUB:case1}

As we are on $Q_{\vec \l}$, we define
\begin{align*}
A_{0}^{(J)}(\cj{n}) : = \sum_{\ld_{n_1}\sim 2M} \wt{\mathcal{K}}_{\vec \l}(m,\cj{n}){\mathrm T}^{\mathrm b, m,S_{J}}_{nn_1\cdots n_{2 k+1}} c(n,n_1,n_2,\ldots,n_{2k+1})  (\ft h_{\text{avg}})_{n_1}^{2M,L_{1}}\frac{g_{n_1}}{\ld_{n_1}} ,
\end{align*}
where $\wt{\mathcal{K}}_{\vec \l}(m,\cj{n})$ is given in \eqref{avgker}, so that
\begin{align*}
\mathcal{X}_{n}^{J} = \sum_{|m|\les M^2} \int_{\pi_0^{\perp} Q_{\vec \l}} \sum_{n_2, \ldots, n_{2k+1}} \prod_{j=2}^{2k+1} \ft{w}_{N_j}^{\iota_j}(\tau_j,n_j) A_{0}^{(J)} (\cj{n}) d\tau_{1 \cdots 2k+1}
\end{align*}
Then, by Minkowski's inequality and Cauchy-Schwarz,
\begin{align*}
\eqref{redu1} & \les \sup_{\l} \sum_{\l_1,\cdots,\l_{2k+1}}  \bigg( \sum_{|m|\les M^{2}} \int_{\pi_{0}^{\perp}Q_{\vec \l}}  \bigg\|  \sum_{n_2, \ldots, n_{2k+1}} \prod_{j=2}^{2k+1}  \ft{w}_{N_j}^{\iota_j}(\tau_j,n_j) A_{0}^{(J)} (\cj{n})\bigg\|_{\l^{2}_{n}}d\tau_{1 \cdots 2k+1} \bigg)^{2} \\
& \les  \sup_{\l} \sum_{\l_j}  \bigg(\sum_{|m|\les M^{2}} \int_{\pi_{0}^{\perp}Q_{\vec \l}} \bigg( \prod_{j=2}^{2k+1} \| \ft{w}_{N_j}^{\iota_j}(\tau_j,\cdot )\|_{\l^2_n}  \bigg)   \| A_{0}^{J}(\cj{n}) \|_{\l^2_{n,n_2,\ldots, n_{2k+1}}}   d\tau_{1 \cdots 2k+1}  \bigg)^2  \\
& \les \sup_{\l} \sum_{\l_j} \bigg( \prod_{j=2}^{2k+1} \| \jb{\tau_j}^{b} \wt w_{N_j}(\tau_j,n)\|_{\l^2(\N; L^2(Q_{\l_j}))}^2 \bigg) \bigg( \sum_{|m|\les M^2} \int_{\pi_1 Q_{\l}} \| A_{0}^{J}(\cj{n}) \|_{\l^2_{n,n_2,\ldots, n_{2k+1}}} d\tau_1\bigg)^{2} \\
& \les  \bigg( \prod_{j=2}^{2k+1} N_{j}^{-1+\g}\bigg)^{2} \sup_{\l} \sup_{\l_j, j \neq 1} \sum_{\l_1} \bigg( \sum_{|m|\les M^2} \int_{Q_{\l_1}} \| A_{0}^{J}(\cj{n}) \|_{\l^2_{n,n_2,\ldots, n_{2k+1}}} d\tau_1\bigg)^{2}.
\end{align*} 

Now, by Lemma~\ref{LEM:condchaos}, \eqref{havg}, and \eqref{Kavgbd}, we have $T^{-1}M$-certainly\footnote{Note that the exceptional set is of size $O\big(e^{-\left(T^{-1} M\right)^{2 \theta}}\big)$ for some $\theta \ll 1$. 
However, the exceptional set may depend on $\l_1$. 
As there are only $K_M^2 = M^{2 \theta^{-2} k^3}$ many such values of $\l_1$, the union of all these exceptional sets is still of size at most $O\big(e^{-\left(T^{-1} M\right)^{\theta}}\big)$. 
This is the main reason for the meshing argument, which discretizes the interval $\tau_1 \in [-K_M, K_M]$ into $\l_1 \in \Z \cap [-K_M^2,K_M^2]$.} 
\begin{align*}
\int_{Q_{\l_1}} & \| A_{0}^{J}(\cj{n}) \|_{\l^2_{n,n_2,\ldots, n_{2k+1}}} d\tau_1 = |Q_{\l_1}| \| A_{0}^{J}(\cj{n}) \|_{\l^2_{n,n_2,\ldots, n_{2k+1}}} \\
 &\les M^{\ta} |Q_{\l_1}| \bigg( \sum_{n,n_1, \ldots, n_{2k+1}} |\wt{\mathcal{K}}_{\vec \l}(m,\cj{n})|^2 {\mathrm T}^{\mathrm b, m,S_{J}}_{nn_1\cdots n_{2 k+1}} |c(\bar n)|^2 |(\ft h_{\text{avg}})_{n_1}^{2M,L_{1}}|^2 \ld_{n_1}^{-2} \bigg)^{\frac12}\\
 & \les M^{\ta-1} \| c(\cj{n})\|_{\l^{\infty}_{\cj{n}}} \bigg( \sum_{\cj{n}} |\wt{\mathcal{K}}_{\vec \l}(m,\cj{n})|^2 {\mathrm T}^{\mathrm b, m,S_{J}}_{nn_1\cdots n_{2 k+1}}  \bigg| \int_{Q_{\l_1}} \ft h^{2M,L_{1}}_{n_1} (\tau_1) d\tau_1 \bigg|^2 \bigg)^{\frac12} \\
 & \les M^{\ta-1} \| c(\cj{n})\|_{\l^{\infty}_{\cj{n}}} \| \jb{\tau_1}^{-b} \|_{L^2_{\tau_1} (Q_{\l_1})} \bigg( \sum_{\cj{n}} |\wt{\mathcal{K}}_{\vec \l}(m,\cj{n})|^2 {\mathrm T}^{\mathrm b, m,S_{J}}_{nn_1\cdots n_{2 k+1}}  \| \jb{\tau_1}^b \ft h^{2M,L_{1}}_{n_1} \|_{L^2_{\tau_1} (Q_{\l_1})}^2 \bigg)^{\frac12} \\
& \les M^{\ta-1} \| c(\cj{n})\|_{\l^{\infty}_{\cj{n}}} \| \jb{\tau_1}^{-b} \|_{L^2_{\tau_1} (Q_{\l_1})} \bigg( \sum_{\cj{n}} {\mathrm T}^{\mathrm b, m,S_{J}}_{nn_1\cdots n_{2 k+1}}  \| \jb{\tau_1}^b \ft h^{2M,L_{1}}_{n_1} \|_{L^2_{\tau_1} (Q_{\l_1})}^2 \bigg)^{\frac12}\\
 & \hphantom{XX} \times 
 \frac{1}{|Q_{\vec\l}|} \int_{Q_{\vec\l}} \frac{1}{\jb{ m + \sum_{j=1}^{2k+1}\iota_j s_j   }} ds ds_{1, \ldots, 2k+1},
\end{align*}
where we used the fact that $(\ft h_{\text{avg}})_{n_1}^{2M,L_{1}}$ and thus $A_{0}^{J}(\cj{n})$ are constant on the interval $Q_{\l_1}$.
We note that 
\[
\sum_{|m| \les M^2} \frac{1}{|Q_{\vec\l}|} \int_{Q_{\vec\l}} \frac{1}{\jb{ m + \sum_{j=1}^{2k+1}\iota_j s_j   }} ds ds_{1, \ldots, 2k+1} \les \log (1+ M),
\]
holds uniformly in $\vec \l$.
Therefore, by Cauchy-Schwarz, \eqref{Kavgbd}, Lemma~\ref{LEM:ev1}, and \eqref{Tbs1},
\begin{align*}
\sum_{|m|\les M^2}&  \sum_{\l_1} \int_{Q_{\l_1}} \| A_{0}^{J}(\cj{n}) \|_{\l^2_{n,n_2,\ldots, n_{2k+1}}} d\tau_1 \\
& \les  M^{-1+2\ta}  \| c(\cj{n})\|_{\l^{\infty}_{\cj{n}}} \sum_{\l_1} \| \jb{\tau}^{-b} \|_{L^2_\tau (Q_{\l_1})}    
\bigg(\sum_{\cj{n}} {\mathrm T}^{\mathrm b, m,S_{J}}_{nn_1\cdots n_{2 k+1}}  \| \jb{\tau_1}^b \ft h^{2M,L_{1}}_{n_1} \|_{L^2(Q_{\l_1})}^2  \bigg)^{\frac 12} \\
&\les M^{-1+2\ta}  \| c(\cj{n})\|_{\l^{\infty}_{\cj{n}}} \| \jb{\tau}^{-b} \|_{L^2_\tau (\R)} 
\bigg( \sum_{\l_1} \sum_{\cj{n}} 
{\mathrm T}^{\mathrm b, m,S_{J}}_{nn_1\cdots n_{2 k+1}} \| \jb{\tau_1}^b \ft h^{2M,L_{1}}_{n_1} \|_{L^2(Q_{\l_1})}^2  \bigg)^{\frac 12} \\
&\les M^{-1+2\ta}  \| c(\cj{n})\|_{\l^{\infty}_{\cj{n}}} 
\bigg( \sum_{\cj{n}}  {\mathrm T}^{\mathrm b, m,S_{J}}_{nn_1\cdots n_{2 k+1}} \| \jb{\tau_1}^b \ft h^{2M,L_{1}}_{n_1} \|_{L^2(\R)}^2  \bigg)^{\frac 12}   \\
& \les M^{-1+2\ta}  \| c(\cj{n})\|_{\l^{\infty}_{\cj{n}}} \sup_{n_1 \in E_{2M} \backslash E_{M}} \| \jb{\tau_1}^b \ft h^{2M,L_{1}}_{n_1} \|_{L^2(\R)} 
\| {\mathrm T}^{\mathrm b, m,S_{J}}_{nn_1\cdots n_{2 k+1}}\|_{\l^{2}_{nn_1n_2 \ldots n_{2k+1}}} \\
  & \les M^{-1+2\ta+3\eps } N_{(2)}^{\frac 12} N_{(3)}^{\frac 12+\eps} \prod_{j=4}^{2k+1} N_{(j)}^{1-\eps},
\end{align*}
where we used \eqref{TypeCnorm} in the last step.
Here $N_{(j)}$ is the $j$-th largest frequency among $N_j$ for $j = 1, \cdots, 2k+1$.
Therefore, 
\begin{align*}
\eqref{redu1}^{\frac 12} &\les M^{-1+2\ta+3\eps} N_{(2)}^{-\frac 12+\g} N_{(3)}^{-\frac 12+\g} \prod_{j=4}^{2k+1} N_{(j)}^{\g-\eps} \\
& \les M^{-1+2\ta+3\eps}N_{(2)}^{-\frac 12+\g} N_{(3)}^{-\frac 12 +\g + (2k-3)(\g-\eps)} \\
& \les (2M)^{-1+\frac{1}{4}\g} N_{(2)}^{-\frac 14},
\end{align*}
where we used \eqref{Msump2}.
This completes the proof of \eqref{J60} for the case $(D,\ldots, D)$.

\subsubsection{Case 2: At least one Type (C) or Type (G) }

By relabelling $(n_2, \ldots, n_{2k+1})$, we may assume that there exists an integer $k_0\geq 2$ such for $1\leq j \leq k_0$, $w_{N_j}$ is of Type (C) or Type (G), and for $k_0+ 1 \leq j \leq 2k+1$, $w_{N_j}$ is of Type (D). We apply a similar argument as in Case 1. As in \eqref{redu1}, we only need to control $\mathcal{X}_{n}^{(J)}$ in \eqref{Xnt} in $\l^2_{n}$, uniformly in $\l$.
We write 
\begin{align*}
\mathcal{X}_{n}^{J} & = \sum_{|m|\les M^2} \int_{\pi_{0}^{\perp}Q_{\vec \l}} \sum_{n_{k_0+1}, \ldots, n_{2k+1}} \prod_{j=k_0+1}^{2k+1} \ft{w}_{N_j}^{\iota_j}(\tau_j,n_j)  \\
&\hphantom{XXXXXXXXXXXX} \times \sum_{n_{1} ,\ldots, n_{k_0}} \wt{\mathcal{K}}_{\vec \l}(m,\cj{n}){\mathrm T}^{\mathrm b, m,S_{J}}_{nn_1\cdots n_{2 k+1}}  \cdot c(\bar n) \prod_{j=1}^{k_0} (\ft h_{\text{avg}}^{N_{j}, L_{j}})_{n_j}^{\iota_j} \frac{g_{n_j}^{\iota_j}}{\ld_{n_j}} d\tau_{1\cdots 2k+1}. 
\end{align*}
Then, by Minkowski's inequality and Cauchy-Schwarz,  we have
\begin{align}
\begin{split}
\| \chi_{n}^{J}\|_{\l^2_{n}} 
& \leq \sum_{|m|\les M^2} \int_{\pi_{0}^{\perp}Q_{\vec \l}} d\tau_{1\cdots 2k+1} \prod_{j=k_0+1}^{2k+1} \| \ft w_{N_j}(\tau_j, n_j)\|_{\l^{2}_{n_j}}  \\
&\hphantom{XX} \times \bigg\|  \sum_{n_{1} ,\ldots, n_{k_0}} \wt{\mathcal{K}}_{\vec \l}(m,\cj{n}){\mathrm T}^{\mathrm b, m,S_{J}}_{nn_1\cdots n_{2 k+1}}  c(\bar n) \prod_{j=1}^{k_0} (\ft h_{\text{avg}}^{N_{j}, L_{j}})_{n_j}^{\iota_j} \frac{g_{n_j}^{\iota_j}}{\ld_{n_j}} \bigg\|_{\l^{2}_{n n_{k_0+1} \cdots n_{2k+1}}}  \\
& \les  \bigg( \prod_{j=k_0+1}^{2k+1}N_{j}^{-1+\g}\bigg) \sum_{|m|\les M^2} \int_{Q_{\l_1} \times \cdots Q_{\l_{k_0}}} d\tau_{1\cdots k_0}  \\
&\hphantom{XX} \bigg\|  \sum_{n_{1} ,\ldots, n_{k_0}} \wt{\mathcal{K}}_{\vec \l}(m,\cj{n}){\mathrm T}^{\mathrm b, m,S_{J}}_{nn_1\cdots n_{2 k+1}}  c(\bar n) \prod_{j=1}^{k_0} (\ft h_{\text{avg}}^{N_{j}, L_{j}})_{n_j}^{\iota_j} \frac{g_{n_j}^{\iota_j}}{\ld_{n_j}} \bigg\|_{\l^{2}_{n n_{k_0+1} \cdots n_{2k+1}}}. 
\end{split}
\label{zpair0}
\end{align}
Recall our standing assumption \eqref{J35a} which implies that $n_1 \gg n_{(2)}$, so that $n_1$ is an unpaired frequency.
By Lemma~\ref{LEM:condchaos}, we have $T^{-1}M$-certainly,
\begin{align*}
&\bigg|\sum_{n_{1} ,\ldots, n_{k_0}} \wt{\mathcal{K}}_{\vec \l}(m,\cj{n}){\mathrm T}^{\mathrm b, m,S_{J}}_{nn_1\cdots n_{2 k+1}}  c(\bar n) \prod_{j=1}^{k_0} (\ft h_{\text{avg}}^{N_{j}, L_{j}})_{n_j}^{\iota_j} \frac{g_{n_j}^{\iota_j}}{\ld_{n_j}} \bigg|^{2} \\
& \les  M^{-2+2\ta} \sum_{n_1 \sim 2M} |  (\ft h_{\text{avg}}^{N_{1}, L_{1}})_{n_1}|^{2} 
\bigg| \sum_{n_{2} ,\ldots, n_{k_0}} \wt{\mathcal{K}}_{\vec \l}(m,\cj{n}){\mathrm T}^{\mathrm b, m,S_{J}}_{nn_1\cdots n_{2 k+1}}  c(\bar n) \prod_{j=2}^{k_0} (\ft h_{\text{avg}}^{N_{j}, L_{j}})_{n_j}^{\iota_j} \frac{g_{n_j}^{\iota_j}}{\ld_{n_j}} \bigg|^{2}
\end{align*}

We assume first that there are no pairings in $\{n_2,\ldots, n_{k_0}\}$ for the quantity
\begin{align}
\sum_{n_{2} ,\ldots, n_{k_0}} \wt{\mathcal{K}}_{\vec \l}(m,\cj{n}){\mathrm T}^{\mathrm b, m,S_{J}}_{nn_1\cdots n_{2 k+1}}  c(\bar n) \prod_{j=2}^{k_0} (\ft h_{\text{avg}}^{N_{j}, L_{j}})_{n_j}^{\iota_j} \frac{g_{n_j}^{\iota_j}}{\ld_{n_j}}. \label{n2nk0} 
\end{align}
 Then, we have the following.

\begin{lemma}\label{LEM:nopair}
Assume that there is no pairing amongst $\{2,\ldots, k_0\}$ and $L_{\max}\gg 1$. Then, $T^{-1}M$-certainly, it holds that
\begin{align}
|\eqref{n2nk0}|^{2} \les M^{\ta} \sum_{n_2,\ldots, n_{k_0}}  |\wt{\mathcal{K}}_{\vec \l}(m,\cj{n}) |^{2}  \wt{{\mathrm T}}^{\mathrm b, m, S_{J}}_{nn_1 n_2 \cdots n_{2 k+1}} |c(\cj{n})|^2  \prod_{j=2}^{k_0} \frac{ | (\ft h_{\textup{avg}})^{L_j,R_j}_{n_j}|^2  }{\ld_{n_j}^{2}}.
\end{align}
\end{lemma}

\begin{proof}
The exceptional sets removed here depend on both $m$ and $\vec \l$, 
but since they have complemental probability at most $e^{-(T^{-1}M)^{\ta}}$ and there are only $O(M^2)$ many possibilities for $m$ and $O(K_{M})$ many for $\vec \l$, we can assume that the following is true for all $(m, \vec \l)$.
We use an argument similar to that in \cite[Lemma 4.2]{DNY2}. We may assume that $L_{2} \geq L_3 \geq \cdots \geq L_{2k+1}$. Define 
\begin{align*}
a= \min\{ j\in \{2,\ldots, 2k+1\} \, : \, L_{j}>2^{10}L_{j+1}\}.
\end{align*}
It follows that $L_{2} \sim L_{a}$ and the implicit constant only depends upon $k_0$. 
We then have that 
$((\ft h_{\textup{avg}})^{L_j,R_j}_{n_j})^{\iota_j} $ for $2\leq j\leq k_0$ are $\mathcal{B}_{\leq L_{2}^{1-\kk}}$ measurable and $L_{2}^{1-\kk} \leq 2^{-10} L_{a}$ for $L_2 \gg 1$, 
and $g_{n_j}^{\iota_j}$ for $a+1\leq j\leq k_0$ are $\mathcal{B}_{\leq L_{a+1}}$ measurable with $L_{a+1}\leq 2^{-10}L_{a}$. By the no-pairing assumption, there is no pairing amongst $\{n_2,\ldots, n_{a}\}$ and amongst $\{n_{a+1},\ldots, n_{k_0}\}$. Thus, we write
\begin{align*}
\eqref{n2nk0} & = \sum_{n_2,\ldots, n_{a}} b_{n_2,\ldots, n_{a}}(\o) \prod_{j=2}^{a}\frac{g_{n_j}^{\iota_j}}{\ld_{n_j}}, 
\end{align*}
where
\begin{align*}
 b_{n_2,\ldots, n_{a}}(\o)  : =&\bigg( \prod_{j=2}^{a}  ((\ft h_{\text{avg}})^{L_j,R_j}_{n_j})^{\iota_j}  \bigg)  \sum_{n_{a+1},\ldots, n_{k_0}}\wt{\mathcal{K}}_{\vec \l}(m,\cj{n}) \wt{{\mathrm T}}^{\mathrm b, m,S_J}_{n_2 \cdots n_{2 k+1}}    \prod_{j=a+1}^{k_0} ((\ft h_{\text{avg}})^{L_j,R_j}_{n_j})^{\iota_j} \frac{g_{n_j}^{\iota_j}}{\ld_{n_j}}
\end{align*}
is $\mathcal{B}_{\leq 2^{-10}L_{a}}$ measurable. Therefore, by Lemma~\ref{LEM:condchaos}, we get $T^{-1}M$-certainly that
\begin{align*}
|\eqref{n2nk0}|^{2}  \les &M^{\ta} \sum_{n_2, \ldots, n_{a}} \prod_{j=2}^{a} \frac{ |(\ft h_{\text{avg}})^{L_j,R_j}_{n_j}|^2}{\ld_{n_j}^{2}}  \\
&\times  \bigg| \sum_{n_{a+1},\ldots, n_{k_0}}\wt{\mathcal{K}}_{\vec \l}(m,\cj{n}) \wt{{\mathrm T}}^{\mathrm b, m,S_{J}}_{nn_1 n_2 \cdots n_{2 k+1}} \prod_{j=a+1}^{k_0} ((\ft h_{\text{avg}})^{L_j,R_j}_{n_j})^{\iota_j} \frac{g_{n_j}^{\iota_j}}{\ld_{n_j}} \bigg|^2 .
\end{align*}
We now repeat the above argument inductively to establish Lemma~\ref{LEM:nopair}.
\end{proof}

Then, by Lemma~\ref{LEM:nopair}, we obtain that $(T^{-1}M)$-certainly, 
\begin{align*}
|\eqref{n2nk0} |^{2} 
& \les M^{-2+2k \ta}  \bigg( \prod_{j=2}^{k_0} N_{j}^{-2}\bigg)\sum_{n_2,\ldots, n_{k_0}}  |\wt{\mathcal{K}}_{\vec \l}(m,\cj{n})|^2 {\mathrm T}^{\mathrm b, m,S_{J}}_{nn_1\cdots n_{2 k+1}}  |c(\bar n)|^2 \prod_{j=2}^{k_0} |(\ft h_{\text{avg}}^{N_{j}, L_{j}})_{n_j} |^{2}.
\end{align*}
Note that here the exceptional sets depend on $(n,n_1)$ but there are only $O(M^2)$ many such choices which can be absorbed by $e^{-(T^{-1}M)^{\ta}}$.
Returning and using Cauchy-Schwarz (in $(\tau_1,\ldots, \tau_{k_0})$), \eqref{Kavgbd}, \eqref{redu1}, and \eqref{Xnt}, $T^{-1}M$-certainly it holds that 
\begin{align*}
\eqref{redu1} & \les  \sum_{\l_1, \cdots, \l_{2k+1}} \|\mathcal{X}_{n}^{J}\|_{\l^2_n}  \\
& \les M^{-1+k\ta} \bigg(
\prod_{ j= 2}^{2k+1}N_{j}^{-1+\g}
\bigg) \sum_{|m|\les M^2} \sum_{\l_1, \cdots, \l_{k_0}}  
\int_{Q_{\l_1} \times \cdots Q_{\l_{k_0}}} \\
& \hphantom{XXXXXX} \bigg\| \wt{\mathcal{K}}_{\vec \l}(m,\cj{n}) {\mathrm T}^{\mathrm b, m,S_{J}}_{nn_1\cdots n_{2 k+1}} c(\cj{n})\prod_{j=1}^{k_0} (\ft h_{\text{avg}}^{N_{j}, L_{j}})_{n_j} \bigg\|_{\l^{2}_{\cj{n}}} d\tau_{1\cdots k_0} \\
& \les  M^{-1+2k\ta} \bigg( \prod_{j=2}^{2k+1}N_{j}^{-1+\g}\bigg) \bigg( \prod_{j=1}^{k_0} \sup_{E_{N_j} \backslash E_{N_j/2}} \| \jb{\tau}^{b}\ft h_{n_j}^{N_{j}, L_{j}} \|_{L^2_{\tau} } \bigg) 
\Big\| {\mathrm T}^{\mathrm b, m,S_{J}}_{nn_1\cdots n_{2 k+1}} c(\cj{n}) \Big\|_{\l^{2}_{\cj{n}}} \\
& \les M^{-1+2k\ta+\eps} \bigg( \prod_{j=2}^{2k+1}N_{j}^{-1+\g}\bigg) \| c(\cj{n})\|_{\l^{\infty}_{\cj{n}}}  \| {\mathrm T}^{\mathrm b, m,S_{J}}_{nn_1\cdots n_{2 k+1}}\|_{\l^{2}_{nn_1 \ldots n_{2k+1}}},
\end{align*}
where the summation over $(\l_1, \cdots, \l_{2k+1})$ is handled analogously to the summation over $\l_1$ in Case 1 in Subsection \ref{SUB:case1}.
As $n\neq n_1$ on $S_{J}$ for $J\in \{3,6\}$, \eqref{Tbs1} holds true.
Thus, by Lemma~\ref{LEM:ev1}, and \eqref{Tbs1} we have
\begin{align*}
\eqref{redu1}  & \les  M^{-1+2k\ta+\eps}\bigg( \prod_{j=2}^{2k+1}N_{j}^{-1+\g}\bigg) N_{(2)}^{\frac 12} N_{(3)}^{\frac 12+\eps} \prod_{j=4}^{2k+1} N_{(j)}^{1-\eps} \\
& \les M^{-1+ 2 k\ta+\eps} N_{(2)}^{-\frac 12 + \g } N_{(3)}^{-\frac 12 + 10k\g+\eps} \\
& \les  (2M)^{-1+\frac{1}{4}\g},
\end{align*}
which is sufficient for our purpose in view of \eqref{redu1} and \eqref{Xnt}.

It remains to consider the pairing cases in \eqref{n2nk0},
\begin{align}
 \sum_{n_{2} ,\ldots, n_{k_0}} \wt{\mathcal{K}}_{\vec \l}(m,\cj{n}){\mathrm T}^{\mathrm b, m,S_{J}}_{nn_1\cdots n_{2 k+1}}  c(\bar n) \prod_{j=2}^{k_0} (\ft h_{\text{avg}}^{N_{j}, L_{j}})_{n_j}^{\iota_j} \frac{g_{n_j}^{\iota_j}}{\ld_{n_j}} .
 \label{zpair1}
\end{align}
To simplify the notation, we write $h_{n_j}^{(j)}$ in place of $(\ft h_{\text{avg}}^{N_{j}, L_{j}})_{n_j}$ and define
\begin{align*}
\mathcal{D}_{m}(\cj{n}) = \mathcal{D}_{m}(n,n_1,\ldots, n_{2k+1}) = \wt{\mathcal{K}}_{\vec \l}(m,\cj{n}){\mathrm T}^{\mathrm b, m,S_{J}}_{nn_1\cdots n_{2 k+1}}  c(\bar n).
\end{align*}
 We rewrite \eqref{zpair1} as a finite linear combination of terms such that for each of them it holds that there is a partition $\{Y_{i}\}$, for $1\leq i\leq p$ and some $p\geq 1$,  and $Z$ of $\{2,\ldots, k_0\}$ such that $n_{j}=n_{j'}$ for all $j,j'\in Y_{i}$ and $Z$ has no pairings. As $N_{j} \sim N_{j'}$ for $j,j'\in Y_{i}$ for some $1\leq i \leq p$, we define $N_{Y_i} = \max_{j\in Y_{i}} N_j$. 
 Also, note that 
 \begin{align}
|Y_{i}|\geq 2 \quad \text{for all} \quad 1\leq i\leq p. \label{Ysize}
\end{align}
 Thus, each of the terms in \eqref{zpair1} can be written as
\begin{align}
\sum_{n_2, \ldots, n_{k_0}} \mathcal{D}_{m}(\cj{n}) 
\bigg( \prod_{i=1}^{p }\prod_{j\in Y_i} (h_{n_j}^{(j)})^{\iota_j} \frac{g_{n_j}^{\iota_j}}{\ld_{n_j}} \bigg) \cdot  \bigg( \prod_{j\in Z} (h_{n_j}^{(j)})^{\iota_j} \frac{g_{n_j}^{\iota_j}}{\ld_{n_j}} \bigg).  
\label{zpair2}
\end{align}
 Then, $T^{-1}M$-certainly it holds that 
 \begin{align*}
|\eqref{zpair2}|  &\les  M^{\ta} \prod_{i=1}^{p} N_{Y_{i}}^{-|Y_{i}|} \bigg( \sum_{k_i \sim N_{Y_i}} \prod_{j\in Y_i}|h_{k_i}^{(j)}|^2 \bigg)^{\frac 12} \prod_{j\in Z} N_{j}^{-1} \bigg( \sum_{\substack{n_{2},\ldots, n_{k_0}\\ n_j = n_{j'} \textup{ for } j, j' \in Y_i}} |\mathcal{D}_{m}(\cj{n})|^2 \prod_{j\in Z} |h_{n_j}^{(j)}|^2 \bigg)^{\frac 12}
\end{align*}
We now use this bound in \eqref{zpair0} along with Cauchy-Schwarz (in the $(\tau_1, \ldots, \tau_{k_0})$ variables), \eqref{TypeCnorm}, \eqref{Kavgbd}, \eqref{Msump2}, Lemma~\ref{LEM:ev1},
\begin{align}
\eqref{redu1} & \les \sum_{\l_1, \cdots, \l_{2k+1}} \eqref{zpair0} \notag\\
& \les M^{-1+2\ta} \prod_{j=k_0+1}^{2k+1} N_{j}^{-1+\g} \prod_{i=1}^{p} N_{Y_i}^{-|Y_i|+\frac 12} \prod_{j\in Z} N_{j}^{-1} \cdot 
N_{(3)}^{\eps} \prod_{j=4}^{2k+1}N_{(j)}^{\frac 12 -} \notag \\
& \hphantom{XXXXX} \times \sup_{|m|\les M^2} \bigg( \sum_{n,n_1,\ldots, n_{2k+1}}  {\mathrm T}^{\mathrm b, m,S_{J}}_{nn_1\cdots n_{2 k+1}} \bigg)^{\frac 12} \notag \\
& \les M^{-1+\ta} \prod_{j=k_0+1}^{2k+1} N_{j}^{-1+\g} \prod_{i=1}^{p} N_{Y_i}^{-|Y_i|+1} \prod_{j\in Z} N_{j}^{-1} \cdot 
N_{(3)}^{\eps} \prod_{j=4}^{2k+1}N_{(j)}^{\frac 12 -}  \notag \\
& \hphantom{XXXXX} \times \sup_{|m|\les M^2} \sup_{ \substack{ n_j  \\  j\in \{2,\ldots, k_0\}\setminus Z}} \bigg( \sum_{n,n_1}\sum_{n_j \, :\, j\in Z} \sum_{n_{k_0 +1}, \ldots, n_{2k+1}}  {\mathrm T}^{\mathrm b, m,S_{J}}_{nn_1\cdots n_{2 k+1}} \bigg)^{\frac 12}.  \notag
\end{align}
Note that because of the definition of a pairing, any summation over $\{k_i \sim L_{Y_i}\}$ actually contains only a single sum.
Now since $n\neq n_1$ on $S_{J}$ for $J\in \{3,6\}$ so that \eqref{Tbs1}, we have 
\begin{align*}
\sup_{|m|\les M^2} \sup_{ \substack{ n_j  \\  j\in \{2,\ldots, k_0\}\setminus Z}} \bigg( \sum_{n,n_1}\sum_{n_j \, :\, j\in Z} \sum_{n_{k_0 +1}, \ldots, n_{2k+1}}  {\mathrm T}^{\mathrm b, m,S_{J}}_{nn_1\cdots n_{2 k+1}} \bigg)^{\frac 12} \les M^{\eps} \prod_{ j\in Z} N_{j}^{\frac 12} \prod_{j=k_0 +1}^{2k+1} N_j^{\frac 12}.
\end{align*}
Thus, 
\begin{align*}
\eqref{redu1} & \les M^{-1+\ta+\eps}  \prod_{j=k_0+1}^{2k+1} N_{j}^{-\frac 12+\g} \prod_{i=1}^{p} N_{Y_i}^{-|Y_i|+1}  \prod_{j\in Z}N_{j}^{-\frac 12} \cdot N_{(3)}^{\eps} \prod_{j=4}^{2k+1} N_{(j)}^{\frac 12 -}. 
\end{align*}
Crudely using \eqref{Ysize}, we have
\begin{align*}
\prod_{i=1}^{p} N_{Y_i}^{-|Y_i|+1} \sim \prod_{i=1}^{p} \prod_{j\in Y_{i}} N_{j}^{-1+\frac{1}{|Y_i|}}  \les  \prod_{i=1}^{p} \prod_{j\in Y_{i}} N_{j}^{-\frac 12} \sim \prod_{j\in \{2,\ldots, k_0\}\setminus Z} N_{j}^{-\frac 12},
\end{align*}
and thus
\begin{align*}
\eqref{redu1} &\les M^{-1+\ta+\eps}   N_{(3)}^{\eps} \prod_{j=k_0+1}^{2k+1} N_{j}^{-\frac 12+\g} \prod_{j=2}^{k_0} N_{j}^{-\frac 12} \prod_{j=4}^{2k+1} N_{(j)}^{\frac 12 - }  \\
&\les  M^{-1+\ta+\eps}  N_{(2)}^{-\frac 12+ \g } N_{(3)}^{-\frac 12 +\g+\eps} N_{(4)}^{10k \g} \\
& \les M^{-1+\ta+\eps}N_{(2)}^{-\frac 12+ \g } N_{(3)}^{-\frac 12 +11k \g+\eps} \\
& \les (2M)^{-1+\frac 14 \g}.
\end{align*}
This completes the proof of the pairing cases, which in turn closes Case 2 and thus, altogether, we have established \eqref{z2M}.

\begin{ackno}\rm
Y.W. would like to thank Prof.~ Carlos E. Kenig 
for his teaching and continuous support over the last decades.
The authors would like to thank Tadahiro Oh for his helpful inputs during the early stages of this work
and for his continued support. 
The second author would like to thank Nikolay Tzvetkov for explaining their work \cite{BCST2} and for his valuable comments, which led to Remark \ref{flow_property}.
J.F. was partially supported by the ARC project FT230100588 and by the European Research Council
(grant no. 864138 “SingStochDispDyn”).
Y.W. was supported by the EPSRC New Investigator
Award (grant no. EP/V003178/1) and by the EPSRC Mathematical Sciences Small Grant (grant no. UKRI1116).
\end{ackno}

\end{document}